\documentclass[11pt,reqno,a4paper]{amsart}
\usepackage{amsmath,amsthm,verbatim,amscd,amssymb,setspace,enumitem,hyperref}
\usepackage{exscale,color}

\usepackage{float}

\usepackage{enumitem}
\usepackage{mathrsfs}
\usepackage[normalem]{ulem}
\usepackage{soul}
\usepackage{cleveref}
\usepackage[T1]{fontenc}
\usepackage{listings}% http://ctan.org/pkg/listings
\usepackage[toc,page]{appendix}
\usepackage{adjustbox}
\usepackage{amsfonts}
\usepackage[english]{babel}
\usepackage{array,epsfig}
\usepackage{amsxtra}
\usepackage{latexsym}
\usepackage{dsfont}
\usepackage{mathrsfs}
\usepackage[mathscr]{eucal}
\usepackage{color}
\usepackage[all]{xy}
\usepackage{hyperref}
\usepackage{graphicx}
\usepackage{mathtools}
\usepackage{caption}
\usepackage{subcaption} 
\usepackage{relsize}
\usepackage{url}
\usepackage{dynkin-diagrams}
\usepackage{tikz}
\tikzstyle{color}=[circle,draw=black!50,fill=black!20,thick, inner sep=0pt,minimum size=1mm]

\usepackage{fancyvrb}
\usepackage{xcolor}

\renewcommand{\epsilon}{\varepsilon}
\newcommand{\N}{\mathbb{N}}

\newcommand{\R}{\mathbb{R}}
\newcommand{\C}{\mathbb{C}}

\renewcommand{\Re}{\operatorname{Re}}

\newcounter{mtheorem}
\newtheorem{mtheorem}[mtheorem]{Theorem}

\newtheorem{mexample}[mtheorem]{Example}

\newcommand{{\vol}}{\rm vol}
\newcommand{\p}{\partial}

\newcommand{\set}[1]{\left\{#1\right\}}
\newcommand{\tuple}[1]{\left(#1\right)}

\newcommand{\sprod}[1]{\left<#1\right>}

\newcommand{\ol}[1]{\overline{#1}}

\newcommand{\wt}[1]{\widetilde{#1}}

\newcommand{\tfrak}{\mathfrak{t}}

\newcommand{\Cbb}{\mathbb{C}}

\newcommand{\Nbb}{\mathbb{N}}

\newcommand{\Pbb}{\mathbb{P}}
\newcommand{\Qbb}{\mathbb{Q}}
\newcommand{\Rbb}{\mathbb{R}}

\newcommand{\Zbb}{\mathbb{Z}}

\newcommand{\Ccal}{\mathcal{C}}
\newcommand{\Dcal}{\mathcal{D}}

\newcommand{\Fcal}{\mathcal{F}}

\newcommand{\Hcal}{\mathcal{H}}
\newcommand{\Ical}{\mathcal{I}}

\newcommand{\Lcal}{\mathcal{L}}
\newcommand{\Mcal}{\mathcal{M}}
\newcommand{\Pcal}{\mathcal{P}}
\newcommand{\Ocal}{\mathcal{O}}
\newcommand{\Qcal}{\mathcal{Q}}

\newcommand{\Xcal}{\mathcal{X}}
\newcommand{\Ycal}{\mathcal{Y}}

\renewcommand{\phi}{\varphi}

\newcommand{\del}{\partial}
\newcommand{\delb}{\overline{\partial}}

\newcommand{\GL}{\operatorname{GL}}

\newcommand{\SL}{\operatorname{SL}}
\newcommand{\Spec}{\operatorname{Spec}}

\newtheoremstyle{fancy}{}{}{\itshape}{}{\textbf\bgroup}{.\egroup}{ }{}
\newtheoremstyle{fancy2}{}{}{\rm}{}{\textbf\bgroup}{.\egroup}{ }{}

\theoremstyle{fancy}
\newtheorem{theorem}{Theorem}[section]
\newtheorem{lemma}[theorem]{Lemma}
\newtheorem{corollary}[theorem]{Corollary}

\newtheorem{prop}[theorem]{Proposition}

\theoremstyle{fancy2}
\newtheorem{definition}[theorem]{Definition}
\newtheorem{example}[theorem]{Example}
\newtheorem{remark}[theorem]{Remark}

\setlist{leftmargin=*}

\numberwithin{equation}{section}

\makeatletter
\newcommand{\Spvek}[2][r]{%
  \gdef\@VORNE{1}
  \left(\hskip-\arraycolsep%
    \begin{array}{#1}\vekSp@lten{#2}\end{array}%
  \hskip-\arraycolsep\right)}

\def\vekSp@lten#1{\xvekSp@lten#1;vekL@stLine;}
\def\vekL@stLine{vekL@stLine}
\def\xvekSp@lten#1;{\def\temp{#1}%
  \ifx\temp\vekL@stLine
  \else
    \ifnum\@VORNE=1\gdef\@VORNE{0}
    \else\@arraycr\fi%
    #1%
    \expandafter\xvekSp@lten
  \fi}
\makeatother

\begin{document}

\title{Examples of complete Calabi--Yau metrics on affine smoothings of irregular toric Calabi--Yau cones}
\date{\today}

\author[R.~Conlon]{Ronan J.~Conlon}
\address[]{Department of Mathematical Sciences, The University of Texas at Dallas, Richardson, TX 75080, USA}
\email{ronan.conlon@utdallas.edu}

\author{Tran-Trung Nghiem}
\address{Institut Camille Jordan, Université Claude Bernard Lyon 1, CNRS, Lyon, France}
\email{nghiem@math.univ-lyon1.fr}

\maketitle
%\date{\today}

%\begin{center}
    %\emph{Dedicated to the memory of T.-T.N.'s paternal grandmother}
%\end{center}

\begin{abstract}
We present infinitely many new examples of affine Calabi--Yau manifolds of Euclidean volume growth and quadratic curvature decay, whose tangent cones at infinity are irregular and have smooth links.
In the process, we demonstrate (and provide the relevant computer code)
how to explicitly compute the Reeb field, the Minkowski summand cone,
and its maximal decompositions for a given toric Calabi--Yau cone with
smooth link from the data of its toric diagram. In complex dimension three, the
component corresponding to a maximal Minkowski decomposition into
lattice summands is a smoothing component if and only if every summand
is either a primitive lattice segment or a unimodular lattice
triangle. Furthermore, we propose an effective strategy to generate smoothable Calabi--Yau cones from a given non-smoothable one by taking Minkowski sums of certain toric diagrams, and provide an example to illustrate the method. 
\end{abstract}

\markboth{Ronan J.~Conlon and Tran-Trung Nghiem}{Calabi--Yau metrics on affine smoothings of irregular toric cones}

\section{Introduction}

\subsection{Overview}
Non-compact complete Calabi--Yau\footnote{By ``Calabi--Yau'', we mean ``Ricci-flat K\"ahler with trivial canonical sheaf''.} manifolds with Euclidean volume growth have a tangent cone at infinity \cite{cheeger1, colding1} which itself is a Calabi--Yau cone, possibly with a singular link. When the curvature of the Calabi--Yau metric decays at a quadratic rate, the link of the tangent cone at infinity is smooth and the tangent cone is unique \cite{cheegertian, CM}. Moreover, the Calabi--Yau metric converges to the Calabi--Yau cone metric at a polynomial rate \cite{sun1}. Such manifolds are called \emph{asymptotically conical} (AC) Calabi--Yau manifolds. In \cite{CH24}, these manifolds were classified by proving that every such example is obtained from its tangent cone via a suitable deformation and K\"ahler crepant resolution in a reversible and exhaustive process. 

\subsection{Main results}
In \cite{Conlon2}, the first example of an affine Calabi--Yau manifold of Euclidean volume growth with irregular tangent cone at infinity
was constructed. It was then asked in \cite[Section 1.4, Question 1]{CH24} 
whether other such examples exist. Our main theorems--Theorems \ref{mtheorem_cfo} and \ref{mtheorem_quadrilateral+segment}--provide an affirmative answer to this question.

Before we give a precise statement, some prerequisites are required. Recall that an $n$-dimensional 
toric K\"ahler cone $C_{0}$ is a K\"ahler cone endowed with an effective 
holomorphic action of the complex torus $T_{\mathbb{C}}=(\mathbb{C}^{*})^{n}$ 
exhibiting an open dense orbit
and whose fixed point set consists only of the apex. K\"ahler cones are affine algebraic with (at most) one singular point at the apex \cite[Section 3.1]{vC11}, and the action of the complex torus determines the action of a real torus $T^{n}\subset T_{\mathbb{C}}$ acting on the cone in 
a holomorphic and isometric fashion. Such a cone admits a Ricci-flat K\"ahler cone metric if and only if it has 
torsion canonical sheaf, namely the apex is a \( \Qbb\)-Gorenstein singularity; cf.~\cite[Theorem 1.2]{CFO08} and \cite{FOW}. For simplicity, we will restrict ourselves to the case where the canonical sheaf is actually a trivial line bundle, i.e., when the apex of the cone is a Gorenstein singularity.

Let $r$ denote the distance from the apex of the cone and $J_{0}$ its complex structure. We identify the link of the cone with the level set $\{r=1\}$. The vector field $\xi:=J_{0}r\partial_{r}$, known as the ``Reeb field'', is tangent to the (compact, smooth) link of the cone, hence its flow is complete. The flow of the Reeb field $\xi$ generates an abelian one-parameter subgroup of the isometry group of the link of the cone, the closure of which is a real torus $(S^{1})^k$ of rank $k$. The Calabi--Yau cone is called \emph{irregular} if $k > 1$ (in which case there exists a non-compact flow-line of $\xi$)
and \emph{quasi-regular} if $k=1$. In this case, every flow-line of $\xi$ is compact. The Calabi--Yau metric is unique up to pullback by a biholomorphism of the underlying complex cone
\cite[Theorem 1.4]{Esp25}.

Every toric Calabi--Yau cone with smooth link is determined by an
\((n-1)\)-dimensional convex lattice polytope whose proper faces are
unimodular simplices, called a ``toric diagram''; see Definition
\ref{def:toric_diagram}.
Two such cones are $T_{\C}$-equivariantly biholomorphic if and only if their toric diagrams are related by a translation and a $\GL_{n-1}(\Zbb)$-transformation. In particular, two such polytopes must have the same number of vertices.

In three dimensions, the simplest examples of toric Calabi--Yau cones are affine cones over the five smooth toric del Pezzo surfaces, embedded using their anti-canonical bundle. The corresponding toric diagrams are well-known to be reflexive polygons in $\Rbb^2$ (cf.~Example \ref{example_reflexive_cone} and the corresponding Figure \ref{figure_delpezzo_family}). From the results of Altmann \cite[Section (4.4)] {Alt94} (based on privately communicated computations of van Straten), Martelli--Sparks--Yau \cite[p.61, equation (3.48)]{MSY06}, and Futaki--Ono--Wang \cite[Corollary 1.3]{FOW}, the only such cone admitting an irregular Calabi--Yau metric \emph{and} an affine smoothing is the cone over \(\textnormal{Bl}_{p_1,\,p_2} (\Pbb^2)\) embedded using its anti-canonical bundle. Calabi--Yau metrics on the smoothing of this cone were found by the first author and Hein \cite{Conlon2}. 

The first examples of irregular toric Calabi--Yau cones were found by Gauntlett--Martelli--Sparks--Waldram \cite{GMSW}. They are defined by quadrilateral toric diagrams $\Ycal^{p,q}$ parametrized by a pair of coprime integers $(p,q)$, $p> q>0$, yielding irregular Calabi--Yau cones if and only if $4p^2 - 3q^2$ is not a square number. Topologically, they are cones over $S^{2}\times S^{3}$. However, in light of Altmann's toric deformation theory (cf.~Theorem \ref{theorem_deformation_classification}), these cones are not smoothable because the corresponding toric diagrams 
have no non-trivial lattice Minkowski decomposition
 (cf.~Proposition \ref{prop:gmsw_trivial}).
In \cite[p.457]{CFO08}, for $r, \,s \in \Nbb_{> 0}$, Cho--Futaki--Ono consider a family of three-dimensional toric Calabi--Yau cones defined by lattice polygons with $(2r + 3)$ vertices. The subfamily \(r=1\) has exactly \(s\) interior lattice points; cf.~Figure \ref{figure_cfo_family}.
The vertices are given by
\begin{equation*}
\begin{split}
(p_0, \,q_0)&= (0,0), \dots,(p_k, q_k) = \left(k,\frac{k(k+1)}{2}\right), \qquad 0 \leq k \leq r, \\
(p_{r+1+j},q_{r+1+j}) &= \left(r+1-j, \frac{(r+1)(r+2)}{2} + s-\frac{j(j+1)}{2}\right), \qquad 0 \leq j \leq r, \\
(p_{2r+2}, q_{2r+2}) &= (0,1).
\end{split}
\end{equation*}
They give rise to pairwise
non-$T_{\C}$-equivariantly biholomorphic
toric Calabi--Yau cones, due to the fact that the polygons have differing numbers of interior lattice points for a fixed number of vertices, and therefore different areas by Pick’s theorem. Our first main theorem is concerned with these specific Calabi--Yau cones.

\begin{mtheorem}[Examples of smoothable irregular Calabi--Yau cones]
\label{mtheorem_cfo}

Every toric Calabi--Yau cone from the Cho--Futaki--Ono family
\((r,s)\in\Nbb_{>0}^2\), realised as an affine variety, admits an
\(r\)-parameter affine smoothing obtained by base change of the
Altmann versal family to an irreducible component
\(
B^{r,s}\simeq\Cbb^r
\)
of the reduced versal base. Moreover, each cone appears as the
tangent cone at infinity of an AC Calabi--Yau metric on the corresponding smoothing.
Finally, if \(r=1\) and \(s\equiv 2\pmod{5}\), then the corresponding toric
Calabi--Yau cone is irregular.
\end{mtheorem}

In \cite[p.441]{CFO08}, the authors speculate that ``most'' of the family should be irregular. Due to computational constraints\footnote{
The computation in Mathematica involving large values of $r$ results in excessive run times.}, we restrict the volume calculation to the subfamily $r=1$. Within this subfamily, Proposition \ref{prop:irreg} proves irregularity whenever $s\equiv 2 \pmod{5}$, thereby yielding infinitely many irregular toric Calabi--Yau cones. In the process of proving this theorem, we give the explicit codes that can be used to compute the Minkowski summand
cone and its maximal decompositions
(cf.~Appendix \ref{s:A1}) and Reeb field (cf.~Appendix \ref{s:A2}) of a toric Calabi--Yau cone.
 
It follows from \cite{CdlO, Got12, vC10, vC11} that every three-dimensional toric Calabi--Yau cone admits a (toric) crepant resolution that carries AC Calabi--Yau metrics. Theorem \ref{mtheorem_cfo} therefore yields new examples of
\emph{geometric transitions}, whereby one passes from a crepant
resolution to a smoothing via the cone, first contracting the
exceptional set to a point and then deforming.
We refer the reader to \cite{Rossi}
for more information on geometric transitions in the study of Calabi--Yau manifolds and their relevance to physics. Since these smoothings admit a $(\mathbb{C}^{*})^{n-1}$-action \cite[Section (5.3)]{Alt97}
and the resolutions are toric, we see from \cite{Gro01} that both
admit special Lagrangian fibrations. 
Theorem \ref{mtheorem_cfo} should therefore be of interest in the Strominger--Yau--Zaslow conjecture and mirror symmetry.

With hindsight, one possible way to generate examples answering \cite[Section 1.4, Question 1]{CH24} is the following.
\begin{enumerate}
    \item Take a lattice polygon $\Pcal$ in $\Rbb^2$ that defines a toric diagram (cf.~Definition \ref{def:toric_diagram}) with a non-trivial maximal Minkowski decomposition into lattice summands. We only consider toric diagrams of toric Kähler cones (i.e., cones smooth away from the apex) to ensure that the relevant deformation theory is finite-dimensional (see Remark \ref{rmk:toric_diagram} and Altmann's rigidity theorem as recalled in Theorem \ref{thm:rigidity}). A toric cone with infinite-dimensional Kodaira--Spencer space prevents a direct application of \cite[Theorem 4.3]{CH24}. 
    
The polygon \(\Pcal\) can be built in an ad hoc manner by taking the
Minkowski sum of a lattice segment \(\Lcal\) in \(\Rbb^2\) with no
interior lattice point and a toric diagram \(\Qcal\) in \(\Rbb^2\)
with no edge parallel to \(\Lcal\) and no non-trivial lattice
Minkowski decomposition. This always generates a toric diagram, hence
a Calabi--Yau cone; see Lemma \ref{lem:sumtoricdiag}. The  decomposition
\(
\Pcal=\Qcal+\Lcal
\)
is already a maximal Minkowski decomposition into lattice summands and
therefore determines an irreducible component of the reduced versal
base. When \(\Qcal\) is a quadrilateral, this component is not a
smoothing component. A subtle point is whether \(\Pcal\) admits an
additional maximal Minkowski decomposition into lattice segments and
lattice triangles\footnote{A ``lattice segment'' or ``lattice triangle''
means a segment or triangle whose vertices are lattice points, respectively.
A lattice segment is called \emph{primitive} if it contains no
lattice points other than its endpoints. A lattice triangle is
called \emph{unimodular} if it is affine lattice-equivalent to
$\operatorname{Conv}\{(0,0),(1,0),(0,1)\}$, or equivalently if its
normalized lattice area is one.}. In many situations such a decomposition can be
guaranteed by a careful choice of \(\Lcal\); see Lemma
\ref{cor:sumtoricdiag} and Example
\ref{ex:latticemaxvsmaxlattice}.

    \item By Altmann's classification \cite{Alt97}, as stated in Theorem
\ref{theorem_deformation_classification}, a maximal Minkowski
decomposition of \(\Pcal\) into \(m+1\) lattice summands determines an
irreducible component
\(
B\simeq\Cbb^m
\)
of the reduced versal base. If every summand is either a primitive
lattice segment or a unimodular lattice triangle, then \(B\) is a
smoothing component by Lemma \ref{lemma_smoothing_component}; that is,
the associated family is smooth over a nonempty Zariski-open subset of
\(B\). An application of \cite[Theorem 4.3]{CH24} then yields AC
Calabi--Yau metrics in every K\"ahler class on each smooth fiber.
    
    \item Under the normalization condition in Lemma
\ref{lem:reebnormalization}, the Reeb field associated with a toric
Calabi--Yau metric is irregular if and only if at least one of the
coordinates \((\xi_1,\ldots,\xi_{n-1})\) is irrational. For the
Cho--Futaki--Ono subfamily considered in Section \ref{s:minn}, this is
proved in Proposition \ref{prop:irreg} whenever
\(s\equiv2\pmod5\). The code in Appendix \ref{s:A2} may also be used
to compute the minimizer in individual examples.
\end{enumerate}

This approach is only relevant for polygons in $\Rbb^2$ (i.e., three-dimensional toric cones) and breaks down for toric diagrams with $\dim_{\Rbb} \Pcal \geq 3$ (i.e., $\dim_{\Cbb} C_0 \geq 4$) due to Altmann's rigidity theorem, namely Theorem \ref{thm:rigidity}. Our next result illustrates the above strategy. As previously mentioned, the reduced versal base of every member of
the Gauntlett--Martelli--Sparks--Waldram family is a point. This is
reflected by the fact that \(\Ycal^{p,q}\) admits no non-trivial
lattice Minkowski decomposition. To turn the situation around, we take the Minkowski sum of $\Ycal^{p,q}$ with a lattice segment in the following way.
\begin{mtheorem}[Generating deformations and smoothings of Calabi--Yau cones] \label{mtheorem_quadrilateral+segment} Let \(\Qcal^{p,q}\) be the polygon obtained as the Minkowski sum \[ \Qcal^{p,q}:=\Ycal^{p,q}+\Lcal, \] where \(\Ycal^{p,q}\) is a Gauntlett--Martelli--Sparks--Waldram toric diagram and \(\Lcal\) is a lattice segment in \(\Rbb^2\). Then: \begin{enumerate}[label=\textnormal{(\alph{*})}, ref=(\alph{*})] \item \label{mtheorem_diagram} \(\Qcal^{p,q}\) is a toric diagram associated with a toric Calabi--Yau cone \(C_0^{p,q}\) with an isolated singularity if and only if \(\Lcal\) has no interior lattice points and is not parallel to any edge of \(\Ycal^{p,q}\).

\item\label{mtheorem_smoothable}
If \(\Qcal^{p,q}\) is a toric diagram, then the
decomposition
\[
\Qcal^{p,q}=\Ycal^{p,q}+\Lcal
\]
is a maximal Minkowski decomposition into lattice summands and
determines a one-dimensional irreducible component of the reduced
versal base. The restriction of the versal family to this component
has affine toric total space, but the component is not a smoothing component.

In addition, \(\Qcal^{p,q}\) admits an additional maximal Minkowski decomposition
into two lattice triangles if and only if either:
\begin{itemize}
\item
$q=1$ and $\Lcal$ is a translate of 
$\textnormal{Conv}
\set{\Spvek{0;0},\Spvek{1;1}}$,
or
\item
$p-q$ is odd and 
$\Lcal$ is a translate of 
$\textnormal{Conv}
\set{\Spvek{0;0},\Spvek{-p+q+2;-p+q}}$.
\end{itemize}
In either case, this additional decomposition determines another
one-dimensional irreducible component of the reduced versal base. In
the first case, both triangle summands are unimodular, so this
additional component is a smoothing component. Consequently, each
smooth fiber admits an AC Calabi--Yau metric with asymptotic cone \(C_0^{p,1}\) in every K\"ahler class. In the second case, the additional
component is not a smoothing component and its general fiber is
singular.
 \end{enumerate}
 \end{mtheorem}

As an application, we have:
\begin{mexample} \label{ex:gmsw21+segment}
For both choices of $\Lcal$, the above construction with $\Ycal^{2,1}$ yields pairwise non-$T_{\C}$-equivariantly biholomorphic toric Gorenstein K\"ahler cones with irregular toric Calabi--Yau metrics. 
\end{mexample}
\noindent Only the first of these two choices corresponds to a smoothing
component; the second illustrates the non-smoothing branch of
Theorem \ref{mtheorem_quadrilateral+segment}.

Of course, one can build a lattice polygon with a maximal Minkowski
decomposition into lattice summands from any diagram
\(\Ycal^{p,q}\), as indicated by the proof of Lemma
\ref{lem:gmsw-extrarigid}. However, the associated cone may have
non-isolated singularities and consequently possibly
infinite-dimensional deformation theory. Even though Altmann's theory
still applies to some extent in this context \cite{Alt00}, a general
construction of Calabi--Yau metrics on affine smoothings of algebraic
cones with non-isolated singularities remains open, even in the toric
setting, despite recent progress in
\cite{BD19,CCR26,Ngh24a} and the references therein.

For obvious reasons, the code in
Appendix \ref{s:A2} (and Figure \ref{figure_mathematica_cfo13}) can only determine whether or not a Calabi--Yau cone is irregular on a case-by-case basis. We are not aware of any general and non-computational method to determine whether or not a toric Calabi--Yau cone is irregular. Such a method should involve non-trivial arithmetic properties of the cones, formulated in terms of suitable properties of the corresponding toric diagrams. On the other hand, 
Martelli--Sparks--Yau conjectured that the degree of the normalized volume as an algebraic number is equal to $(n-1)^{k-1}$ \cite[equation (1.19)]{MSY08}, where $k$ is the dimension of the compact torus generated by the Calabi--Yau Reeb field. Even assuming this conjecture, it is not clear to us how to determine the irregularity of the cone metric.

\subsection{Outline of paper}
In what follows, all computer code has been relegated to Appendix \ref{appendix_hilbert_volume}, where an introduction to the Macaulay2Web Interface is given in Appendix \ref{s:started}.

We begin in Section \ref{toric-section} 
by defining and discussing toric Calabi--Yau cones. Specifically, in Section \ref{coness} we define what we mean by a Calabi--Yau cone and an AC Calabi--Yau manifold in Definitions \ref{d:cycone} and \ref{ACCY}, respectively (in particular, both have trivial canonical bundle), and in Section \ref{toric-geom} we present the relevant background material on toric geometry that we require. Important material introduced includes the definition of a toric Gorenstein cone in Definition \ref{gorenstein}, and how this property is reflected in the corresponding toric diagram in Theorem \ref{goren-equiv} and Proposition \ref{proposition_combinatorial_qgorenstein}.

In Section \ref{reeb-identify}, we describe how to determine the Reeb field of a Calabi--Yau cone metric on a toric Gorenstein K\"ahler cone.
This vector field minimizes the volume function which is introduced, along with the index character, in Section \ref{subsection_volume_minimizer}. These are defined in Definitions \ref{d:volume} and \ref{d:index} respectively. A computer algorithm for how to compute this minimizer using the index character is then given in Section \ref{s:minimizer}. Finally, we provide examples of how to implement the relevant computer code for these computations in Section \ref{exampless}.

The subject of Section \ref{section_deformation} is the deformation theory of toric Calabi--Yau cones. The relevant facts on deformation theory that we need are collected together in Section \ref{prelim}. Minkowski decompositions and Minkowski schemes, along with some examples, are introduced in Section \ref{s:versal}. Necessary and sufficient conditions for such a decomposition to be maximal are mentioned in Corollary \ref{cor:minkalways}, Lemma \ref{lem:practical-crit} identifies which lattice quadrilaterals admit Minkowski decompositions, and Lemma \ref{lem:sumtoricdiag} notes when the sum of a toric diagram and a lattice segment is again a toric diagram. Examples illustrating the relevant concepts are given throughout this section. The correspondence between deformations of a toric Gorenstein cone and maximal Minkowski decompositions into lattice summands given by Altmann's theory is explained in 
Section \ref{s:4.3}, where the Minkowski scheme is identified with the versal deformation base; see Theorems \ref{thm:versalfamily} and \ref{theorem_deformation_classification}.
We then reinforce these ideas in Example \ref{examplesss}. 
A computer algorithm, based on Altmann's theory, for computing the
components of the reduced versal base of a given toric Gorenstein cone
is the content of Section \ref{subsection_minkowski}.
This we demonstrate with an example in Section
\ref{egsss}.

Section \ref{s:thmb} comprises the proof of Theorem \ref{mtheorem_cfo}. In particular, Theorem \ref{theorem_minkowski_decomposition} in Section \ref{s:mink}  constructs a
distinguished maximal Minkowski decomposition into lattice summands
for the Cho--Futaki--Ono family and identifies the corresponding
component \(B^{r,s}\simeq\Cbb^r\), and Proposition \ref{prop_volume} in Section \ref{s:minn} provides a formula for the volume of the pertinent subfamily of these cones. When $s\equiv 2 \pmod{5}$, the minimizer is irregular, as stated in Proposition \ref{prop:irreg}. We conclude Section \ref{s:thmb} by explicitly computing in Section \ref{egs} the Reeb field of the Calabi--Yau cone metric for some specific cones in this family.

In Section \ref{s:thma}, we prove Theorem \ref{mtheorem_quadrilateral+segment} in Section \ref{s:6.1} and provide the justification for Example \ref{ex:gmsw21+segment} in Section \ref{s:6.2}. 

\subsection{Summary of notation}
We provide a summary of the notation regarding the various toric Calabi--Yau cones and toric diagrams that we consider throughout.
\begin{itemize}
    \item $\Qcal_1,\Qcal_2, \Qcal_3, \Qcal_4,\Qcal_5$ denote 
    the toric diagrams of the del Pezzo cones over $\Pbb^1 \times \Pbb^1$, $\Pbb^2$, $\text{Bl}_{p_1}(\Pbb^2)$, $\text{Bl}_{p_1,\,p_2}(\Pbb^2)$, $\text{Bl}_{p_1,\,p_2,\,p_3}(\Pbb^2)$ embedded using their anti-canonical bundles, respectively.
    \item $\Pcal^{r,s}$ denotes the toric diagram of a member of the Cho--Futaki--Ono family of toric Calabi--Yau cones. Moreover, we denote by $\Pcal^{s}$ the toric diagram $\Pcal^{1,s}$. 
    \item $\Ycal^{p,q}$ is the toric diagram of the Gauntlett--Martelli--Sparks--Waldram cone $Y^{p,q}$. 
    \item $\Qcal^{p,q}$ is the sum of $\Ycal^{p,q}$ and a lattice segment such that $\Qcal^{p,q}$ is a toric diagram. 
\end{itemize}

\subsection{Acknowledgements}
We thank Tristan Collins for showing us how to explicitly compute the Reeb field of a toric Calabi--Yau cone; Klaus Altmann, Andrei Balakin, Charles Cifarelli, Hans-Joachim Hein, Yuji Odaka, Helge Ruddat, Bernd Sturmfels, Hendrik S\"{u}{\ss}, Song Sun, and Junsheng Zhang for valuable discussions; Eveline Legendre for helpful exchanges regarding \cite{MSY08}; and Eveline and Junsheng for their insightful comments on a preliminary version of this article. In addition, we thank Hans-Joachim Hein and the two anonymous referees for their detailed comments, in particular for pointing out an oversight in the proof of Proposition \ref{prop:irreg} in the first version of this article.

This material is based upon work supported by the National Science Foundation under Grant
No. DMS-1928930 while the first author was in residence at the Simons Laufer Mathematical Sciences Institute (formerly MSRI) in Berkeley, California, during the Fall 2024 semester. He wishes to thank
SLMath for their excellent working conditions and hospitality during this time.

The first author is currently supported by a Simons Travel Grant. The second author is supported by the ANR--FAPESP-21-CE40-0017 project BRIDGES and partially supported by the Crafoord Foundation grant CR2024-0029. He is grateful to the Institutionen för matematik och matematisk statistik at Ume\r{a} universitet and Jakob Hultgren for their hospitality during the drafting phase of this work.

\section{Toric Calabi--Yau cones}\label{toric-section}

The main objective of this section is to introduce what we mean by a
``toric Calabi--Yau cone'' and relate its algebraic properties to its
moment cone. In passing, we mention some other important definitions. We begin by introducing the model geometries that we are concerned with in Section \ref{coness}. A background in toric geometry is then given in Section \ref{toric-geom}, where the notion of a  ``Gorenstein'' K\"ahler cone is introduced, a
``toric Calabi--Yau cone'' is defined, and the relationship between the two is given.

\subsection{Asymptotically conical Calabi--Yau manifolds and Calabi--Yau cones}\label{coness}

\subsubsection{Riemannian cones} For us, {the definition of a Riemannian cone will take the following form}.

\begin{definition}\label{cone}
Let $( L, g)$
be a compact connected Riemannian manifold. The \emph{Riemannian cone} $C$ with link $L$ is defined to be $\R^+ \times L$ with metric $g_0 = dr^2 \oplus r^2g$ up to isometry. The radius function $r$ is then characterized intrinsically as the distance from the apex in the metric completion. 
\end{definition}

Throughout, we denote by $C_0$ the metric completion of $C$, realized by adding the apex to $C$. This is only to make statements precise, as it will be essentially irrelevant whether or not we are referring to the manifold $C$ or to the singular affine cone $C_0$.

\subsubsection{K{\"a}hler and Calabi--Yau cones}\label{conesss}
Boyer--Galicki \cite{BG08} is a comprehensive reference here.

\begin{definition}\label{kahlercone} A \emph{K{\"a}hler cone} is a Riemannian cone $(C,g_0)$ such that $g_0$ is K{\"a}hler, together with a choice of $g_0$-parallel complex structure $J_0$. This will in fact often be unique up to sign. We then have a K{\"a}hler form $\omega_0(X,Y) = g_0(J_0X,Y)$, and $\omega_0 = \frac{i}{2}\p\bar{\p} r^2$ with respect to $J_0$.
\end{definition}

The vector field $r\partial_{r}$ on a K\"ahler cone is real holomorphic and $J_{0}r\partial_r$ is real holomorphic and Killing \cite[Appendix A]{MSY08}. This latter vector field is known as the \emph{Reeb field}. The closure of its flow in the isometry group of the link of the cone generates the holomorphic isometric action of a real torus on $C$ that fixes the apex of the cone. We call a K{\"a}hler cone ``quasiregular'' if this action is an
$S^1$-action (and, in particular, ``regular'' if this $S^1$-action is free), and ``irregular'' if the action generated is that of a real torus of rank $>1$.

Given a K\"ahler cone $(C,\,\omega_{0}=\frac{i}{2}\partial\bar{\partial}r^{2})$ with radius function $r$, it is true that
\begin{equation*}
\omega_{0}=rdr\wedge\eta +\frac{1}{2}r^{2}d\eta
\end{equation*}
where $\eta=i(\bar{\partial}-\partial)\log r$ is a contact one-form on the link of the cone.  

After adding the apex, a K\"ahler cone becomes a normal affine algebraic variety.

\begin{theorem}\label{t:affine}
For every K\"ahler cone \((C,g_0,J_0)\), there is a normal affine
algebraic variety \(V\subset\Cbb^N\), smooth away from the origin,
such that $(C,J_0)\simeq V\setminus\set{0}$.
In addition, \(V\) can be taken to be invariant under a
\(\Cbb^*\)-action
\(
(t,z_1,\ldots,z_N)
\longmapsto
(t^{w_1}z_1,\ldots,t^{w_N}z_N),
\)
where \(w_i\in\Zbb_{>0}\) for all \(i\).
\end{theorem}
\noindent This can be deduced from arguments written down by van Coevering in \cite[\S 3.1]{vC11}.

The particular type of K\"ahler cone that concerns us is the following.
\begin{definition}\label{d:cycone}
We say that $(C,g_0,J_0,\Omega_0)$ is a \emph{Calabi--Yau cone} if
\begin{enumerate}[label=\textnormal{(\roman{*})}, ref=(\roman{*})]
\item $(C,g_0, J_0)$ is a Ricci-flat K\"ahler cone of complex dimension $n$,
\item the canonical bundle $K_{C}$ of $C$ with respect to $J_0$ is trivial, and
\item $\Omega_{0}$ is a nowhere vanishing section of $K_{C}$ with $\omega_0^n = i^{n^2}\Omega_0 \wedge \overline{\Omega}_0$.
\end{enumerate}
\end{definition}

These serve as the asymptotic models for our manifolds of interest.

\begin{definition}\label{ACCY}
Let $(C,g_0,J_0,\Omega_0)$ be a Calabi--Yau cone as above. Let $(M,g,J,\Omega)$ be a Ricci-flat K\"ahler manifold with K\"ahler
form \(\omega\) and a parallel holomorphic volume form such that $\omega^n = i^{n^2}\Omega \wedge \overline\Omega$. We call $M$ an \emph{asymptotically conical (AC) Calabi--Yau manifold with asymptotic cone $C$} if there exist a compact subset $K\subset M$ and a diffeomorphism $\Phi: \{r > 1\} \to M\setminus K$ such that for some $\lambda_1, \lambda_2 < 0$ and all $j \in \N_0$,
\begin{align*}
|\nabla_{g_0}^j(\Phi^{*}g-g_{0})|_{g_{0}} =O(r^{\lambda_1-j}),\\
|\nabla_{g_0}^j(\Phi^{*}J-J_{0})|_{g_{0}} =O(r^{\lambda_2-j}),\\
|\nabla_{g_0}^j(\Phi^{*}\Omega-\Omega_{0})|_{g_{0}} = O(r^{\lambda_2-j}).
\end{align*}
We abbreviate the words ``asymptotically conical'' by AC.
\end{definition}

Implicit in this definition is the fact that
AC Calabi--Yau manifolds have one end. This follows from the Cheeger--Gromoll splitting theorem. We refer the reader to \cite[Section 1.2.2]{CH24} for more remarks on these manifolds. 

\subsection{Toric geometry}\label{toric-geom}
In this subsection, we collect together some standard facts from toric geometry.

\subsubsection{Toric manifolds}
We begin with the following definition.

\begin{definition}\label{toricmanifold}
A \emph{toric manifold} is an $n$-dimensional complex manifold $M$ endowed
with an effective holomorphic action of the algebraic torus $T_{\mathbb{C}}=(\mathbb{C}^{*})^{n}$ such that the following hold true.
\begin{itemize}
  \item The fixed point set of the $T_{\mathbb{C}}$-action is compact.
  \item There exists a point $p\in M$ with the property that the orbit $T_{\mathbb{C}}\cdot p \subset M$ forms a dense open subset of $M$.
\end{itemize}
\end{definition}
We will often denote the dense orbit simply by $(\mathbb{C}^{*})^{n} \subset M$ in what follows. The $T_{\mathbb{C}}$-action of course determines the action of the real torus $T^n \subset T_{\mathbb{C}}$.

\subsubsection{K\"ahler metrics on toric manifolds}\label{finito}

For a given toric manifold $M$ endowed with a Riemannian metric $g$ invariant under the action of the real torus $T^n \subset T_{\mathbb{C}}$ and K\"ahler with respect to the underlying complex structure
of $M$, the K\"ahler form $\omega$ of $g$ is also invariant under the $T^n$-action. We call such a manifold a \emph{toric K\"ahler manifold}.
In what follows, we always work with a fixed complex structure on $M$.
We then have:
\begin{definition}
Let $(M,\,\omega)$ be a symplectic manifold and let $T$ be a real torus acting by symplectomorphisms on $(M,\,\omega)$.
Denote by $\mathfrak{t}$ the Lie algebra of $T$ and by $\mathfrak{t}^{*}$ its dual. Then we say that the action of $T$ is \emph{Hamiltonian}
if there exists a smooth map $\mu:M\to\mathfrak{t}^{*}$ such that for all $\zeta\in\mathfrak{t}$,
$-\omega\lrcorner\zeta=du_{\zeta}$, where $u_{\zeta}(x)=\langle\mu(x),\,\zeta\rangle$ for all $\zeta\in\mathfrak{t}$ and $x\in M$. We call
$\mu$ the \emph{moment map} of the $T$-action and we call $u_{\zeta}$ the \emph{Hamiltonian (potential)} of $\zeta$.
\end{definition}
Notice that $u_{\zeta}$ is invariant under the flow of $\zeta$. Indeed, we have that
$$\mathcal{L}_{\zeta}u_{\zeta}=du_{\zeta}\lrcorner\zeta=-\omega(\zeta,\,\zeta)=0.$$
Consequently, each integral curve of $\zeta$ must be contained in a level set of $u_{\zeta}$.

Hamiltonian K\"ahler metrics have a useful characterisation due to Guillemin.

\begin{prop}[{\cite[Theorem 4.1]{Guil}}]\label{propB6}
	Let $\omega$ be any $T^n$-invariant K\"ahler form on $M$. Then the $T^{n}$-action is Hamiltonian with respect to $\omega$ if and only if the restriction of $\omega$ to the dense orbit $(\mathbb{C}^{*})^{n} \subset M$ is exact, i.e., there exists a $T^{n}$-invariant potential $\phi$ such that
	\begin{equation*}
		\omega = 2i\partial\bar{\partial} \phi.
	\end{equation*}
\end{prop}

Let
\(
\mathfrak{N}:=\ker(\exp:\mathfrak t \to T^n)
\)
denote the \emph{coweight lattice} of \(T^n\). Fix once and for all a
\(\Zbb\)-basis \((X_1,\ldots,X_n)\) of \(\mathfrak{N} \). This induces a
coordinate system \(\xi=(\xi_1,\ldots,\xi_n)\) on \(\mathfrak t\). Using the natural inner product on $\mathfrak{t}$ to identify $\mathfrak{t} \cong \mathfrak{t}^*$, we can also identify $\mathfrak{t}^* \cong \mathbb{R}^n$.
We denote the induced coordinates on $\mathfrak{t}^*$ by $y=(y_1,\ldots,y_n)$.
Let $(z_1, \dots, z_n)$ be the natural
coordinates on $(\mathbb{C}^{*})^{n}$ as an open subset of $\mathbb{C}^n$. There is a natural diffeomorphism $\textnormal{Log} \colon
(\mathbb{C}^{*})^{n} \to \mathfrak{t}\times T^n$ which provides a one-to-one correspondence between $T^n$-invariant smooth functions on
 $(\mathbb{C}^{*})^{n}$ and smooth functions on $\mathfrak{t}$. Explicitly,
\begin{equation}\label{diffeoo}
(z_1, \dots, z_n)\xmapsto{\operatorname{Log}}(\log(r_1), \dots, \log(r_n), \theta_1, \dots, \theta_n)=(\xi_{1},\ldots,\xi_{n},\,\theta_{1},\ldots,\theta_{n}),
\end{equation}
where $z_j = r_j e^{i \theta_j}$,\,$r_{j}>0$. Given a function $H(\xi)$ on $\mathfrak{t}$, we can extend $H$ trivially to $\mathfrak{t}\times T^n$ and pull back by Log to
obtain a $T^n$-invariant function on $(\mathbb{C}^{*})^{n}$. Clearly, any $T^n$-invariant function on $(\mathbb{C}^{*})^{n}$ can be written in this form.

Choose any branch of $\log$ and write $w = \log(z)$. Then clearly $w = \xi + i \theta$, where $\xi=(\xi_1,\ldots,\xi_n)$ are real coordinates on $\mathfrak{t}$
(or, more precisely, there is a corresponding lift of $\theta$ to the universal cover with respect to which this equality holds),
and so if $\phi$ is $T^n$-invariant and $\omega = 2i \p \bar{\partial} \phi$, then we have that
\begin{equation}\label{e:T5}
	\omega = 2i\frac{\partial^2 \phi}{ \p w_{i} \p\bar{w}_j} dw_i \wedge d\bar{w}_j = \frac{\partial^2 \phi}{ \p \xi_{i} \p\xi_{j}} d\xi_{i} \wedge d\theta_{j}.
\end{equation}
In this setting, the metric $g$ corresponding to $\omega$ is given on $\tfrak \times T^n$ by
\begin{equation*}
	g = \phi_{ij}(\xi)d\xi_{i} d\xi_{j} + \phi_{ij}(\xi)d\theta_{i} d\theta_{j},
\end{equation*}
and the moment map $\mu$ as a map $\mu: \tfrak \times T^n \to \tfrak^*$ is defined by the relation
		\begin{equation*}
		\langle \mu(\xi, \theta), b \rangle = \langle \nabla \phi(\xi), b \rangle
	\end{equation*}
for all $b \in \mathfrak{t}$, where $\nabla \phi$ is the Euclidean gradient of $\phi$.
The $T^n$-invariance of $\phi$ implies that it depends only on $\xi$ when considered a function on $\tfrak \times T^n$
via \eqref{diffeoo}. Since $\omega$ is K\"ahler, we see from \eqref{e:T5} that the Hessian of $\phi$ is positive-definite so that $\phi$ itself is strictly convex.
In particular, $\nabla \phi$ is a diffeomorphism onto its image.
Using the identifications mentioned above, we view $\nabla \phi$ as a map from $\mathfrak{t}$ into an open subset of $\tfrak^*$.

\subsubsection{Toric K\"ahler cones} \label{subsubsection:torickahlercones}
The following is sourced from \cite{MS06, MSY06, MSY08, P06}.

Recall Definition \ref{kahlercone} of a K\"ahler cone with metric completion $(C_{0},\,g_{0})$. Let $n:=\dim_{\mathbb{C}}C_{0}$.
As previously noted in Section \ref{conesss}, $C_0$ is a normal affine algebraic variety and the flow of the Reeb field $\xi$ generates an abelian one-parameter subgroup of the isometry group of the link of the cone, the closure of which is a real torus. If the flow of \(\xi\) lies in a real torus \(T^{n}\) of maximal
possible rank (i.e., \(n\)) in the isometry group of the link of the
cone, then \(T^{n}\) induces an effective holomorphic and isometric
action on \(C_{0}\), preserving its slices and fixing its apex. This
action complexifies to an effective holomorphic action of $T_{\mathbb{C}}=(\mathbb{C}^{*})^{n}$
on \(C_{0}\), with an open dense orbit and the apex as its unique fixed point. In this case, we say that $(C,\,g_{0},\,T_{\mathbb{C}})$ is a \emph{toric K\"ahler cone} and the corresponding K\"ahler form $\omega_{0}$ is invariant under the action of $T^{n}$.

Given the above situation,
write $\eta$ for the induced contact form on the link of the cone $C_{0}$ and 
let $\mathfrak{t}$ denote the Lie algebra of $T^{n}$.
Recall the \(\Zbb\)-basis \(X_1,\ldots,X_n\) of \(\mathfrak{N}\) fixed
above. As previously explained, this allows us to identify $\mathfrak{t}$
with $\mathbb{R}^{n}$ with coordinates $(\xi_{1},\ldots,\xi_{n})$, and in turn, we obtain induced coordinates $(y_{1},\ldots,y_{n})$ on $\mathfrak{t}^{*}$.
We denote the \emph{weight lattice} of $T^{n}$ by $\mathfrak{M}\subset\mathfrak{t}^{*}$ which is the dual lattice of $\mathfrak{N}$.
Recall that $\mathfrak{N}:=\ker(\exp:\mathfrak{t}\to T^{n})$ and so $\mathfrak{M} =\operatorname{Hom}_{\mathbb{Z}}(\mathfrak{N},\,\mathbb{Z})$. Thanks to the
invariance of $\xi$ and $\omega_{0}$ under the action of $T^{n}$, the action of $T^{n}$ on $(C_{0},\,g_{0})$ is automatically Hamiltonian, and so we have a moment map $\mu:C_{0}\to\mathfrak{t}^{*}\cong\mathbb{R}^{n}$, the components of which are given by 
\[y_{i}:=\frac{r^{2}}{2}\eta(X_{i}).\]
Since $\xi\in\mathfrak{t}$ (meaning that $(C_{0},\,g_{0})$ is of ``Reeb type\footnote{Symplectic toric cones that are \emph{not} of Reeb type are rather uninteresting: they are either cones over
$S^{2}\times S^{1}$, cones over principal $T^{3}$-bundles over $S^{2}$, or cones over products
$T^{m}\times S^{m+2j-1}$, $m>1,\,j\geq0$ \cite{Ler03}.}'')
the image of the moment map is a strictly convex rational polyhedral cone
$\mu(C_{0})$ in $\mathbb{R}^{n}$
 with the apex of the cone mapping to the origin.
 This means in particular that the image $\mu(C_{0})$ is a subset of $\mathbb{R}^{n}$ of the form
\begin{equation}\label{momentimage}
\mathcal{C}^{*}:=\{y\in\mathbb{R}^{n}\,|\,\langle y,\,v_{i}\rangle\geq0,\,i=1,\ldots,d\},
\end{equation}
where the $v_{i}$ are the inward pointing normal vectors to the $d$ facets of the polyhedral cone. These normals are rational and so can be normalized to be primitive elements of $\mathbb{Z}^{n}$, i.e.,
they cannot be written as \(kv\), where
\(k\in\Zbb_{\geq2}\) and $v$ is also a vector with integral entries. We can also assume that this set is minimal in the sense that removing any vector $v_{i}$ in the definition \eqref{momentimage} of $\mathcal{C}^{*}$ changes $\mathcal{C}^{*}$. 

The condition that $\mathcal{C}^{*}$ be strictly convex means that $\mathcal{C}^{*}$ is a cone over a convex polytope of dimension $(n-1)$. $\mathcal{C}^{*}$ is furthermore ``good'' in the following sense.

\begin{definition}[{\cite[Definition 2.17]{Ler03}}] \label{def:goodcone}
A rational polyhedral cone $\mathcal{C}^{*}$ with $d$ facets of the form
\eqref{momentimage} with
primitive integral minimal inward normals $\mathcal{C}^{*}(1)=\{v_{1},\ldots,v_{d}\}$
is \emph{good} if given any subset $\set{v_{i_1},\ldots,v_{i_k}}$ of these normals such that
\begin{equation} \label{eq:codimkface}
\set{y\in\mathcal{C}^{*}\,|\,\langle y,\,v_{i_j}\rangle = 0\quad\textrm{for all $j = 1, \dots, k$}}
\end{equation}
defines a positive-dimensional face
of $\mathcal{C}^{*}$,
the set $\{v_{i_1}, \dots, v_{i_k}\}$ is linearly independent over $\Zbb$ and 
\[ \set{\sum_{j=1}^k a_j v_{i_j} \mid a_j \in \Rbb} \cap \mathfrak{N} = \set{\sum_{j=1}^k m_j v_{i_j} \mid m_j \in \Zbb}. \] 
\end{definition}
This is a generalization of the Delzant condition to symplectic toric cones. 
The \emph{dual cone} $\mathcal{C}\subset\mathfrak{t}$ of $\mathcal{C}^{*}$ is defined by $$\mathcal{C}:=\{\xi\in\mathbb{R}^{n}\,|\,\langle \xi,\,y\rangle\geq0\quad\textrm{for all $y\in\mathcal{C}^{*}$}\}\subset\mathfrak{t}.$$
This is also a convex rational polyhedral cone by
Farkas’ Lemma. Every face $\tau^{*} \subset \Ccal^{*}$ given by \eqref{eq:codimkface} corresponds bijectively to a face of $\Ccal$ defined by 
\begin{equation} \label{eq:facecorr}
\tau := \Rbb_{\geq 0 }v_{i_1} + \dots + \Rbb_{\geq 0} v_{i_k}
\end{equation}
in such a way that $\dim \tau + \dim \tau^{*} = n$ (cf.~\cite[Proposition 1.2.10]{CLS11}).  
Moreover,  the set $\set{v_{i_1}, \dots, v_{i_k}}$ is exactly the minimal set of primitive generators of $\tau$. 
With this, we have the following characterization of good cones. 

\begin{prop} \label{prop:goodequiv}
Given a rational polyhedral cone $\mathcal{C}^{*}$ of the form \eqref{momentimage}, the following are equivalent. 
\begin{enumerate}
    \item \label{goodmoment} The cone $\mathcal{C}^{*}$ is good.
    \item \label{primitivegenerators} For every proper face $\tau \subsetneq \Ccal$, the minimal set of primitive generators $\set{v_{i_1}, \ldots, v_{i_k}}$ forms part of a lattice basis of $\mathfrak{N}$. Namely, the set is linearly independent over $\Zbb$ and 
 \[ \set{\sum_{j=1}^k a_j v_{i_j} \mid a_j \in \Rbb} \cap \mathfrak{N} = \set{\sum_{j=1}^k m_j v_{i_j} \mid m_j \in \Zbb}. \]
  \item\label{uniquesingularity}
The affine toric variety \(C_0\) is smooth away from its apex, i.e., 
\(
C=C_0\setminus\set{0}
\)
is smooth.  
    \item \label{smoothlink} For any given radius function $r$, the link $L = \set{r=1}$ is a smooth manifold. 
\end{enumerate}
\end{prop}

\begin{proof}
The equivalence \ref{goodmoment} $\Leftrightarrow$ \ref{primitivegenerators} follows from the faces correspondence \eqref{eq:facecorr} and Definition \ref{def:goodcone}, whereas \ref{uniquesingularity} $\Leftrightarrow$ \ref{smoothlink} is clear by Definition \ref{cone}. It remains to show that \ref{primitivegenerators} $\Leftrightarrow$ \ref{uniquesingularity}. From the algebro--geometric point of view, \( C = C_0 \backslash \set{0} \) is a toric variety defined by the fan consisting of all proper faces of $\Ccal$. By \cite[Theorem 3.1.19]{CLS11}, $C$ is smooth if and only if every proper face $\tau$ is smooth in the sense of \cite[Definition 1.2.16]{CLS11}; that is, $\set{v_{i_1}, \dots, v_{i_k}}$ forms part of a lattice basis of $\mathfrak{N}$. This allows us to conclude the result.
\end{proof}

Now, the torus fibration is non-degenerate over the interior $\mathcal{C}^{*}_{0}$ of
$\mathcal{C}^{*}$, and so the $T^{n}$-action is free on the corresponding subset $\mu^{-1}(\mathcal{C}^{*}_{0})$. Moreover, the latter is a Lagrangian torus fibration over $\mathcal{C}^{*}_{0}$.
The boundary $\partial\mathcal{C}^{*}$ of $\mathcal{C}^{*}$ then describes $C_{0}$ as a compactification of $\mathcal{C}^{*}_{0}\times T^{n}$. More precisely, the normal vector $v_{i}$ to a facet $\{\langle y,\,v_{i}\rangle=0\}$ determines a one-cycle in $T^{n}$ which collapses over this facet. Thus, each facet corresponds to a toric symplectic submanifold of $C$ of real codimension two. Similarly, lower-dimensional faces of the cone correspond to higher codimension toric symplectic submanifolds. The fact that the cone is good guarantees that this compactification yields a cone over a smooth manifold.

The image of the link \(\{r=1\}\) under the moment map is the compact
slice
\[
\Delta_{\xi}
:=
\set{
y\in\mathcal{C}^{*}
\mid
2\langle y,\xi\rangle=1
},
\]
as follows from the identity
\[
\langle\mu(x),\xi\rangle
=
\frac{r(x)^2}{2}.
\]
The affine hyperplane
$
H_{\xi}
:=
\set{
y\in\mathfrak{t}^{*}
\mid
2\langle y,\xi\rangle=1
}$
is called the \emph{characteristic hyperplane}, and
$\Delta_{\xi}=H_{\xi}\cap\mathcal{C}^{*}$
is the corresponding characteristic polytope. The compactness of
\(\Delta_{\xi}\) means that
$\langle y,\xi\rangle>0$ for every
$y\in\mathcal{C}^{*}\setminus\set{0}$,
which is equivalent to
$\xi\in\mathcal{C}_{0}$,
the interior of the dual cone \(\mathcal{C}\).
Thus, the data of a toric K\"ahler cone give rise to a good strictly
convex rational polyhedral cone
$\mathcal{C}^{*}\subset\mathfrak{t}^{*}$,
together with an element
$\xi\in\mathcal{C}_{0}\subset\mathfrak{t}$.
Conversely, such data gives rise to a toric K\"ahler cone with Reeb field
generated by \(\xi\) and moment-map image \(\mathcal{C}^{*}\);
cf.~\cite{Ler03,MSY06}.

We next recall how the rational polyhedral cone
\(\mathcal{C}^{*}\) determines the underlying affine toric variety with toric action $(\Cbb^{*})^n$.
This construction depends only on \(\mathcal{C}^{*}\); the Reeb vector
\(\xi\) will reappear in Section \ref{reeb-identify}, where we discuss
the volume functional. Let
\(
\mathcal{S}_{\mathcal{C}^{*}}
=
\mathcal{C}^{*}\cap\mathfrak{M}
\)
be the lattice cone of \(\mathcal{C}^{*}\). This is an abelian semi-group which, by Gordan's Lemma, is finitely generated, i.e., there are a finite number of elements 
$\alpha_{1},\ldots,\alpha_{N_0}\in\mathcal{S}_{\mathcal{C}^{*}}$ such that every element of $\mathcal{S}_{\mathcal{C}^{*}}$ is of the form
$$a_{1}\alpha_{1}+\cdots+a_{N_0}\alpha_{N_0},\qquad a_j\in\Zbb_{\geq0}.$$ 
The minimal set of generators of $\mathcal{S}_{\Ccal^{*}}$ is denoted by $\Hcal$ and goes by the name of \emph{Hilbert basis} in the literature \cite{CLS11}. The embedding of $(\Cbb^{*})^n$ into $(\Cbb^{*})^{N_0}$ given by $\Hcal$ is determined by the coordinate functions
    \[ Z_j := z_1^{\alpha_{1j}} \dots z_n^{\alpha_{nj}}, \]
    where \( \alpha_j = (\alpha_{1j}, \dots, \alpha_{nj} ), \; j = 1, \dots, N_0 \) run over the elements of \( \Hcal \).  The scheme-theoretic closure of the image of this map, given by the
    toric ideal
\[
I
:=
\left\langle
Z^u-Z^v
\;\middle|\;
u,v\in\Zbb_{\geq0}^{N_0},
\quad
\sum_{j=1}^{N_0}u_j\alpha_j
=
\sum_{j=1}^{N_0}v_j\alpha_j
\right\rangle,
\]
defines a minimal embedding of the affine toric cone into $\Cbb^{N_0}$, where the toric action is given by a diagonal subgroup of $\text{GL}(N_0,\Cbb)$ isomorphic to $(\Cbb^{*})^n$
(cf. \cite[Propositions 1.1.14, 1.2.23, $\&$ Lemma 1.3.10]{CLS11}). 
Thus, the corresponding affine toric cone in \(\Cbb^{N_0}\) is cut
out by binomials representing all the linear relations between the elements of $\Hcal$. 
%For example,
%\(
%Z_i^aZ_j^b-Z_k^cZ_l^d\in I_{\Hcal}
%\)
%if and only if
%\[
%a\alpha_i+b\alpha_j-c\alpha_k-d\alpha_l=0.
%\]
    % We next write down all 
%the linear relations among
%the generators. Suppose that there are $S$ such relations.  Then the cone is defined as an intersection of $S$ binomial varieties, i.e., varieties defined by a difference of two monomials  in $\mathbb{C}^{N_0}$, where the $S$ linear relations translate to multiplicative relations among the respective complex variables. For example, the intersection is contained in a quadric given by $Z_i^a Z_j^b - Z_k^c Z_l^d$ if and only if  
%\begin{equation} 
%a \alpha_i + b \alpha_j  - c \alpha_k - d \alpha_l = %0. 
%\end{equation}
The computer code in Appendix \ref{s:A1} (and Figure \ref{figure_macaulay2_cfo13}) provides an explicit expression of the toric ideal for a given toric Calabi--Yau cone. 

\subsubsection{Toric Gorenstein cones}

We begin with the following definition. 

\begin{definition}\label{gorenstein}
A K\"ahler cone $C$ is said to be \emph{Gorenstein} if 
the normal affine cone $C_0$ is Cohen--Macaulay and
the canonical sheaf \( K_{C_0} \) of $C_{0}$ is a locally free sheaf of rank $1$.

A full-dimensional strictly convex rational polyhedral cone $\Ccal \subset \tfrak\cong\Rbb^n$ is called a \textit{Gorenstein cone} if there is a (necessarily unique) primitive element $\gamma \in \mathfrak{M} \subset\mathfrak{t}^{*}$ such that
\[ \sprod{\gamma,\,v} = 1\qquad \textrm{for all $v \in \mathcal{C}(1)$}, \]
where $\Ccal(1)$ is the minimal set of primitive generators of $\Ccal$. 
\end{definition}

Recall from Theorem \ref{t:affine} that a K\"ahler cone (including the apex) has the structure of a normal affine variety. If the cone is in addition toric, then the apex is a rational singularity \cite[Theorem 11.4.2]{CLS11}, hence Cohen--Macaulay \cite[Proposition 1.17]{FZ03}. With this in mind, the relationship between the two notions is given by the following theorem.
\begin{theorem}[{\cite[Fact (1.2)]{Alt94}}]\label{goren-equiv}
Let $(C,T_{\Cbb})$ be a toric Kähler cone
corresponding to the good rational polyhedral cone $\mathcal{C}^{*}$ with dual $\Ccal$. Then the following are equivalent. 
\begin{enumerate}
    \item $(C,T_{\Cbb})$ is Gorenstein. 
    \item $\Ccal$ is a Gorenstein cone. 
\end{enumerate}
\end{theorem}

\begin{remark} \label{rmk:canotrivial}
In the situation of Theorem \ref{goren-equiv}, the canonical sheaf $K_{C_{0}}$, being locally trivial and of rank $1$, defines a line bundle, the canonical bundle of the singular cone $C_{0}$. As it turns out, this line bundle admits a 
$T^{n}$-invariant global trivializing section, explicit on the dense $T_{\mathbb{C}}$-orbit and unique up to scaling; see \cite[Proof of Proposition 3.2]{Ber23} for details. Such a section is called a \emph{holomorphic volume form}. We consider cones with trivial canonical bundle rather than torsion canonical bundle (so-called ``$\mathbb{Q}$-Gorenstein'' cones) because only the former property is preserved under general deformations when the cone singularity is Gorenstein; cf.~\cite[Theorem 2.4]{CH24} and \cite[Theorem 9.1.6]{Ishii}.
\end{remark}

The data of a Gorenstein cone reduce to that of a ``toric diagram'', which
we define next.

\begin{definition}\label{def:toric_diagram}
A polytope \(\Pcal\subset\Rbb^{n-1}\) is called a \emph{toric diagram}
associated with a toric Gorenstein K\"ahler cone with smooth link if:
\begin{enumerate}[label=(\Alph*{})]
\item \label{lattice}
\(\Pcal\) is a full-dimensional convex lattice polytope.

\item \label{properface}
Every proper face \(\Fcal\subsetneq\Pcal\) is a unimodular simplex with
respect to its affine lattice. More precisely, if
\(\dim_{\Rbb}\Fcal=k\) and the vertices of \(\Fcal\) are
\(v_0,\ldots,v_k\), then
\(
v_1-v_0,\ldots,v_k-v_0
\)
form a \(\Zbb\)-basis of
\(
\left(\operatorname{span}_{\Rbb}(\Fcal-v_0)\right)
\cap\Zbb^{n-1}.
\)
\end{enumerate}
\end{definition}

\begin{remark}\label{rmk:toric_diagram}
Let $\widetilde{\mathfrak{N}} \simeq \mathbb{Z}^{n-1}$ be the ambient lattice of $\Pcal$, let
\(
\mathfrak{N}=\widetilde{\mathfrak{N}}\oplus\Zbb
\),
and let
\(
\Ccal=\textnormal{Cone}(\Pcal\times\set{1}).
\)
Condition (B) is equivalent to requiring that, for every nonempty
proper face $\Fcal\subsetneq\Pcal$ with vertices $v_0,\ldots,v_k$,
the vectors $(v_0,1),\ldots,(v_k,1)$ form part of a
$\Zbb$-basis of $\mathfrak{N}$. These are precisely the
primitive generators of the face
\(
\textnormal{Cone}(\Fcal\times\set{1})\subsetneq\Ccal.
\)
Consequently, Proposition \ref{prop:goodequiv} shows that condition
\ref{properface} is equivalent to the dual cone \(\Ccal^{*}\) being
good, or equivalently, the corresponding cone has a smooth link.

When \(n=3\), the toric diagram is a lattice polygon and condition
\ref{properface} simply says that every edge is a primitive lattice
segment, or equivalently, that no edge contains a relative interior
lattice point. In \cite{MSY06}, ``toric diagram'' is also used in the broader sense of a full-dimensional convex lattice polytope without condition \ref{properface}; see \cite[Figure 3]{MSY06}. A polytope in this broader sense may determine a cone with a singular link.
\end{remark}

We next describe the correspondence between toric diagrams and toric
Gorenstein K\"ahler cones.

\begin{prop}[{\cite{CFO08}}]
\label{proposition_combinatorial_qgorenstein}
Given a complex torus
\(
T_{\Cbb}=(\Cbb^{*})^n,
\)
the following data are equivalent:
\begin{enumerate}
\item A good rational polyhedral strictly convex cone
\(
\Ccal^{*}\subset\mathfrak{t}^{*}\cong\Rbb^n
\)
whose dual cone \(\Ccal\) is Gorenstein.

\item A toric diagram
\(
\Pcal\subset\Rbb^{n-1}
\)
in the sense of Definition \ref{def:toric_diagram}, with
\(
\Ccal=\textnormal{Cone}(\Pcal\times\set{1})
\)
and \(\Ccal^{*}\) its dual.
\end{enumerate}
These data determine a toric Gorenstein K\"ahler cone
\((C,T_{\Cbb})\), up to \(T_{\Cbb}\)-equivariant biholomorphisms fixing
the apex.
\end{prop}

\begin{proof}
Suppose first that \(\Pcal\) is a toric diagram. The cone
\[
\Ccal
=
\textnormal{Cone}(\Pcal\times\set{1})
=
\set{\lambda(v,1)\mid v\in\Pcal,\ \lambda\in\Rbb_{\geq0}}
\]
is Gorenstein: the primitive covector
\(
\gamma=(0,\ldots,0,1)\in\mathfrak{M}
\)
satisfies
\(
\sprod{\gamma,(v,1)}=1
\)
for every primitive ray generator \((v,1)\) of \(\Ccal\). Every proper face of \(\Ccal\) is of the form
\(
\textnormal{Cone}(\Fcal\times\set{1})
\)
for a proper face \(\Fcal\subsetneq\Pcal\). By condition
\ref{properface}, its primitive generators form part of a lattice basis
of \(\mathfrak{N}\). Proposition \ref{prop:goodequiv} therefore implies that
the dual cone \(\Ccal^{*}\) is good.

Conversely, suppose that \(\Ccal^{*}\) is good and that its dual
\(\Ccal\) is Gorenstein. Let \(\gamma\in\mathfrak{M}\) be the primitive
Gorenstein covector. After applying an element of
\(\SL(n,\Zbb)\), we may assume that
\(
\gamma=(0,\ldots,0,1);
\)
see \cite[Proposition 3.3]{CFO08}. The height-one slice of \(\Ccal\)
then has the form
\[
\Ccal\cap\set{x\in\Rbb^n\mid\sprod{\gamma,x}=1}
=
\Pcal\times\set{1}
\]
for a full-dimensional convex lattice polytope
\(\Pcal\subset\Rbb^{n-1}\). Since \(\Ccal^{*}\) is good, Proposition \ref{prop:goodequiv} implies
that the primitive generators of every proper face of \(\Ccal\) form
part of a lattice basis of \(\mathfrak{N}\). Applying this to the face
corresponding to each proper face \(\Fcal\subsetneq\Pcal\), we conclude
that every proper face of \(\Pcal\) is a unimodular simplex. Thus,
\(\Pcal\) is a toric diagram in the sense of Definition
\ref{def:toric_diagram}.
\end{proof}

We conclude this subsection with an example.
\begin{example} \label{example_reflexive_cone}
A Gorenstein cone \( \mathcal{C} \subset \Rbb^n \) is called \emph{reflexive} if its dual \( \mathcal{C}^{*} \) is also Gorenstein. A Gorenstein cone \( \mathcal{C} \) is reflexive (of index one) if and only if its toric diagram $\mathcal{P}$ is reflexive \cite[Proposition 2.11]{BB97}. Namely, $\mathcal{P}$ contains only one lattice point in its interior (which we translate to the origin). The \emph{dual polytope} $\mathcal{P}^{*}$ of 
$\mathcal{P}$, defined by
\[ \mathcal{P}^{*} := \set{u \in \Rbb^{n-1} \mid \sprod{u,\,v} \geq -1\quad\textrm{for all $v \in \mathcal{P}$}},\]
is also a lattice polytope. 
In the reflexive case, condition
\ref{def:toric_diagram}\ref{properface} is equivalent to
\(\Pcal\) being a smooth Fano polytope: the vertices of every facet of
\(\Pcal\) form a \(\Zbb\)-basis of \(\Zbb^{n-1}\). Equivalently, the
dual reflexive polytope \(\Pcal^{*}\) is Delzant.

In dimension two, there are precisely five lattice-equivalence classes
of smooth reflexive polygons. These are the polygons
\(\Qcal_1,\Qcal_2,\Qcal_3,\Qcal_4,\Qcal_5\) associated with the toric
del Pezzo surfaces
\[
\Pbb^1\times\Pbb^1,\qquad
\Pbb^2,\qquad
\textnormal{Bl}_{p_1}(\Pbb^2),\qquad
\textnormal{Bl}_{p_1,p_2}(\Pbb^2),\qquad
\textnormal{Bl}_{p_1,p_2,p_3}(\Pbb^2),
\]
respectively; see Figure \ref{figure_delpezzo_family}.
\end{example}

\subsubsection{Toric Calabi--Yau cones}

For a toric K\"ahler cone, being Gorenstein is equivalent to admitting
a toric Calabi--Yau cone metric.

\begin{theorem}[{\cite[Theorem 1.2]{CFO08} \&
\cite[Theorem 1.2]{FOW}}]\label{equiv}
Let \((C,g,T_{\Cbb})\) be a toric K\"ahler cone.
The following are equivalent.
\begin{enumerate}
\item There exists a toric K\"ahler cone metric \(g_0\) on \(C\) such
that
\(
(C,g_0,T_{\Cbb})
\)
is a toric Calabi--Yau cone in the sense of Definition
\ref{d:cycone}.
\item The affine toric cone
\(
(C_0,T_{\Cbb})
\)
is Gorenstein in the sense of Definition \ref{gorenstein}.
\end{enumerate}
\end{theorem}

\section{Identification of the Reeb field}\label{reeb-identify}

The main aim of this section is to describe
how to identify, via computer code, the Reeb field of a toric Calabi--Yau cone through the volume function and index character. Section \ref{subsection_volume_minimizer} 
introduces these latter two notions. Using 
their properties, the algorithm to find the Reeb field is contained in Section
\ref{s:minimizer}. Examples are then given in Section \ref{exampless}.

\subsection{Volume function and index character} 

\label{subsection_volume_minimizer} 

Let $(C_0,T_{\Cbb})$ be a toric Kähler cone given by a good polyhedral cone $\Ccal^{*} \subset \tfrak^{*} \simeq \Rbb^n$ with dual cone $\Ccal \subset \tfrak$. Recall that in the given lattice basis $(X_1,\dots,X_n)$ of the coweight lattice $\mathfrak{N} = \ker(\exp: \tfrak \to T^n)$, a general Reeb field of $C_0$ can be written as 
\[ \xi = \sum_{j=1}^n \xi_j X_j, \]
where $(\xi_1, \dots, \xi_n) \in \mathcal{C}_0\subset\mathfrak{t}$, the interior cone of $\Ccal$. The Reeb field $\xi$ is then \emph{quasi-regular} if and only if a positive real multiple of $\xi$ lies in $\mathfrak{N}$. Otherwise, $\xi$ is \emph{irregular}. In \cite[equation (2.51)]{MSY06}, based on Guillemin's boundary condition \cite{Guil} and the construction of toric cones by symplectic reduction, the authors showed that the space of smooth Kähler cone potentials on $C = C_0 \backslash \set{0}$ is naturally in one-to-one correspondence with
\[ \mathcal{K}_c = \mathcal{C}_0 \times \Hcal(1), \]
where
$\Hcal(1)$ is the space of smooth homogeneous degree one functions on $\mathcal{C}^{*}$ satisfying a convexity condition. The correspondence is given as follows. By Proposition \ref{propB6}, the Kähler cone metric $\omega_0$ (see Definition \ref{kahlercone}) can be written as $\omega_0 = 2i \del \delb \phi$ on the open dense orbit  $(\Cbb^{*})^n$. Here, the $T^n$-invariant Kähler cone potential $\phi$ is naturally identified with a smooth and strictly convex function on $\tfrak \simeq \Rbb^n$.
Let $\psi:\nabla\varphi(\mathfrak{t})\subset\mathfrak{t}^{*} \to\mathbb{R}$ denote the
Legendre transform of $\varphi$; cf.~\cite[Section 2.5.5]{CCD}.
Recall that $\set{v_1, \dots, v_d} \subset \tfrak$ is the set of primitive inward-pointing normal vectors to the $d$ facets of $\Ccal^{*}$. Let $\xi$ be the Reeb field associated with $\omega_0$. Then in the symplectic coordinates $(y_1,\dots,y_n)$ of $\tfrak^{*}$, $\psi$ has the following expression: 
\begin{equation} \label{eq:reebconemetric}
\psi(y) = h(y) + G_{\xi}(y) + G^{\textnormal{can}}(y), 
\end{equation}
where $h \in \Hcal(1)$ and given the linear functions $l_i(y) = \sprod{v_i,y}$, $l_{\xi}(y) = \sprod{\xi, y}$, and $l_{\infty}(y) = \sprod{\sum_{i=1}^d v_i, y}$, we have that
\[ G_{\xi}(y) = \frac{1}{2} l_{\xi}(y) \log l_{\xi}(y) - \frac{1}{2} l_{\infty}(y) \log l_{\infty}(y), \qquad G^{\textnormal{can}}(y) = \frac{1}{2} \sum_{i=1}^d l_{i}(y) \log l_i(y). \]

On the other hand, from the variational point of view developed in \cite{MSY06}, the potential of the Calabi--Yau metric on a Gorenstein toric cone is a critical point of the \emph{Einstein--Hilbert action} $\mathcal{S} \colon \mathcal{K}_c \to \Rbb$:  
\begin{equation*}
 \mathcal{S}(\phi)\coloneqq \int_L\left(R_L + 2(n-1)(3-2n)\right)\,d\mu_{L},
\end{equation*}
where $\phi$ is the $T^n$-invariant potential of a Kähler cone metric $g_0$, and $R_L$ and $d \mu_{L}$ are the scalar curvature and volume form of the metric $g_L =g_0|_L$ on the link $L$ of the cone, respectively. 
A crucial observation of Martelli--Sparks--Yau is that the Einstein--Hilbert functional $\mathcal{S}$ actually depends only on $\xi$ and not on the homogeneous part $\Hcal(1)$ \cite[equation (3.9)]{MSY06}. 
Moreover, given a toric Calabi--Yau cone $(C,g_0,J_0,\Omega_0)$ with Reeb field $\xi$ and holomorphic volume form $\Omega_{0}$ invariant under the action of  $T^{n}$ (see Remark \ref{rmk:canotrivial}),
$\Omega_0$ necessarily satisfies the \emph{gauge-fixing condition} \cite[equation (2.65)]{MSY06}
\begin{equation} \label{eq:reebnormalization}
\Lcal_{\xi} \Omega_0 = \sqrt{-1}n\Omega_0, 
\end{equation}
where $\Lcal_{\xi}$ is the Lie derivative along $\xi$. This determines one of the coordinates of the Reeb field of the Calabi--Yau cone metric.

\begin{lemma} [{\cite[equation (2.92)]{MSY06}}]\label{lem:reebnormalization} 
Let \((C,T_{\Cbb})\) be an \(n\)-dimensional toric Gorenstein K\"ahler
cone associated with
$\Ccal=\textnormal{Cone}(\Pcal\times\set{1})$,
where $\Pcal\subset\Rbb^{n-1}$ is a toric diagram. Let \(\gamma = (0,\dots,0,1) \in \mathfrak{M} \) denote the primitive Gorenstein covector and
define the \emph{normalized Reeb slice}
\[ \Sigma_{\gamma}
:=
\set{\xi\in\Ccal_0\mid\sprod{\gamma,\xi}= \xi_n = n} = \set{(nv,n) \in \Ccal_0, v \in \operatorname{int}(\Pcal)}.\]
Then the gauge-fixing condition \eqref{eq:reebnormalization} is equivalent to the Reeb field $\xi=(\xi_{1},\ldots,\xi_{n})$ of the Calabi--Yau cone metric lying in $\Sigma_{\gamma}$. 
\end{lemma}

As remarked in \cite[p.56]{MSY06}, finding a critical metric of the functional $\mathcal{S}$ under the gauge-fixing condition then boils down to considering the function
\[\textnormal{vol}[L]: \Ccal_0 \to \Rbb, \]
defined such that $\textnormal{vol}[L](\xi)$ is the volume of $L$ with respect to the metric obtained via the relation \eqref{eq:reebconemetric}, where we can take $h = 0$ since $\mathcal{S}$ does not depend on the $\Hcal(1)$-part. 
In addition, $\textnormal{vol}[L](\xi)$ is essentially the Euclidean volume of the truncated cone
\[  \set{y \in \mathcal{C}^{*}\,|\, \sprod{\xi,y} \leq \frac{1}{2}}; \]
see \cite[equation (2.74)]{MSY06}.

We define the ``volume function'' as the following ratio. 
\begin{definition} \label{d:volume}
For a given Reeb field $\xi \in \mathcal{C}_0$, the \emph{volume function} $a_0 (\xi)$ is defined as the volume of $L$ normalized by the volume of the unit sphere $S^{2n-1}$, i.e.,
\begin{equation*}
a_0(\xi) := \frac{\textnormal{vol}[L](\xi)}{\textnormal{vol}[S^{2n-1}]}.    
\end{equation*}
\end{definition}
\noindent Its value at the volume minimizer is always an algebraic number
\cite{MSY08}.

\begin{remark}
In \cite{CS18}, the volume function is defined as $a_0(\xi) / (n-1)!.$
\end{remark}
 
\noindent This function has the following properties.
\begin{theorem}[{\cite[Section 3]{MSY06} and
\cite[Sections 3.2--3.3]{MSY08}}]
\label{theorem_volume_minimization}
Let \((C,T_{\Cbb})\) be an \(n\)-dimensional toric Gorenstein K\"ahler
cone associated with
$\Ccal=\textnormal{Cone}(\Pcal\times\set{1})$,
where $\Pcal\subset\Rbb^{n-1}$ is a toric diagram. Then the volume
function
$a_0:\Ccal_0\longrightarrow\Rbb_{>0}$
is smooth, strictly convex, and homogeneous of degree \(-n\); that is,
\[
a_0(\lambda\xi)=\lambda^{-n}a_0(\xi)
\]
for every \(\xi\in\Ccal_0\) and \(\lambda>0\). Moreover, the restriction to the normalized Reeb slice
$a_0|_{\Sigma_{\gamma}}:\Sigma_{\gamma}\longrightarrow\Rbb_{>0}$
is proper and strictly convex. In particular, it admits a unique
minimizer.
\end{theorem}
\noindent In what follows, by the \emph{volume minimizer} we always mean the unique minimizer of \(a_0\) on the normalized Reeb slice \(\Sigma_{\gamma}\), or equivalently, in the coordinates above, under the constraint \(\xi_n=n\).

Recall Theorem \ref{equiv}. We also have the
following characterisation of toric Calabi--Yau cones 
given by the existence theorem of Futaki--Ono--Wang \cite{FOW}, which can be reformulated in algebro-geometric terms in the following way.

\begin{theorem}[{\cite{FOW}}]
Let $(C,\,g_0,\,J_0,\,T_{\Cbb})$ be an $n$-dimensional toric Gorenstein K\"ahler cone and $\omega_0 = g_0(J_0(\,\cdot\,),\,\cdot) = \frac{i}{2} \del_{J_0} \delb_{J_0} r_0^2$ the associated K\"ahler form.
The following are equivalent. 
\begin{enumerate}
    \item There is a smooth function $f$ on $C$ such that $r^2 = r_0^2 e^f$ is the K\"ahler cone potential of a toric Calabi--Yau metric $\omega =  \frac{i}{2}\del_{J_0} \delb_{J_0}  r^2$ (having the same Reeb field as $\omega_0$).
   \item The corresponding Reeb field \(\xi\) is the unique minimizer of the volume function \(a_0\) on the normalized Reeb slice \(\Sigma_{\gamma}\), where $\Sigma_{\gamma}$ is as in Lemma \ref{lem:reebnormalization}. 
\end{enumerate} 
\end{theorem}

Computationally, the volume function can be determined using the ``index character'', introduced by Martelli--Sparks--Yau \cite{MSY08} and refined by Collins--Székelyhidi \cite{CS18} in a rigorous and purely algebraic fashion for affine cones with more general singularities. It is this last point of view that we adopt in the following. 

Given a toric affine variety $(C_0,T_{\Cbb})$, recall that the Hilbert basis $\Hcal$ defines a minimal embedding $C_0 \subset \Cbb^{N_0}$ with the corresponding toric ideal $I \subset R := \Cbb[Z_1, \dots, Z_{N_0}]$, and a faithful representation $T_{\Cbb} \hookrightarrow \textnormal{GL}(N_0,\Cbb)$ such that the image of $T_{\Cbb}$ is a diagonal subgroup acting effectively on $C_0$ (see the discussion on Hilbert bases in the end of Section \ref{subsubsection:torickahlercones}). Let $R[C_0] = R/I$ be the ring of regular functions of $C_0$. Then $R[C_0]$ inherits an action of $T^n$ and has a weight decomposition as a $T^n$-module:
\begin{equation} \label{eq:weight_decomposition}
R[C_0]
=
\bigoplus_{\alpha\in\mathfrak{M}}R_\alpha,
\end{equation}
where \(R_\alpha\) is the \(\Cbb\)-vector subspace
\(
R_\alpha
=
\set{
f\in R/I
\mid
\Lcal_\zeta f
=
 \sqrt{-1}\alpha(\zeta)f,\ 
\forall\zeta\in\tfrak
}.
\)
With this data, the index character is defined as follows.

\begin{definition}[{\cite[equation (1.16)]{MSY08} \& \cite[Definition 4.1]{CS18}}] \label{d:index}
The (\emph{equivariant}) \emph{index character} is defined for $\xi \in \mathcal{C}_0$ and $t \in \Cbb$ with $\Re(t) > 0$ (for now as a formal series) by 
\[ F(\xi, t) = \sum_{\alpha \in \mathfrak{M}} e^{-t \alpha(\xi)} \dim R_{\alpha}. \]
Equivalently, this is the trace of the diagonal operator acting by
\(e^{-t\alpha(\xi)}\) on \(R_\alpha\). 
\end{definition}

Let $\alpha_j$ be the weight of the torus action $T^n$ on $Z_j$. Fixing a \(\Zbb\)-basis of \(\mathfrak{M}\) dual to the given basis of $\mathfrak{N}$, and expressing $\alpha_j$ as $(\alpha_{1j}, \dots, \alpha_{nj}) \in \Zbb^n$ gives rise to an $(n \times N_0)$-matrix $W$ with integral coefficients; the set of columns of which turn out to be precisely the Hilbert basis of the lattice cone $\mathcal{S}_{\Ccal^{*}}$.
If \( e_1, \dots, e_n \) are the generators of the \( T^n \)-action, i.e., a basis of $\mathfrak{N} = \ker(\exp:\tfrak \to T^{n})$,
then  \( \alpha_{ij} \) is the weight of the action of \( e_i \) on \( Z_j \). 
 In particular, every $\xi \in \tfrak$ with coordinates $(\xi_1,\dots,\xi_n)$ induces an action
\[
\Lcal_\xi Z_j
=
\sqrt{-1}\alpha_j(\xi)Z_j
=
\sqrt{-1}(\alpha_{1j}\xi_1+\dots+\alpha_{nj}\xi_n)Z_j.
\]
More precisely, the faithful representation
\(
T^n\longrightarrow (S^1)^{N_0}
\)
induces an injective homomorphism of integral lattices
\[
\mathfrak{N}\longrightarrow\Zbb^{N_0},
\qquad
u\longmapsto
\bigl(
\sprod{\alpha_1,u},
\dots,
\sprod{\alpha_{N_0},u}
\bigr).
\]
Since
\(
\alpha_{ij}=\sprod{\alpha_j,e_i},
\)
the image of \(e_i\) under this homomorphism is
\[
(\alpha_{i1},\dots,\alpha_{iN_0}),
\]
which is precisely the \(i\)-th row of \(W\). Thus, after identifying
\(\mathfrak{N}\) with its image in \(\Zbb^{N_0}\), the elements
\(e_1,\dots,e_n\) can be identified with the rows of \(W\):
\begin{equation}\label{eq:hilbertbasismatrix}
W=(\alpha_{ij})
=
\begin{pmatrix}
e_1\\
\vdots\\
e_n
\end{pmatrix}.
\end{equation}
Clearly, $(\xi_1,\dots,\xi_n) \in \Ccal_0$ if and only if
\(
\xi_1\alpha_{1j}+\dots+\xi_n\alpha_{nj}>0
\)
for all \(1\leq j\leq N_0\). Equivalently, after identifying
\(e_1,\dots,e_n\) with the rows of \(W\) as above, the row vector
\(
\xi_1e_1+\dots+\xi_ne_n\in\Rbb^{N_0}
\)
has all of its components strictly positive. The Reeb cone is thus
determined by the \(N_0\) inequalities
\begin{equation}
\mathcal{C}_0
=
\set{
(\xi_1,\dots,\xi_n)\in\Rbb^n
\mid
\xi_1e_1+\dots+\xi_ne_n\in\Rbb_{>0}^{N_0}
}.
\end{equation}
In this situation, we say that $R = \Cbb[Z_1,\dots,Z_{N_0}]$ is \emph{graded} by $W$. Since $I$ is homogeneous, the ring $R[C_0] = R/I$ also inherits a $W$-grading. For every weight $\alpha$ in the decomposition \eqref{eq:weight_decomposition}, the \emph{multigraded Hilbert function of $R[C_0]$} (\emph{with respect to the grading $W$}) is defined as
\[ HF_{W, I}(\alpha) := \dim_{\Cbb} R_{\alpha}.\]
By choice of the basis on $\tfrak^{*}$, we can identify every weight $\alpha$ with an integral vector $(\alpha_1,\dots, \alpha_n) \in \Zbb^n$. The \emph{formal multigraded Hilbert series of $R[C_0]$} (\emph{with respect to the grading $W$}) is given by 
\begin{equation} \label{eq:hilbertseries}
H_{W,I}(z_1,\dots,z_n) := \sum_{(\alpha_1,\dots,\alpha_n) \in \Zbb^n} HF_{W,I}(\alpha_1,\dots,\alpha_n) z_1^{\alpha_1} \dots z_n^{\alpha_n}. 
\end{equation}
Using intrinsic properties of the Hilbert series, Collins and Sz\'ekelyhidi showed that 
\begin{equation} \label{eq:indexchar}
F(\xi,t) = H_{W, I}(e^{-\xi_1 t}, \dots, e^{-\xi_n t});
\end{equation}
see \cite[Lemma 4.8 \& Proof of Theorem 4.10]{CS18} for details. 
The index character converges for Reeb fields and extends as a meromorphic function to all of $\mathbb{C}$, as noted by the next proposition.
\begin{prop}[{\cite[Lemma 4.2 \& Theorem 4.10]{CS18}}] \label{prop:indexcharexpansion}
The formal series $F(\xi,t)$ converges if $\xi$ is a Reeb field and $\Re(t) > 0$. Moreover, $F(\xi,t)$ as a function of $t$ has a meromorphic extension to $\Cbb$ with poles along $\Re(t) = 0$ and with a Laurent expansion near $t = 0$ given by 
\[ F(\xi,t) = \frac{a_0(\xi)}{t^{n}} + \frac{a_1(\xi)}{t^{n-1}} + O(t^{2-n}),\]
where $a_0(\xi)$ is the volume function of $C_0$. Namely, 
\[ a_0(\xi) = \lim_{t \to 0} t^n F(\xi,t). \]
\end{prop}
\noindent This last expression recovers \cite[equation (1.17)]{MSY08}. 

\subsection{Computing the volume minimizer}\label{s:minimizer}
By virtue of its purely algebraic character, the normalized volume minimizer can be computed algorithmically using built-in functions in suitable mathematical software. For this purpose, we devise a scheme, with the main steps specialized to three-dimensional toric K\"ahler cones, as they form our main class of interest (see Section \ref{section_deformation}). 

\begin{itemize}
 \item Given the moment cone \( \mathcal{C}^{*} \subset \mathfrak{M}_{\mathbb{R}} \) of a toric affine cone $C_0$, find the Hilbert basis \( \Hcal \) of the lattice cone \( \mathcal{S}_{\mathcal{C}^*} := \mathcal{C}^{*} \cap \mathfrak{M} \). This yields an embedding of \( C_0 \) in \( \mathbb{C}^{N_0}, \) where \( N_{0} \) is the cardinality of \( \Hcal \). The basis $\Hcal$ is represented in a $(3 \times N_0)$ -matrix $W$ as in \eqref{eq:hilbertbasismatrix}.
    \item Use computer algebra to obtain the ideal \(  I\) of \( C_0 \subset \mathbb{C}^{N_0} \) by viewing \( R = \mathbb{C}[Z_1, \dots, Z_{N_0}] \) as a ring graded by $W$. 
    \item Compute the multivariate Hilbert series \( H_{W,I} \) as in \eqref{eq:hilbertseries}. In the non-reduced form, this is a three-variable series with denominator being given by
    \[ \prod_{j=1}^{N_0} (1 - T_0^{\alpha_{1j}} T_1^{\alpha_{2j}} T_2^{\alpha_{3j}}). \] 
    \item  Recall that the Reeb cone \( \mathcal{C}_0 \) (i.e., the interior of \( \mathcal{C} \)) is determined by the $N_0$ inequalities
    \[ \mathcal{C}_0 = \set{(a,b,c) \in \Rbb^3 \mid  a e_1 + b e_2 + c e_3 \in\mathbb{R}_{>0}^{N_{0}}}. \]  
    For a given Reeb field \( \xi = a e_1 + b e_2 + c e_3 \), the index character \eqref{eq:indexchar} is obtained from $H_{W,I}$ by the change of variables 
    \[ F(\xi,t) = H_{W,I}(e^{-a t}, e^{-b t}, e^{-c t}), \]
    which admits an expansion near $t=0$ up to order \( t^{-1} \), namely 
        \[ F(\xi,t) = a_0(\xi) t^{-3} + a_1(\xi) t^{-2} + O(t^{-1}), \]
        by Proposition \ref{prop:indexcharexpansion}.
    \item Finally, minimize \( a_0(\xi)\) under the constraints \( \xi \in \Ccal_0 \) and $c = 3$ (cf.~Lemma \ref{lem:reebnormalization} and Theorem \ref{theorem_volume_minimization}). 
\end{itemize}

To carry out the above steps, we run codes in Macaulay2 to compute the Hilbert series of a toric cone given its toric diagram, then proceed to obtain an expansion of the index character using Mathematica. The Calabi--Yau Reeb field is finally obtained using the \verb|Minimize| function in Mathematica. The translation of Macaulay2 strings to Mathematica-friendly ones requires a few more lines of code and is necessary in order to make use of the \verb|Minimize| function, which, to our knowledge, is not available in Macaulay2. For details of the Macaulay2 and Mathematica codes, we refer the reader to Appendices \ref{s:A1} and \ref{s:A2} (as well as to the companion Figures \ref{figure_macaulay2_cfo13} and \ref{figure_mathematica_cfo13}), respectively.

\subsection{Examples}\label{exampless}
The following examples only concern affine cones over toric del Pezzo surfaces embedded using their anti-canonical bundle. In each example, we compute the Hilbert series using the code in Appendix \ref{s:A1} and recover the Calabi--Yau Reeb field and minimized volume value by means of the code in Appendix \ref{s:A2}. The toric diagram input and relevant outputs are also provided. 

\begin{example}
Consider the del Pezzo cone $C_1$ over \( \Pbb^1 \times \Pbb^1 \), embedded using \( \Ocal(2,2) \). This slightly differs from the cone referred to as the ``conifold'' in the literature \cite{CdlO} (see  \cite[Section 3]{MSY06} and \cite[Section 7.4]{MSY08} for the toric point of view), which is precisely the affine cone over $\Pbb^1 \times \Pbb^1$ embedded using $\Ocal(1,1)$. The conifold is actually an algebraic finite covering of degree two over $C_1$, ramified at the apex. In our situation, the toric diagram of $C_1$ is given by 
\[\Qcal_1 := \textnormal{Conv} \set{ \Spvek{-1;0},\Spvek{0;-1},\Spvek{1;0},\Spvek{0;1}}.\]
The dual polygon $\Qcal_1^{\vee}$ is then given by $\Qcal_1^{\vee} := \textnormal{Conv} \set{\Spvek{-1;1},\Spvek{-1;-1},\Spvek{1;-1},\Spvek{1;1}}$. Since $\Qcal_1$ is reflexive, the moment cone $\mathcal{C}^{*}$ of $C_1$ is precisely
\[ \mathcal{C}^{*} = \textnormal{Cone}(\Qcal_1^{\vee} \times \set{1}) = \set{(\lambda u, \lambda) \mid u \in \Qcal_1^{\vee}, \lambda \in \Rbb_{\geq 0}}; \]
see \cite[Definition 2.10]{BB97}. By \cite[Lemma 2.2.14]{CLS11}, the Hilbert basis of the semigroup $\mathcal{S}_{\mathcal{C}^{*}} = \mathcal{C}^{*} \cap \Zbb^3$ is exactly the set of lattice points of $\Qcal_1^{\vee} \times \set{1}$, comprising nine points identified with the columns of the weight matrix
\[ W = \begin{pmatrix}
-1 & -1 & -1 & 0 & 0 & 0 & 1 & 1 & 1\\
-1 & 0 & 1 & -1 & 0 & 1 & -1 & 0 & 1 \\
1 & 1 & 1 & 1 & 1 & 1 & 1 & 1 & 1
\end{pmatrix}. \] 
The toric ideal \(I\) is generated by the following \(20\) quadrics:
\begin{align*}
&z_{8}^{2}-z_{7}z_{9},\,z_{6}z_{8}-z_{5}z_{9},\,z_{5}z_{8}-z_{4}z_{9},\,z_{3}z_{8}-z_{2}z_{9},\\
&z_{2}z_{8}-z_{1}z_{9},\,z_{6}z_{7}-z_{4}z_{9},\,z_{5}z_{7}-z_{4}z_{8},\,z_{3}z_{7}-z_{1}z_{9},\\
&z_{2}z_{7}-z_{1}z_{8},\,z_{6}^{2}-z_{3}z_{9},\,z_{5}z_{6}-z_{2}z_{9},\,z_{4}z_{6}-z_{1}z_{9}, \\
&z_{5}^{2}-z_{1}z_{9},\,z_{4}z_{5}-z_{1}z_{8},\,z_{3}z_{5}-z_{2}z_{6},\,z_{2}z_{5}-z_{1}z_{6}, \\
&z_{4}^{2}-z_{1}z_{7},\,z_{3}z_{4}-z_{1}z_{6},\,z_{2}z_{4}-z_{1}z_{5},\,z_{2}^{2}-z_{1}z_{3}. 
\end{align*}
The Hilbert series of \( (W,I) \) is then given by 
\[ H_{W,I} (T_0,T_1,T_2) =\frac{1+T_{2}(1 + T_0+ T_0^{-1} + T_1 + T_1^{-1}) -T_{2}^2(1 + T_0 + T_0^{-1} + T_1 + T_1^{-1})-T_{2}^{3}}{\left(1-T_{0}^{-1}T_{1}^{-1}T_{2}\right)\left(1-T_{0}^{-1}T_{1}T_{2}\right)\left(1-T_{0}T_{1}^{-1}T_{2}\right)\left(1-T_{0}T_{1}T_{2}\right)}. \]
The volume is thus
\[ a_0(\xi) = \frac{8 c}{(a - b - c) (a + b - c) (a - b + c) (a + b + c)}. \]
Under the gauge fixing condition \eqref{eq:reebnormalization}, we have \( c = 3 \) by Lemma \ref{lem:reebnormalization}, and
the volume is minimized by \( \xi_0 = (0,0,3)\) with minimum value
\[ a_0(\xi_0) = \frac{8}{27}. \]
\end{example}

\begin{example}
Consider the cone \( C_3 \) over the del Pezzo surface of degree \( 8 \) (which is the blowup \( \textnormal{Bl}_{p_1} (\Pbb^2) \) of $\Pbb^2$ at one point). The toric diagram of $C_3$ is given by the polygon 
\[\Qcal_3 := \textnormal{Conv} \set{\Spvek{-1;0}, \Spvek{-1;-1}, \Spvek{1;0}, \Spvek{0;1}}. \] 
The dual polygon is then
\( \Qcal_3^{\vee} := \textnormal{Conv} \set{ \Spvek{-1;2}, \Spvek{-1;-1},\Spvek{1;-1}, \Spvek{1;0}} \). Again, by reflexivity, the moment cone $\mathcal{C}^{*}$ of $C_3$ is $ \mathcal{C}^{*} = \textnormal{Cone}(\Qcal_3^{\vee} \times \set{1})$. The Hilbert basis of the semigroup $\mathcal{S}_{\mathcal{C}^{*}} = \mathcal{C}^{*} \cap \Zbb^3$ is exactly the set of lattice points of $\Qcal_3^{\vee} \times \set{1}$, consisting of nine points identified with the columns of the weight matrix
\[ W = \begin{pmatrix}
-1 & -1 & -1 & -1 & 0 & 0 & 0 & 1 & 1 \\
-1 & 0 & 1 & 2 & -1 & 0 & 1 & -1 & 0 \\
1 & 1 & 1 & 1 & 1 & 1 & 1 & 1 & 1 
\end{pmatrix}. \] 
Then \( C_3 \) is the intersection of the following \( 20 \) quadrics in \( \mathbb{C}^{9} \):
\begin{align*}
&z_{7}z_{8}-z_{6}z_{9},\,z_{6}z_{8}-z_{5}z_{9},\,z_{4}z_{8}-z_{3}z_{9},\,z_{3}z_{8}-z_{2}z_{9},\,z_{2}z_{8}-z_{1}z_{9},\\
&z_{7}^{2}-z_{4}z_{9},\,z_{6}z_{7}-z_{3}z_{9},\,z_{5}z_{7}-z_{2}z_{9},\,z_{6}^{2}-z_{2}z_{9},\,z_{5}z_{6}-z_{1}z_{9},\\
&z_{5}^{2}-z_{1}z_{8},\,z_{4}z_{6}-z_{3}z_{7},\,z_{3}z_{6}-z_{2}z_{7},\,z_{2}z_{6}-z_{1}z_{7},\,z_{4}z_{5}-z_{2}z_{7},\\
&z_{3}z_{5}-z_{1}z_{7},\,z_{2}z_{5}-z_{1}z_{6},\,z_{3}^{2}-z_{2}z_{4},\,z_{2}z_{3}-z_{1}z_{4},\,z_{2}^{2}-z_{1}z_{3}.
\end{align*}
The Hilbert series of $(W,\,I)$ is 
\begin{align*} 
H_{W,I}(T_0,T_1,T_2) = &\frac{1+T_{1}T_{2}+T_{2}+T_ {1}^{-1}T_{2}+T_ {0}^{-1}T_{1}T_{2} }{G(T)} \\
&\frac{+T_ {0}^{-1}T_{2}-T_{0}T_ {2}^{2}-T_{0}T_ {1}^{-1}T_ {2}^{2}-T_{1}T_{2}^{2}-T_ {2}^{2}-T_ {1}^{-1}T_ {2}^{2}-T_ {2}^{3}}{G(T)},
\end{align*}
where  $$G(T) = {\left(1-T_{0}T_{2}\right)\left(1-T_{0}^{-1}T_{1}^{-1}T_{2}\right)\left(1-T_{0}T_{1}^{-1}T_{2}\right)\left(1-T_{0}^{-1}T_{1}^{2}T_{2}\right)}.$$ 
The volume of a Reeb field $\xi = (a,b,c)$ is then 
\[ a_0(\xi) = \frac{2 (2 a - b + 4 c)}{(a + c) (a + b - c) (a - b + c) (a - 2 b - c)}. \] 
Under the gauge fixing condition \( c = 3 \), we find that
\[ a_0(\xi_0) = \frac{46 + 13 \sqrt{13}}{324}, \qquad a = -4 + \sqrt{13},\qquad b = 0.\]
We thus recover the Calabi--Yau Reeb field and minimized volume found in \cite[equation (5.43)]{MSY08}. Note that the value of the volume in \cite{MSY08} is written as ``$\displaystyle\frac{43 + 13 \sqrt{13}}{324}$'', but this is a typo as the reader can readily check by plugging the Calabi--Yau Reeb field given therein as $(b_1,b_2,b_3) =(3, -4+\sqrt{13},0)$ into the volume formula \cite[equation (5.42)]{MSY08}
\[V(b) = \frac{8b_1 + 4b_2}{(b_1^2 - b_2^2)^2},\]
obtained using non-toric methods under the assumption that the metric is only $T^2$-invariant. 
\end{example}

\begin{example} \label{ex:dp2reeb}
Let \( C_4 \) be the cone over the del Pezzo surface of degree \( 7 \) (which is geometrically \( \textnormal{Bl}_{p_1,\,p_2}(\Pbb^2) \)). The  toric diagram of the cone is given by the lattice polytope 
\[ \Qcal_4 := \textnormal{Conv}\set{\Spvek{-1;1},\Spvek{-1;0}, \Spvek{0;-1}, \Spvek{1;0},\Spvek{0;1}}.\]
The dual polygon is then
\( \Qcal_4^{\vee} := \textnormal{Conv} \set{ \Spvek{-1;1}, \Spvek{-1;-1}, \Spvek{0;-1}, \Spvek{1;0}, \Spvek{1;1}} \). By reflexivity,
\(
\Ccal^*
=
\textnormal{Cone}(\Qcal_4^\vee\times\set{1}).
\)
For the computation below, we apply the lattice automorphism
\[
A:\Zbb^2\longrightarrow\Zbb^2,
\qquad
A(x,y)=(-y,x),
\]
and express \(\Ccal^*\) in the resulting lattice coordinates. Then
\(
\bigl(A(\Qcal_4^\vee)\cap\Zbb^2\bigr)\times\set{1}
\)
is the Hilbert basis, whose elements form the columns of the weight
matrix
\[ W = \begin{pmatrix}
-1 & -1 & -1 & 0 & 0 & 0 &  1 & 1 \\
-1 & 0 & 1 & -1 & 0 & 1 &  -1 & 0 \\
1 & 1 & 1 & 1 & 1 & 1 &  1 & 1 
\end{pmatrix}. \] 
The cone \( C_4 \) is the intersection of the following \( 14 \) quadrics in \( \mathbb{C}^8 \):
\begin{align*}
&z_{6}z_{7}-z_{5}z_{8},\,z_{5}z_{6}-z_{3}z_{8},\,z_{5}z_{7}-z_{4}z_{8},\,z_{3}z_{7}-z_{2}z_{8}, \\
&z_{4}z_{6}-z_{2}z_{8},\,z_{5}^{2}-z_{2}z_{8},\,z_{3}z_{5}-z_{2}z_{6},\,z_{2}z_{7}-z_{1}z_{8},\\
&z_{4}z_{5}-z_{1}z_{8},\,z_{2}z_{5}-z_{1}z_{6},\,z_{3}z_{4}-z_{1}z_{6},\,z_{4}^{2}-z_{1}z_{7},\\
&z_{2}z_{4}-z_{1}z_{5},\,z_{2}^{2}-z_{1}z_{3}. 
\end{align*}
The Hilbert series of \( (W,I) \) is then given by 
\begin{align*}
H_{W,I}(T_0, T_1, T_2) = &\frac{1+T_{2}+T_{1}^{-1}T_{2}+T_{0}^{-1}T_{2}-T_{0}T_{2}^{2}-T_{1}T_{2}^{2}-T_{0}T_{1}^{-1}T_{2}^{2} }{G(T)} \\ &\frac{-2\,T_{2}^{2}-T_{0}^{-1}T_{1}T_{2}^{2}-T_{1}^{-1}T_{2}^{2}-T_{0}^{-1}T_{2}^{2}+T_{0}T_{2}^{3}+T_{1}T_{2}^{3}+T_{2}^{3}+T_{2}^{4}}{G(T)},
\end{align*}
where 
\[ G(T) = \left(1-T_{0}^{-1}T_{1}^{-1}T_{2}\right)\left(1-T_{0}^{-1}T_{1}T_{2}\right)\left(1-T_{0}T_{1}^{-1}T_{2}\right)\left(1-T_{1}T_{2}\right)\left(1-T_{0}T_{2}\right). \]
We then obtain the volume of a Reeb field \( \xi = (a,b,c) \) as
\[ a_0(\xi) = \frac{-a^2 + 2 a b - b^2 + 2 a c + 2 b c + 
 7 c^2}{(a - b - c) (a + b - c) (a + c) (a - b + c) (b +
    c)}. \]
Under the normalized condition \( c = 3 \), we find that the minimized volume is 
\[a_0(\xi_0) = \frac{59 + 11 \sqrt{33}}{486}, \qquad a = b = \frac{-57 + 9 \sqrt{33}}{16}. \]
This agrees with \cite[equations (3.48) \& (3.49)]{MSY06} and \cite[equations (7.41) \& (7.42)]{MSY08}. 
\end{example} 

We conclude this section by collecting together the 
affine cones over the 
five toric del
Pezzo surfaces
\[
\Pbb^1\times\Pbb^1,\qquad
\Pbb^2,\qquad
\textnormal{Bl}_{p_1}(\Pbb^2),\qquad
\textnormal{Bl}_{p_1,p_2}(\Pbb^2),\qquad
\textnormal{Bl}_{p_1,p_2,p_3}(\Pbb^2),
\]
embedded using their anti-canonical bundles. The natural
regular Reeb fields associated with these cone structures generate
\(S^1\)-actions on their links, the quotients of which are precisely the respective
toric del Pezzo surfaces. Figure \ref{figure_delpezzo_family} displays
the corresponding toric diagrams and their dual polygons.

\begin{figure}[H] 
    \begin{subfigure}[b]{0.18\textwidth}
        \centering
        \resizebox{\linewidth}{!}{
                            \begin{tikzpicture}
                \draw[step=1cm,gray,dashed] (-1.9,-1.9) grid (1.9,1.9);
                                \node at (0,0) [circle, fill, inner sep = 1.5pt] {};

                \draw (1,0) -- (0,1) -- (-1,0) -- (0,-1) -- cycle;
            \end{tikzpicture}
        }
        \subcaption{\( \Qcal_1 \)}
        \label{fig:subfigdp0}
    \end{subfigure}
        \begin{subfigure}[b]{0.18\textwidth}
        \centering
        \resizebox{\linewidth}{!}{
                            \begin{tikzpicture}
                \draw[step=1cm,gray,dashed] (-1.9,-1.9) grid (1.9,1.9);
                                \node at (0,0) [circle, fill, inner sep = 1.5pt] {};

                \draw (1,0) -- (0,1) -- (-1,-1) -- cycle;
            \end{tikzpicture}
        }
        \caption{\( \Qcal_2 \)}
        \label{fig:subfigdpproj}
    \end{subfigure}
    \begin{subfigure}[b]{0.18\textwidth}
    \centering
        \resizebox{\linewidth}{!}{
                    \begin{tikzpicture}
                \draw[step=1cm,gray,dashed] (-1.9,-1.9) grid (1.9,1.9);
                \node at (0,0) [circle, fill, inner sep = 1.5pt] {};
                \draw (-1,-1) -- (-1,0) -- (0,1) -- (1,0) -- cycle;
                \end{tikzpicture}
        }
        \caption{\( \Qcal_3 \)}   
        \label{fig:subfigdp1}
    \end{subfigure}
    \begin{subfigure}[b]{0.18\textwidth}
        \centering
        \resizebox{\linewidth}{!}{
            \begin{tikzpicture}
                \draw[step=1cm,gray,dashed] (-1.9,-1.9) grid (1.9,1.9);
                                \node at (0,0) [circle, fill, inner sep = 1.5pt] {};

                \draw (0,-1) -- (1,0) -- (0,1) -- (-1,1) -- (-1,0) -- cycle;
            \end{tikzpicture}
        }
        \caption{\( \Qcal_4 \) }
        \label{fig:subfigdp2}
    \end{subfigure}
        \begin{subfigure}[b]{0.18\textwidth}
        \centering
        \resizebox{\linewidth}{!}{
                            \begin{tikzpicture}
                \draw[step=1cm,gray,dashed] (-1.9,-1.9) grid (1.9,1.9);
                                \node at (0,0) [circle, fill, inner sep = 1.5pt] {};

                \draw (-1,-1) -- (0,-1) -- (1,0) -- (1,1) -- (0,1) -- (-1,0) -- cycle;
            \end{tikzpicture}
        }
        \caption{\( \Qcal_5 \)}
        \label{fig:subfigdp3}
    \end{subfigure}    
    \begin{subfigure}[b]{0.18\textwidth}
        \centering
        \resizebox{\linewidth}{!}{
                            \begin{tikzpicture}
                \draw[step=1cm,gray,dashed] (-1.9,-1.9) grid (1.9,1.9);
                                \node at (0,0) [circle, fill, inner sep = 1.5pt] {};

                \draw (-1,1) -- (-1,-1) -- (1,-1) -- (1,1) -- cycle;
            \end{tikzpicture}
        }
        \caption{\( \Qcal_1^{\vee} \)}
        \label{fig:subfigdp00}     
    \end{subfigure}
            \begin{subfigure}[b]{0.18\textwidth}
        \centering
        \resizebox{\linewidth}{!}{
                            \begin{tikzpicture}
                \draw[step=1cm,gray,dashed] (-1.9,-1.9) grid (1.9,2.9);
                                \node at (0,0) [circle, fill, inner sep = 1.5pt] {};

                \draw (-1,2) -- (2,-1) -- (-1,-1) -- cycle;
            \end{tikzpicture}
        }
        \caption{\( \Qcal_2^{\vee} \)}
        \label{fig:subfigdpprojj}
    \end{subfigure}
    \begin{subfigure}[b]{0.18\textwidth}
        \centering
        \resizebox{\linewidth}{!}{
                            \begin{tikzpicture}
                \draw[step=1cm,gray,dashed] (-1.9,-1.9) grid (1.9,2.9);
                                \node at (0,0) [circle, fill, inner sep = 1.5pt] {};

                \draw (-1,2) -- (-1,-1) -- (1,-1) -- (1,0) -- cycle;
            \end{tikzpicture}
        }
        \caption{\( \Qcal_3^{\vee} \)}
        \label{fig:subfigdp11}
    \end{subfigure}
    \begin{subfigure}[b]{0.18\textwidth}
        \centering
        \resizebox{\linewidth}{!}{
                            \begin{tikzpicture}
                \draw[step=1cm,gray,dashed] (-1.9,-1.9) grid (1.9,1.9);
                                \node at (0,0) [circle, fill, inner sep = 1.5pt] {};

                \draw (-1,1) -- (-1,-1) -- (0,-1) -- (1,0) -- (1,1) -- cycle;
            \end{tikzpicture}
        } \caption{\( \Qcal_4^{\vee} \)}
        \label{fig:subfigdp22}
    \end{subfigure}         
        \begin{subfigure}[b]{0.18\textwidth}
        \centering
        \resizebox{\linewidth}{!}{
                            \begin{tikzpicture}
                \draw[step=1cm,gray,dashed] (-1.9,-1.9) grid (1.9,1.9);
                                \node at (0,0) [circle, fill, inner sep = 1.5pt] {};

                \draw (0,-1) -- (1,-1) -- (1,0) -- (0,1) -- (-1,1) -- (-1,0) -- cycle;
            \end{tikzpicture}
        }
        \caption{\( \Qcal_5^{\vee} \)}
        \label{fig:subfigdp33}
    \end{subfigure}
\caption{Polygons 
\protect(\subref{fig:subfigdp0})--\protect(\subref{fig:subfigdp3}) are the toric diagrams of the affine cones over the toric del Pezzo surfaces $\Pbb^1 \times \Pbb^1$, $\Pbb^2$, $\textnormal{Bl}_{p_1} (\Pbb^2)$, $\textnormal{Bl}_{p_1,\,p_2}(\Pbb^2)$, $\textnormal{Bl}_{p_1,p_2,p_3}(\Pbb^2)$, respectively, embedded using their anti-canonical bundle. Figures \protect(\subref{fig:subfigdp00})--(\protect\subref{fig:subfigdp33})
are the corresponding dual polygons given by $\Qcal_j^{\vee} = \set{u \in \Rbb^2 \mid \sprod{u,v} \geq -1, \forall v \in \Qcal_j}$.}
\label{figure_delpezzo_family}
\end{figure}

\section{Deformation theory of toric Calabi--Yau cones} \label{section_deformation}

In this section, we recall the main results in the deformation theory of isolated toric Gorenstein singularities developed by Altmann \cite{Alt97}. We begin with some background in Section \ref{prelim}, before introducing Minkowski decompositions and the Minkowski scheme of a toric diagram in Section \ref{s:versal}. Their relation with 
deformations of a toric Gorenstein cone and maximal Minkowski decompositions into lattice summands given by Altmann's theory is discussed in Section \ref{s:4.3}. 
Based on this theory, Section \ref{subsection_minkowski} explains how
to compute the Minkowski summand cone, its extremal summands, and its
maximal decompositions, and how this information can be combined with
the zero-sum partition criterion to determine the components of the
reduced versal base. An example of its implementation is given in Section
\ref{egsss}. All relevant computer code is contained in Appendix \ref{s:A1}.

\subsection{Deformation theory}\label{prelim} Recall that a \emph{deformation family of a complex scheme $X$ over a given base $S$} is defined as a triplet $(\Xcal, \phi, 0)$, where $\phi: \Xcal \to S$ is a flat morphism and $0$ is a marked point of $S$ with a given isomorphism $\phi^{-1}(0) \simeq X$. The set of equivalence classes of such deformation families is
denoted by \(\textnormal{Def}_X(S,0)\).
 A deformation $\phi:\Xcal \to S$ is called \emph{trivial} if it is equivalent to the projection family $\pi_2:X\times S \to S$; that is, there is a scheme isomorphism $\Xcal \to X \times S$ compatible with the maps $\pi_2$ and $\phi$. 

A deformation $\phi:\Xcal \to S$ is said to be \emph{complete} if every other deformation $(\wt{\phi}:\wt{\Xcal} \to T) \in \textnormal{Def}_X(T,0)$ is obtained from $\Xcal \to S$ by a base change via a morphism of algebraic germs $f: (T,0) \to (S,0)$, i.e., $\wt{\Xcal}$ is equivalent to the family $\pi_2: \Xcal \times_S T \to T $. 
%where set-theoretically
%\[ \Xcal \times_S T = \set{(x,t) \in \Xcal \times T\,|\, \phi(x) = f(t)}. \]
If in addition the morphism $f$ is uniquely determined on the level of tangent spaces, then the deformation $\phi:\Xcal \to S$ is called \emph{versal} (or \emph{semi-universal}). 
In this situation, the base space $S$ is said to be \emph{versal}. By the work of Artin \cite{Art76} and Elkik \cite{Elk73},
if $X$ is an affine variety with a finite set of singularities, then $X$ always admits a versal family of deformations, which is an affine family (meaning that the morphism $\phi$ is affine) over an affine base. If in addition a fiber $X_t = \phi^{-1}(t)$ ($t \neq 0$) is smooth, then $X_t$ is called an \emph{affine smoothing} of $X$.  

The infinitesimal deformations of $X$ are parametrized by the \emph{Kodaira--Spencer space}, defined as
\begin{equation*}
T^1_X :=  \textnormal{Def}_{X}(\Spec \Cbb[\varepsilon] / \varepsilon^2).
\end{equation*}
This is also the set of equivalence classes of first-order
deformations of \(X\) over the scheme
$\Spec \Cbb[\varepsilon] / \varepsilon^2$ (the ``fat point''). If $X$ has a versal deformation family over a versal base $S$ (in particular, when $X$ is an affine variety with a finite set of singularities), then by definition of versality, we know that
\[ \textnormal{Def}_{X}(\Spec \Cbb[\varepsilon]/ \varepsilon^2) \simeq \textnormal{Hom} (\Spec \Cbb[\varepsilon]/ \varepsilon^2 \to (S,0)), \]
where the latter is the set of scheme morphisms sending the fat point to the marked point $0$ of $S$, which is bijective onto the Zariski tangent space of $S$ at $0$ \cite[Chapter II, Exercise 2.8]{Har77}. Thus, in this situation, one can view $T^1_{X}$ as $T_0 S$---the \emph{Zariski tangent space at the marked point of the versal base $S$}.     

Following Hartshorne \cite[Chapter III, Exercise 9.10]{Har77}, we say that $X$ is \emph{rigid} if \( T^1_{X} = 0 \). Any versal deformation family of a rigid scheme is trivial, 
but there may be non-rigid schemes whose (reduced) versal base space is a point; see \cite[(9.2)]{Alt97} for an example, or Proposition \ref{prop:gmsw_trivial} for an infinite family of examples. In the toric case, we have the following rigidity statement of Altmann. Recall that a K\"ahler cone with a smooth link defines an isolated singularity (cf.~Theorem \ref{t:affine}).

\begin{theorem}[\cite{Alt97}] \label{thm:rigidity}
If \( C_0 \) is an isolated toric singularity, then $\dim_{\Cbb} T^1_{C_0} < +\infty$. If in addition $C_0$ is Gorenstein and $\dim_{\mathbb{C}}(C_0) \geq 4$, then $C_0$ is rigid.  
\end{theorem}
\noindent In particular, one can capture the essence of Altmann's theory in terms of lattice polytopes with \emph{real dimension at most two} 
in light of the rigidity theorem and the classification of toric Gorenstein cones with an isolated singularity; cf.~Proposition \ref{proposition_combinatorial_qgorenstein}.

In complex dimension two, a rational Gorenstein K\"ahler cone is precisely a quotient of $\mathbb{C}^{2}$ by a finite subgroup of $\text{SU}(2)$, known as a \emph{Du Val singularity}. The toric examples correspond to 
quotients by finite cyclic subgroups, whose deformation theory was studied by Arndt \cite{Arn88}, Christophersen \cite{Chr91}, Koll\'{a}r--Shepherd-Barron \cite{KSB88}, Riemenschneider \cite{Rie74}, and Stevens \cite{Ste91}, resulting in a detailed description of the reduced versal base space, including the number and dimension of their components (which are all smooth). On the other hand, two-dimensional complete Calabi--Yau metrics of maximal volume growth and $L^{2}$-integrable curvature were classified by Kronheimer using twistor methods \cite{Kro1, Kro2} and by the first author and Hein \cite{CH24} using non-twistor methods. Such metrics exist 
only on smoothings or crepant resolutions of 
Du Val singularities, the latter all being \textit{quasi-regular} Calabi--Yau cones. Thus, since our focus is on potentially non-rigid, irregular Calabi--Yau cones, we are primarily interested in three-dimensional toric Gorenstein K\"ahler cones, whose deformation theory can be effectively described using lattice polygons.

Recall from Definition \ref{def:toric_diagram} that a two-dimensional toric diagram is a convex lattice polygon whose edges are primitive lattice segments, or equivalently whose edges contain no relative interior lattice points.
Throughout this section, we fix a two-dimensional toric diagram $\Pcal$ and its associated three-dimensional isolated toric Gorenstein singularity $C_0$ (cf.~Proposition \ref{proposition_combinatorial_qgorenstein}). Unless otherwise stated, we work solely with toric diagrams in the sense of Definition \ref{def:toric_diagram}, though much of the discussion involving Minkowski summand cones and Minkowski decompositions also holds for more general polygons (cf. Remark \ref{rmk:toric_diagram} and Example \ref{ex:trapezoid}).

\subsection{Minkowski decompositions}\label{s:versal}

\subsubsection{Minkowski summand cone and decomposition}
Recall that a Minkowski decomposition of $\Pcal$ is a collection of
convex polytopes $(\Pcal_0,\dots,\Pcal_m)$ in $\Rbb^2$ such that
\begin{equation*} 
\Pcal = \Pcal_0 + \dots + \Pcal_m,
\end{equation*}
where the right-hand sum is defined by $$\set{p_0 + \dots + p_m\,|\, p_j \in \Pcal_j,\quad 0 \leq j \leq m}.$$ 
A polytope $\Pcal_k$ appearing in such a decomposition is called a
\emph{summand} of $\Pcal$. Strictly speaking, the summands in a fixed
Minkowski decomposition are defined up to translation. Indeed, if
$v_0,\dots,v_m\in\Rbb^2$ satisfy
\[
v_0+\dots+v_m=0,
\]
then
\[
\Pcal
=
(\Pcal_0+v_0)+\dots+(\Pcal_m+v_m).
\]
We will use the phrase \emph{up to homothety} in what follows when discussing the
rays of the Minkowski summand cone. Here, two polytopes are called
homothetic if one is obtained from the other by a positive scaling and
a translation. The summands in a Minkowski decomposition are not required to be lattice
polytopes. We say that the decomposition is a \emph{lattice Minkowski
decomposition} if every summand is a lattice polytope.
A real Minkowski decomposition is  called \emph{trivial}
if every positive-dimensional summand is homothetic to the polytope
being decomposed, whereas
a lattice Minkowski decomposition is said to be \emph{non-trivial} if it contains at least two possibly homothetic
positive-dimensional lattice summands.

The Minkowski decomposition is said to be \emph{maximal} if each
\(\Pcal_k\) does not split into any non-trivial Minkowski summands. Equivalently, $\Pcal_k$ is either a segment or a triangle (cf. Corollary \ref{cor:minkalways}). Following \cite[Section (2.5)]{Alt97}, a lattice Minkowski
decomposition is called a \emph{maximal Minkowski decomposition into
lattice summands} if it admits no proper refinement into lattice
Minkowski decompositions. Equivalently, none of its summands admits a
further non-trivial lattice Minkowski decomposition. As usual, point
summands are omitted, since they only translate the other summands.

%\textcolor{red}{what you define here seems to be ``maximal lattice decomposition'' while Altmann's ``maximal decomposition into lattice summands'' is a maximal decomposition whose summands are all lattice. To see the difference: any quadrilateral admits a maximal decomposition which is not necessarily lattice. However, only parallelogram has a maximal decomposition which is lattice. So maximal decomposition here means decomposition into segments and triangles. To put the adjective ``lattice'' before ``maximal'' means that the segments and triangles are lattice.}

These two notions should not be confused. In the first, maximality is
taken among all Minkowski decompositions, whereas in the second it is
taken only among lattice Minkowski decompositions (see also Figure \ref{figure_latticemaxvsmaxlattice}). The latter is the
notion that indexes the irreducible components of the reduced
Minkowski scheme in Theorem \ref{thm:minkschemecomp}.

\medskip
\noindent\textbf{Terminology.}
Throughout the paper, a \emph{lattice segment} means a segment whose
endpoints are lattice points, and a \emph{lattice triangle} means a
triangle whose vertices are lattice points. These terms do not by
themselves imply that the segment is primitive or that the triangle is
unimodular. A lattice segment is called \emph{primitive} if it contains
no lattice points other than its endpoints. A lattice triangle is called
\emph{unimodular} if it is affine lattice-equivalent to
\[
\textnormal{Conv}
\set{\Spvek{0;0},\Spvek{1;0},\Spvek{0;1}},
\]
or equivalently if its normalized lattice area is one.

A maximal Minkowski decomposition into lattice summands determines an
irreducible component of the reduced versal base. In complex dimension
three, the corresponding component is a smoothing component if and
only if every summand is either a primitive lattice segment or a unimodular
lattice triangle; see Lemma \ref{lemma_smoothing_component}.

\medskip

Let $N$ be the number of vertices of $\Pcal$, ordered
counter-clockwise as \(v_1,\dots,v_N\), and set \(v_{N+1}:=v_1\).
Let
\[
d_j:=v_{j+1}-v_j\in\Rbb^2,
\qquad
1\leq j\leq N,
\]
be the corresponding oriented edges of \(\Pcal\). Then
\[
d_1+\dots+d_N=0.
\]
Unless stated otherwise, we always assume this counter-clockwise
orientation of \(d_1,\dots,d_N\), starting from a given vertex. Next, associate to $\Pcal$ a vector space \( V(\Pcal) \subset \mathbb{R}^N \) defined by
\begin{equation*} 
V(\Pcal) := \set{(t_1, \dots, t_N)\,|\,  \sum_{i=1}^N t_i d_i = 0}. 
\end{equation*}
Since \(\Pcal\) is a full-dimensional polygon, its oriented edge vectors span \(\Rbb^2\). Consequently, \(V(\Pcal)\) is a real \((N-2)\)-dimensional vector space.

\begin{lemma}[{\cite[Lemma (2.2)]{Alt97}}]
\label{lemma_minkowski_cone}
The cone
\(
C(\Pcal):=V(\Pcal)\cap\Rbb_{\geq0}^N
\)
is a rational polyhedral cone in \(V(\Pcal)\). Its points are in
one-to-one correspondence with the translation classes of Minkowski
summands of positive multiples of \(\Pcal\), and its integral points
correspond to lattice Minkowski summands. In particular, the primitive
generators of the extremal rays of \(C(\Pcal)\) correspond to
homothety classes of extremal Minkowski summands.
\end{lemma}

\begin{proof}[Sketch of proof]
The correspondence is roughly given as follows. If
\(t=(t_1,\dots,t_N)\in C(\Pcal)\), set \(x_0:=0\) and
\[
x_j:=\sum_{i=1}^j t_i d_i,
\qquad
1\leq j\leq N.
\]
Since \(t\in V(\Pcal)\), we have
\[
x_N=\sum_{i=1}^N t_i d_i=0=x_0.
\]
Thus, the corresponding summand \(\Pcal_t\) is obtained, up to
translation, by starting at the origin and going successively along
the edges \(t_i d_i\). More precisely,
\(
\Pcal_t
=
\textnormal{Conv}\set{x_0,\dots,x_{N-1}}.
\)

Conversely, a Minkowski summand whose oriented edges are parallel to
the corresponding edges of \(\Pcal\) can, after translation, be
obtained by starting at the origin and walking successively along
\(t_i d_i\), for some \(t_i\geq0\). Since the resulting path is closed,
\[
\sum_{i=1}^N t_i d_i=0,
\]
and the summand is sent to \(t=(t_i)\in C(\Pcal)\).
\end{proof}

We call $C(\Pcal)$ the \emph{Minkowski summand cone} and denote by
$\rho$ the map from the set of Minkowski summands, considered up to
translation, to $C(\Pcal)$. If the summand $\Pcal_t$ corresponds to
$t\in C(\Pcal)$, then $\lambda\Pcal_t$ corresponds to
$\lambda t\in C(\Pcal)$ for every $\lambda>0$. Consequently, points
lying on the same ray of $C(\Pcal)$ correspond to homothetic summands.

In this sense, the minimal set of primitive generators of $C(\Pcal)$
corresponds to the homothety classes of the extremal Minkowski summands. The minimal set of primitive generators
depends on the orientation of the edges, but
$\dim_{\mathbb{R}}C(\Pcal)$ is independent of any orientation. We
highlight these notions with some examples.

\begin{example}
We have that \( \rho(\Pcal) = (1, \dots, 1) \) and \( \rho(t\Pcal) = (t, \dots, t) \). 
\end{example}

\begin{example} \label{ex:segmenttriangle}
Lattice segments and lattice triangles have no non-trivial real Minkowski decomposition. In both cases, $\dim_{\Rbb} V(\Pcal) = 1$ and $C(\Pcal)$ has only one primitive generator corresponding to the segment/triangle as the unique Minkowski summand (up to homothety) of itself. 
\end{example}

\noindent The existence of a Minkowski decomposition is given by the following corollary.

\begin{corollary}\label{cor:minkalways}
A toric diagram \(\Pcal\) always admits a Minkowski decomposition which is
non-trivial if \(N>3\). Moreover, the decomposition is maximal if and only
if it involves only segment and triangle summands. 
\end{corollary}

\begin{proof}
If $\Pcal$ is a segment or a triangle, then it only has a trivial Minkowski decomposition by Example \ref{ex:segmenttriangle}. If $\Pcal$ has more than three vertices, then $\dim_{\R} C(\Pcal) \geq 2$, hence $\Pcal$ has at least one non-trivial Minkowski summand. Clearly, the decomposition is maximal if it involves only segments and triangles. Conversely, if a summand has more than three vertices, then it has a non-trivial decomposition by the same argument as before. 
\end{proof}

We next compute Minkowski decompositions in some examples.  

\begin{example} \label{ex:dp1mink}
 The oriented edges of the reflexive polygon $\Qcal_1$ in Figure \ref{figure_delpezzo_family} (starting from $(0,-1)$) are $d_1 = (1,1), d_2 = (-1,1), d_3 = (-1,-1), d_4 = (1,-1)$. The vector space 
 \[ V(\Qcal_1) = \set{(t_1,\dots,t_4) \in \Rbb^4\,|\, t_1 - t_2 - t_3 + t_4 = 0,\quad t_1 + t_2 - t_3  - t_4 = 0}\]
has basis $\{u_1 = (1,0,1,0), u_2 = (0,1,0,1)\}$. These are also the generators of the Minkowski summand cone $C(\Qcal_1)$, therefore 
the extremal Minkowski summands of \(\Qcal_1\) up to homothety are
\[\Lcal_1 = \textnormal{Conv} \set{\Spvek{0;0},\Spvek{1;1}}, \qquad \Lcal_2 = \textnormal{Conv} \set{\Spvek{0;0},\Spvek{-1;1}}. \]
Here, $\mathcal{L}_1$ is obtained by going forward along $d_1$ and returning to the origin along $d_3$, and $\mathcal{L}_2$ is obtained likewise from $d_2$ and $d_4$, respectively. In particular, $\rho(\Lcal_1) = (1,0,1,0)$ and $\rho(\Lcal_2) = (0,1,0,1)$. The unique (non-trivial) lattice decomposition is, up to translation, 
given by
\[
\Qcal_1=\Lcal_1+\Lcal_2;
\]
see Figure \ref{subfig:dp1mink}. This decomposition corresponds to a
partition of \(\set{d_1,\dots,d_4}\) into two disjoint subsets
\(\set{d_1,d_3}\cup\set{d_2,d_4}\).
\end{example}

\begin{figure}[H] 
    \begin{subfigure}[b]{0.45\textwidth}
    \centering
        \resizebox{\linewidth}{!}{
                    \begin{tikzpicture}[ thick,decoration={markings,mark=at position 0.5 with {\arrow{Stealth}}}]
                \draw[step=1cm,gray,dashed] (-1.9,-1.9) grid (5.9,1.9);
                \node at (0,0) [circle, fill, inner sep = 1.5pt] {};
                \draw[postaction={decorate}] (1,0) -- (0,1);
                \draw[postaction={decorate}] (0,1) -- (-1,0);
                \draw[postaction={decorate}] (-1,0)-- (0,-1);
                \draw[postaction={decorate}] (0,-1) -- (1,0);
                \node at (1.5,0) {\huge \textbf{=}};
                \node at (2,0) [circle, fill, inner sep = 1.5pt]{};
                \draw[postaction={decorate}] (2,0) -- (3,1);
                \draw[postaction={decorate}] (3,1) -- (2,0);
                \node at (3,0) {\huge \textbf{+}};
                \node at (4,0) [circle, fill, inner sep = 1.5pt]{};
                \draw[postaction={decorate}] (4,0) -- (5,-1);
                \draw[postaction={decorate}] (5,-1)--(4,0);
                \end{tikzpicture}
        }
        \caption{\( \Qcal_1 \)}  
        \label{subfig:dp1mink}
    \end{subfigure}
    \begin{subfigure}[b]{0.45\textwidth}
        \centering
        \resizebox{\linewidth}{!}{
                            \begin{tikzpicture}[ thick,decoration={markings,mark=at position 0.5 with {\arrow{Stealth}}}]
                \draw[step=1cm,gray,dashed] (-1.9,-1.9) grid (6.9,1.9);
                                \node at (0,0) [circle, fill, inner sep = 1.5pt] {};

                \draw[postaction={decorate}] (0,-1) -- (1,0); 
                \draw[postaction={decorate}] (1,0) -- (0,1);
                \draw[postaction={decorate}] (0,1) -- (-1,1);
                \draw[postaction={decorate}] (-1,1) -- (-1,0);
                \draw[postaction={decorate}] (-1,0) -- (0,-1);
                \node at (1.5,0) {\huge \textbf{=}}; 
                \node at (3,0) [circle, fill, inner sep = 1.5pt]{};
                \draw[postaction={decorate}] (3,0) -- (2,1);
                \draw[postaction={decorate}] (2,1) -- (3,0);
                \node at (4,0) {\huge \textbf{+}};
                \node at (5,0) [circle, fill, inner sep = 1.5pt]{};
                \draw[postaction={decorate}] (5,0) -- (6,1);
                \draw[postaction={decorate}] (6,1) -- (5,1);
                \draw[postaction={decorate}] (5,1) -- (5,0);
            \end{tikzpicture}
        }
        \caption{\( \Qcal_4 \)}
        \label{subfig:dp3mink}
    \end{subfigure}

\caption{Lattice Minkowski decompositions for the toric diagrams $\Qcal_1$ and $\Qcal_4$. Each solid circle is the origin of the relative vector space.}
\label{figure_dp_family_decomposition}
\end{figure}

\begin{example} \label{ex:trapezoid}
Consider the trapezoid 
\begin{equation*}
 \Qcal = \textnormal{Conv} \set{\Spvek{0;0},\Spvek{3;0},\Spvek{2;1},\Spvek{1;1}},   
\end{equation*}
where the oriented edges are $d_1 = (3,0), d_2 = (-1,1), d_3= (-1,0), d_4 = (-1,-1)$. A basis of \(V(\Qcal)\) is given by
\[
\set{
u_1=(1,0,3,0),\quad
u_2=(0,1,-2,1)
}.
\]
The primitive generators of \(C(\Qcal)\) are
\[
u_1=(1,0,3,0)
\qquad\textnormal{and}\qquad
\widetilde{u}_2=2u_1+3u_2=(2,3,0,3).
\]
The first polytope \(\Lcal\) corresponding to \(u_1\) is obtained by
first moving along \(d_1\), then taking three steps along \(d_3\). The
second primitive generator corresponds to \(3\Delta\), which is obtained
by first taking two steps along \(d_1\), then three steps along \(d_2\)
and finally three steps along \(d_4\). Hence, we have that
\[
\Lcal
=
\textnormal{Conv}
\set{\Spvek{0;0},\Spvek{3;0}},
\qquad
\Delta
=
\textnormal{Conv}
\set{\Spvek{0;0},\Spvek{1;1},\Spvek{2;0}}.
\]
Here, \(\Lcal\) corresponds to the primitive generator \(u_1\) of
\(C(\Qcal)\), whereas the coefficient \(\frac{1}{3}\) of \(u_1\) in
the expression for \(\rho(\Qcal)\) determines the scale with which
\(\Lcal\) occurs in the decomposition.
Since
\[
(1,1,1,1)
=
\frac{1}{3}u_1+\frac{1}{3}\widetilde{u}_2,
\]
the unique non-trivial lattice Minkowski decomposition of \(\Qcal\) is given by
\[
\Qcal=\frac{1}{3}\Lcal+\Delta.
\]
Note that $\Qcal$ does not define a toric diagram of an isolated toric Gorenstein singularity because $d_1$ contains interior lattice points. 
\end{example}

\begin{remark}\label{rmk:toricdiagramminkdec}
Examples \ref{ex:dp1mink} and \ref{ex:trapezoid} illustrate the
following fact: \emph{for the toric diagram \(\Pcal\) of an isolated
toric Gorenstein singularity, a lattice Minkowski decomposition of
\(\Pcal\) into \(m+1\) summands corresponds to a partition of the set
\(\set{d_1,\dots,d_N}\) into $m+1$ disjoint subsets, each of which
sums to zero}. Indeed, suppose that
\(
\Pcal=\Pcal_0+\dots+\Pcal_m
\)
is a lattice Minkowski decomposition and write
\(
\rho(\Pcal_k)
=
(t_{k1},\dots,t_{kN}).
\)
Since
\[
\rho(\Pcal)
=
\rho(\Pcal_0)+\dots+\rho(\Pcal_m)
=
(1,\dots,1),
\]
we have
\(
0\leq t_{ki}\leq1.
\)
Moreover, since \(\Pcal_k\) is a lattice polytope and each \(d_i\) is
primitive, \(t_{ki}d_i\) is a lattice vector, and hence
\(t_{ki}\in\Zbb\). It follows that
\(
t_{ki}\in\set{0,1}.
\)
Thus, the subsets
\(
I_k
:=
\set{
i\in\set{1,\dots,N}
\mid
t_{ki}=1
}
\)
form a partition of \(\set{1,\dots,N}\). Since
\(\rho(\Pcal_k)\in V(\Pcal)\), we also have that
\[
\sum_{i\in I_k}d_i=0.
\]

Conversely, given a subset
\(I\subset\set{1,\dots,N}\) of such a partition satisfying
\[
\sum_{i\in I}d_i=0,
\]
then, after a possible reordering, the summand polytope is obtained by
moving exactly one step along the oriented edges in their respective
order. In particular, if
\(
\Pcal=\Pcal_0+\dots+\Pcal_m
\)
is a lattice decomposition, then each \(\rho(\Pcal_k)\) comprises only
\(0\)'s and \(1\)'s. More precisely,
\(
\rho(\Pcal_k)
=
(\delta_{I_k}^1,\dots,\delta_{I_k}^N),
\)
where
\[
\delta_{I_k}^i
=
\begin{cases}
1, & i\in I_k,\\
0, & i\notin I_k.
\end{cases}
\]
Moreover, let \(v(\Pcal)\) denote the number of vertices of \(\Pcal\).
Then
\begin{equation}\label{eq:verticesconstraint}
v(\Pcal)
=
N
=
\sum_{j=0}^m v(\Pcal_j).
\end{equation}
As demonstrated by Example \ref{ex:trapezoid}, this relation does not
hold if \(\Pcal\) is not a good toric diagram.
\end{remark}

The following lemma gives a practical criterion to determine whether or not a given toric diagram has a lattice Minkowski decomposition. 

\begin{lemma} \label{lem:practical-crit}
Let $\Pcal\subset\R^{2}$ be a toric diagram. If $\Pcal$ has a pair of parallel oriented edges (necessarily of equal length), then $\Pcal$ has a non-trivial lattice Minkowski decomposition. In particular, a toric diagram given by a lattice quadrilateral has a non-trivial lattice Minkowski decomposition if and only if it is a parallelogram. 
\end{lemma}

\begin{proof}
Any pair of parallel edges on a toric diagram $\Pcal\subset\mathbb{R}^{2}$ must be of equal length, for otherwise one edge would contain interior lattice points. 
Let $d_i, d_j$ be such a pair of edges. Then under our orientation convention, $d_i + d_j = 0$, and so we can partition the oriented edge set of $\Pcal$ into disjoint subsets $\set{d_i, d_j}$ and $\{d_k\,|\, k \neq i, j\}$, each of which sums up to zero. It follows that $\Pcal$ has a lattice Minkowski decomposition. In particular, a parallelogram has such a decomposition. Conversely, if $\Pcal$ is a quadrilateral, then by the vertices constraint \eqref{eq:verticesconstraint}, the only way to split it into lattice polygons is by summing two lattice segments, in which case $\Pcal$ must be a parallelogram.
\end{proof}

By summing two polytopes, we can also generate toric diagrams that admit a lattice Minkowski decomposition.

\begin{lemma} \label{lem:sumtoricdiag}
Let $\Qcal$ be a toric diagram in $\Rbb^2$ and $\Lcal$ a lattice segment in $\Rbb^2$. Then $\Pcal = \Qcal + \Lcal$ is a toric diagram if and only if $\Lcal$ is not parallel to any edge of $\Qcal$ and $\Lcal$ contains no interior lattice points.
\end{lemma}

\begin{proof}
The Minkowski sum of two convex affine lattice polytopes is clearly again a convex affine lattice polytope, hence $\Pcal$ satisfies Definition \ref{def:toric_diagram}\ref{lattice} as a polygon. Since we are in dimension two, it remains to verify condition \ref{def:toric_diagram}\ref{properface}, namely that every edge of \(\Pcal\) is a primitive lattice segment.
This follows from the fact that the edges of $\Pcal$ are the union of the edges of $\Qcal$ and two translates of $\Lcal$.  

Conversely, suppose that $\Pcal$ is a toric diagram with a lattice decomposition $\Qcal + \Lcal$. Then by Remark \ref{rmk:toricdiagramminkdec}, $\Pcal$ has an edge parallel to $\Lcal$ and of equal length, and so $\Lcal$ does not contain any interior lattice points. Furthermore, $\Lcal$ is not parallel to an edge of $\Qcal$, for otherwise the sum would give an edge with at least one interior lattice point.  
\end{proof}

A lemma we need is:

\begin{lemma}\label{cor:sumtoricdiag}
Let \(\Qcal\) be a quadrilateral toric diagram with oriented edges
\(d_1,d_2,d_3,d_4\) which is not a parallelogram, and let \(\Lcal\)
be a lattice segment such that \(\Qcal+\Lcal\) is a toric diagram.
Then, in addition to the maximal Minkowski decomposition
into lattice summands
\(
\Qcal+\Lcal,
\)
the toric diagram \(\Qcal+\Lcal\) has a maximal Minkowski
decomposition into two lattice triangles if and only if, given the
suitably oriented edge \(d\) corresponding to \(\Lcal\), either:
\begin{enumerate}
\item\label{12} \(d_1+d_2=d\), or
\item\label{13} \(d_1+d_3=-d\), or
\item\label{14} \(d_1+d_4=d\).
\end{enumerate}
In all cases, the additional decomposition is the sum of two lattice
triangles, which are not necessarily unimodular.
\end{lemma}

\begin{proof}
Let \(\Qcal\) be the quadrilateral as given in the statement of the
lemma. Since \(\Qcal\) is not a parallelogram, it has no non-trivial
lattice Minkowski decomposition by Lemma \ref{lem:practical-crit}.
Furthermore, since \(\Qcal+\Lcal\) is a toric diagram, Lemma
\ref{lem:sumtoricdiag} implies that \(\Lcal\) is a primitive lattice
segment. Thus, the decomposition
\(
\Qcal+\Lcal
\)
is a maximal Minkowski decomposition into lattice summands.

Recall that, since \(d_1,d_2,d_3,d_4\) are the oriented edges of
\(\Qcal\), we have
\(
d_1+d_2+d_3+d_4=0.
\)
Then \ref{12} implies that there is a partition of the edges of
\(\Qcal+\Lcal\) into two disjoint sets, namely
\(
\set{d_1,d_2,-d}\) and
$\set{d_3,d_4,d}$,
each summing up to zero, because
\(
d_1+d_2=-(d_3+d_4)=d.
\)
Furthermore, given that \(\Qcal\) is not a parallelogram and
\(\Lcal\) is not parallel to any edge of \(\Qcal\) by Lemma
\ref{lem:sumtoricdiag}, there are no parallel edges in either set.
Consequently, each set forms a triangle by walking along the
corresponding edges. Thus, in addition to the decomposition,
\(\Qcal+\Lcal\) has a maximal Minkowski decomposition into two lattice
triangles. By the same principle, the conclusion holds if we assume
either \ref{13} or \ref{14}.

Conversely, suppose that \(\Qcal+\Lcal\) has a maximal Minkowski
decomposition into lattice summands different from the
decomposition \(\Qcal+\Lcal\). Then, by the vertices constraint
\eqref{eq:verticesconstraint}, \(\Qcal+\Lcal\) is a hexagon. The only decomposition into a quadrilateral and a segment is the
decomposition \(\Qcal+\Lcal\). Indeed, by Remark
\ref{rmk:toricdiagramminkdec}, the segment summand must arise from a
pair of opposite oriented edges. Since \(\Qcal\) is not a
parallelogram, its oriented edges contain no opposite pair, whereas
\(\Lcal\) is not parallel to any edge of \(\Qcal\). Thus, the only
opposite pair among the six oriented edges of \(\Qcal+\Lcal\) is the
pair \(d,-d\) arising from \(\Lcal\).
It follows that any different maximal Minkowski decomposition into
lattice summands must be the sum of three segments or two triangles.
In the former case, the four oriented edges of \(\Qcal\) must form two
opposite pairs, and hence \(\Qcal\) is a parallelogram. This is
excluded by our assumption. It follows that there is a partition of
the edges of \(\Qcal+\Lcal\) into two disjoint sets of three edges,
each summing up to zero. The only possibilities for this are
\ref{12}--\ref{14}, for a suitably chosen orientation of \(d\).
\end{proof}

We illustrate this lemma with a simple example.
\begin{example}\label{ex:latticemaxvsmaxlattice}
Consider the toric diagram \(\Qcal_3\). Then both
\(
\Pcal_1
:=
\Qcal_3
+
\textnormal{Conv}
\set{\Spvek{0;0},\Spvek{1;0}}
\)
and
\(
\Pcal_2
:=
\Qcal_3
+
\textnormal{Conv}
\set{\Spvek{0;0},\Spvek{1;3}}
\)
have maximal Minkowski decompositions into lattice summands given in their definitions. In addition, \(\Pcal_1\) has a
maximal Minkowski decomposition into two lattice triangles, whereas
\(\Pcal_2\) has no such additional decomposition, even though both
define toric diagrams; see Figure
\ref{figure_latticemaxvsmaxlattice}. Indeed, as explained in the proof of Lemma
\ref{cor:sumtoricdiag}, if a maximal Minkowski decomposition of
\(\Pcal_2\) into lattice summands different from the given
decomposition exists, then it must be the sum of two lattice triangles.
However, one can readily verify from the associated toric diagram that
no three oriented edges of \(\Pcal_2\) sum to zero.
\end{example}

\begin{figure} [!htbp]
        \centering
        \resizebox{\linewidth}{!}{
                            \begin{tikzpicture}[ thick,decoration={markings,mark=at position 0.5 with {\arrow{Stealth}}}]
                \draw[step=1cm,gray,dashed] (-1.9,-6.9) grid (8.9,1.9);
                                \node at (0,0) [circle, fill, inner sep = 1.5pt] {};
                \draw[postaction={decorate}] (0,-1) -- (2,0);
                \draw[postaction={decorate}] (2,0) -- (1,1);
                \draw[postaction={decorate}] (1,1) -- (0,1);
                \draw[postaction={decorate}] (0,1) -- (-1,0);
                \draw[postaction={decorate}] (-1,0)-- (-1,-1);
                \draw[postaction={decorate}] (-1,-1) -- (0,-1);
            \node at (2.5,0) {\huge \textbf{=}};   
            \node at (4,0) [circle, fill, inner sep = 1.5pt]{};
            \draw[postaction={decorate}] (3,-1) -- (5,0);
            \draw[postaction={decorate}] (5,0) -- (4,1);
            \draw[postaction={decorate}] (4,1) -- (3,0);
            \draw[postaction={decorate}] (3,0) -- (3,-1);
             \node at (6,0) {\huge \textbf{+}};
             \node at (7,0) [circle, fill, inner sep = 1.5pt]{};
             \draw[postaction={decorate}] (7,0) -- (8,0);
             \draw[postaction={decorate}] (8,0) -- (7,0);
            
            \node at (0,-5) [circle, fill, inner sep = 1.5pt] {};
            \draw[postaction={decorate}] (-1,-6) -- (1,-5);
            \draw[postaction={decorate}] (1,-5) -- (2,-2);
            \draw[postaction={decorate}] (2,-2) -- (1,-1);
            \draw[postaction={decorate}] (1,-1) -- (0,-2);
            \draw[postaction={decorate}] (0,-2) -- (-1,-5);
            \draw[postaction={decorate}] (-1,-5) -- (-1,-6);
             \node at (2.5, -4) {\huge \textbf{=}};
              \node at (4,-4) [circle, fill, inner sep = 1.5pt]{};
            \draw[postaction={decorate}] (3,-5) -- (5,-4);
            \draw[postaction={decorate}] (5,-4) -- (4,-3);
            \draw[postaction={decorate}] (4,-3) -- (3,-4);
            \draw[postaction={decorate}] (3,-4) -- (3,-5);
             \node at (6,-4) {\huge \textbf{+}};
              \node at (7,-5) [circle, fill, inner sep = 1.5pt]{};
             \draw[postaction={decorate}] (7,-5) -- (8,-2);
             \draw[postaction={decorate}] (8,-2) -- (7,-5);
            \end{tikzpicture}
        }
\caption{The toric diagrams of
\(\Pcal_1
=
\Qcal_3+\textnormal{Conv}
\set{\Spvek{0;0},\Spvek{1;0}}\)
and
\(\Pcal_2
=
\Qcal_3+\textnormal{Conv}
\set{\Spvek{0;0},\Spvek{1;3}}\),
respectively. Both decompositions are maximal Minkowski
decompositions into lattice summands. The former also admits an
additional maximal Minkowski decomposition into two lattice triangles,
whereas the latter admits no such additional decomposition.}
\label{figure_latticemaxvsmaxlattice}
\end{figure}

In the next example, we collect together the lattice Minkowski decompositions of the del Pezzo toric diagrams given in the above figures.
\begin{example} \label{example_toricdelpezzo_minkowski_decomposition}
Recall the del Pezzo toric diagrams $\Qcal_1, \Qcal_2, \Qcal_3, \Qcal_4, \Qcal_5$ represented in Figure \ref{figure_delpezzo_family} (see also Figure \ref{figure_dp_family_decomposition}).
\begin{itemize}
\item $\Qcal_1$ admits only one non-trivial lattice Minkowski decomposition (cf. ~Example \ref{ex:dp1mink} and Figure \ref{subfig:dp1mink}). 
\item $\Qcal_2$ is a triangle, hence does not admit any non-trivial Minkowski decomposition. 
\item $\Qcal_3$ has no parallel edges, and so has no non-trivial lattice decomposition by Lemma \ref{lem:practical-crit}.
\item $\Qcal_4$ has one pair of parallel edges of equal length and the remaining edges form a lattice triangle. $\Qcal_4$ therefore has (up to homothety) at least two Minkowski summands, namely
\[\Lcal = \textnormal{Conv} \set{\Spvek{0;0},\Spvek{-1;1}}, \qquad \Delta = \textnormal{Conv} \set{\Spvek{0;0},\Spvek{0;1},\Spvek{1;1}}.\]
While the real summand cone $C(\mathcal{Q}_4)$ is three-dimensional, applying the binary partition constraints from Remark \ref{rmk:toricdiagramminkdec} restricts us to exactly two
nonzero proper lattice
Minkowski summands up to translation, namely $\mathcal{L}$ and $\Delta$. The only non-trivial lattice Minkowski decomposition of $\Qcal_4$ is therefore given by $\Qcal_4 = \Lcal + \Delta$ (cf.~Figure \ref{subfig:dp3mink}). 
\item $\Qcal_5$ has precisely two maximal Minkowski decompositions into lattice summands:~into three segments, and into two triangles. This example is considered throughout \cite{Alt97}. 
\end{itemize}
\end{example}

\subsubsection{Minkowski schemes}
For brevity, write \(V:=V(\Pcal)\). For each \( k \geq 1 \), we
define an $\Rbb^2$-valued polynomial by
\(
g_k(t):=\sum_{i=1}^N t_i^k d_i.
\)
The ideal 
\[ \Ical := \sprod{ g_k(t), \; k \geq 1} \subset \mathbb{C}[t_1,\dots,t_N] \]
%(where the two coordinate functions of each \(g_k\) are understood as
%generators)
determines a complex affine closed subscheme 
\begin{equation} 
\Mcal := \Spec \mathbb{C}[t_1,\dots,t_N] / \Ical \subset V_{\mathbb{C}} = V\otimes_{\Rbb} \Cbb.
\end{equation}
\noindent The following criterion is useful in practice. 
\begin{lemma}[{\cite[Sections (2.3) \& (3.3)]{Alt97}}]
If \(\Pcal\) is contained in two strips with non-parallel boundary directions, each bounded by a
pair of parallel lattice lines of lattice distance at most \(k_0\),
then all the polynomials \(g_k\) with \(k>k_0\) are superfluous. 
\end{lemma}
\noindent We illustrate an application of this in the following example.

\begin{example}
Since $\Qcal_1$ is clearly contained between two pairs of lines of lattice distance $2$, the scheme $\Mcal$ is defined by 
\[ \Ical = \sprod{t_1 - t_2 - t_3 + t_4, t_1 + t_2 - t_3 - t_4, t_1^2 - t_2^2 - t_3^2 + t_4^2, \; t_1^2 + t_2^2 - t_3^2 - t_4^2}.\]
A Gröbner basis of $\Ical$ (computed using Mathematica) is then given by $\set{t_2 - t_4, t_1 - t_3}$. It follows that $\Mcal$ is topologically the complex affine space $\set{(u,v,u,v)\,|\, u,v \in \Cbb} \simeq \Cbb^2$. 
\end{example}

We denote by $l$ the projection map $\Cbb^N \to \Cbb^N / \Cbb \rho(\Pcal) = \Cbb^N / \Cbb(1,\dots,1)$, and by the same notation its restriction to $V_{\Cbb}$:  
\begin{equation*} l: V_{\mathbb{C}} \to V_{\mathbb{C}} / \mathbb{C} \rho(\Pcal). 
\end{equation*}
We have the following definition.

\begin{definition} \label{def:minkscheme}
The scheme-theoretic image \( \ol{\Mcal} := l(\Mcal) \) is called the \emph{Minkowski scheme} of $\Pcal$ or (as we shall see in Theorem \ref{thm:versalfamily} below) the \emph{versal base space} of the associated toric cone $C_0$.      
\end{definition}
\noindent In general, \( \Mcal \) and \( \ol{\Mcal} \) are neither reduced nor irreducible. 
We denote by $\Mcal^{\textnormal{red}}, \ol{\Mcal}^{\textnormal{red}}$ the reduced scheme structures of \( \Mcal, \ol{\Mcal} \), respectively. Then we have:

\begin{theorem}[{\cite[Sections (2.5) \& (3.5)]{Alt97}}] \label{thm:minkschemecomp}
The set of irreducible components of $\Mcal^{\textnormal{red}}$ consists of complex affine spaces and is in bijection with 
the maximal decompositions of \( \Pcal \) into lattice summands. $\ol{\Mcal}^{\textnormal{red}}$ is the union of the corresponding quotient spaces under $l$. In particular, if the decomposition has $m+1$ summands, then the corresponding irreducible affine components of $\Mcal^{\textnormal{red}}$ and $\ol{\Mcal}^{\textnormal{red}}$ are isomorphic as toric affine varieties to $\Cbb^{m+1}$ and $\Cbb^{m}$, respectively. 
\end{theorem}

Let $\Pcal = \Pcal_0 + \dots + \Pcal_m$ be a lattice Minkowski decomposition. As discussed in Remark \ref{rmk:toricdiagramminkdec}, this corresponds to a partition of the oriented edge set $\set{d_1, \dots, d_N}$ into disjoint subsets $I_0, \dots, I_m$. The summand $\Pcal_k$ is sent by $\rho$ to $(t_1, \dots, t_N)$, where $t_i = 1$ if $i \in I_k$ and $t_i = 0$ otherwise. Consequently, the affine complex space $\Cbb \rho(\Pcal_0) + \dots + \Cbb \rho(\Pcal_m)$, defined by setting $t_i = t_j$ if $i,j \in I_k$, is clearly isomorphic as a vector space to $\Cbb^{m+1}$. 

We discuss the Minkowski schemes of the del Pezzo toric diagrams in the next example.

\begin{example} \label{example_toricdelpezzo_minkowski_scheme}
It is immediate from Example \ref{example_toricdelpezzo_minkowski_decomposition} that among toric diagrams of the toric del Pezzo cones, the polygons $\Qcal_1, \Qcal_2, \Qcal_3, \Qcal_4$ all have an irreducible Minkowski scheme, whereas $\Qcal_5$ is the only one with a reducible Minkowski scheme. In addition, the ones with topologically non-trivial Minkowski schemes are $\Qcal_1, \Qcal_4, \Qcal_5$. Specifically:
\begin{itemize}
\item $\Qcal_1$ has a unique lattice Minkowski decomposition with two summands, and so $\Mcal^{\textnormal{red}}(\Qcal_1)  \simeq \Cbb^2$ and $\ol{\Mcal}^{\textnormal{red}}(\Qcal_1) \simeq \Cbb$ as affine varieties. 
\item $\Qcal_2$ has no non-trivial lattice Minkowski decomposition, so $\ol{\Mcal}^{\textnormal{red}}(\Qcal_2)$ is only a point. 
\item $\Qcal_3$ has only a trivial lattice Minkowski decomposition, therefore $\ol{\Mcal}^{\textnormal{red}}(\Qcal_3)$ is merely a point. By Altmann's computation \cite[Section (9.2)]{Alt97}, we see that $\ol{\Mcal}(\Qcal_3) \simeq \Spec \Cbb[\varepsilon]/\varepsilon^2$ as schemes. 
\item $\Qcal_4$ has one lattice Minkowski decomposition into two summands, so that $\ol{\Mcal}^{\textnormal{red}}(\Qcal_4) \simeq \Cbb$. 
\item $\Qcal_5$ has two maximal Minkowski decompositions into lattice summands (into two and three summands) and $\ol{\Mcal}^{\textnormal{red}}(\Qcal_5)$ has two irreducible components corresponding to $\Cbb$ and $\Cbb^2$, respectively.
\end{itemize}
\end{example}

\subsection{Deformations and Minkowski decompositions}\label{s:4.3}
Recall that $\Pcal$ is a two-dimensional toric diagram, which can be assumed to have $0$ as a vertex after a translation. 
Each vertex \( a \in \Pcal \) can be reached from $0$ by walking along the edges of \( \Pcal \). Hence there is \(\lambda = ( \lambda_1, \dots, \lambda_N) \in \Zbb^N \) such that
\[ a = \sum \lambda_i d_i. \]
Given an element \( t \in C(\Pcal) \) and the polygon $\Pcal_t$, obtained by walking along the edges $t_i d_i$, we get the vertex \( a_t \in \Pcal_t \) by 
\[ a_t := \sum_{i=1}^{N} t_i \lambda_i d_i. \] 
Since $t \in V$, this is independent of the chosen edge path from $0$ to $a$.  
The \emph{tautological cone} \( \wt{C}(\Pcal) \subset \mathbb{R}^2 \times V \) is defined as
\[ \wt{C}(\Pcal) := \set{(a, t), \; a \in \Pcal_t, t \in C(\Pcal)}. \]
Let $V_{\Zbb}$ denote the lattice points of $V$. By the classification of toric morphisms \cite[Theorem 3.3.4]{CLS11}, the natural projection \( \wt{C}(\Pcal) \to C(\Pcal) \) induces a toric morphism 
\[ \pi: X := \Spec \mathbb{C}[\wt{C}(\Pcal)^{\vee} \cap (\Zbb^2 \times V_{\Zbb}^{*})] \to S := \Spec \Cbb[C(\Pcal)^{\vee} \cap V_{\Zbb}^{*}]. \] 
The space \( S \) turns out to be too large for the morphism to be flat. To circumvent this issue, Altmann \cite[Sections (4.4) \& (5.2)]{Alt97}  
first remarks that the inclusion \( C(\Pcal) \subset \mathbb{R}^N_{\geq 0} \) induces a morphism 
\( p \colon  S \to \mathbb{C}^N \)
whose scheme-theoretic image \(\ol S\) is defined by the ideal
\[
\ker\left(
\Cbb[t_1,\dots,t_N]
\longrightarrow
\Cbb[C(\Pcal)^\vee\cap V_{\Zbb}^{*}]
\right).
\]
The map \( \nu_S \colon  S \to \ol{S} \) is moreover a normalization morphism. We have the following property.

\begin{theorem}[{\cite[Remark in (4.4)]{Alt97}}]
The space \( \Mcal \) is the largest closed subscheme which is included in \( \ol{S} \) and comes from $\Cbb^N / \Cbb \rho(\Pcal)$ via $l^{-1}$. More precisely, $\Mcal = l^{-1}(\ol{\Mcal})$. 
\end{theorem}
\noindent The map $s \in \Rbb_{\geq 0} \to s \rho(\Pcal) \in C(\Pcal)$ induces a closed embedding $\Cbb \hookrightarrow S$ \cite[end of Section (4.4)]{Alt97}. On the other hand, recall that the Hilbert basis of the toric cone $C_0$ with toric diagram $\Pcal$ contains $(0,0,1)$ \cite[Lemma 2.2.14]{CLS11}, and embeds $C_0$ into $\Cbb^{N_0-1} \times \Cbb$ where the last factor is given by the coordinate function corresponding to $(0,0,1)$. The monomial functions on the $\Cbb^{N_0-1}$ factor then lift to monomial functions on $X$ and, together with the closed morphism $\Cbb \hookrightarrow S $, induce a closed embedding $X \hookrightarrow \Cbb^{N_0-1} \times S$. Composing this with the map $p \colon S \to \Cbb^N$, let $\ol{X}$ be the scheme-theoretic image of the resulting morphism $X \hookrightarrow \Cbb^{N_0-1} \times S \to \Cbb^{N_0 -1} \times \Cbb^N$. The induced morphism \( \nu_X \colon  X \to \ol{X} \) turns out to be again a normalization morphism.
%where \( \ol{X} \) is the scheme theoretic image of \( X \to \mathbb{C}^{N_0} \times \mathbb{C}^N \) with $N_0$ the cardinality of the Hilbert basis of $C_0$.  

The relationship between $(X,\,S)$ and $(\overline{X},\,\overline{S})$ is given by the following.

\begin{theorem}[{\cite[Theorem (5.1)]{Alt97}}]
The toric morphism \( \pi: X \to S \) induces a map \( \ol{\pi}: 	\ol{X} \to \ol{S} \) such that \( \pi \) is the base change of \( \ol{\pi} \) by the normalization map \( \nu_S : S \to \ol{S} \). Moreover, the family, obtained by restricting $\ol{\pi} : \ol{X} \to \ol{S}$ to $\Mcal \subset \ol{S}$ and composing with the projection map $l:\Mcal \to \ol{\Mcal}$,
is flat at \( 0 \in \ol{\Mcal} \) with fiber at \( 0 \) equal to \( C_0 \). 
\end{theorem}

The above constructions can be summarized in the following commutative diagram of scheme morphisms:

\begin{equation*} 
		\xymatrix{\ol{X} \times_{\ol{S}} \Mcal \ar[d]_{\ol{\pi}} \ar[rr]^{\pi_1} && \ol{X}
		\ar[d]_{\ol{\pi}} 
		&&  \ar[ll]_{\nu_X} X
		\ar[d]_{\pi} \\
		\Mcal \ar[d]_{l} \ar[rr]^{\textnormal{inclusion}} && \ol{S} &&\ar[ll]_{\nu_S} S. \\
        \ol{\Mcal}}
\end{equation*}

By a detailed consideration of the Kodaira--Spencer map \cite[Section 6]{Alt97} and the obstruction map \cite[Section 7]{Alt97}, one arrives at the main result of \cite{Alt97}, stated as follows.

\begin{theorem}[{\cite[Corollary (7.2)]{Alt97}}] \label{thm:versalfamily}
The family \( \ol{X} \times_{\ol{S}} \Mcal \xrightarrow[]{l \circ \ol{\pi}} \ol{\Mcal} \) is a versal affine deformation family of $C_0$. 
\end{theorem}
\noindent 
Thus, the Minkowski scheme \(\ol{\mathcal{M}}\) introduced in Definition
\ref{def:minkscheme} is indeed the versal base space of \(C_0\). Its reduction is
denoted by \(\ol{\mathcal{M}}^{\textnormal{red}}\), and the irreducible
components of \(\ol{\mathcal{M}}^{\textnormal{red}}\) are described in Theorem
\ref{thm:minkschemecomp}. The base change of the versal family to each
irreducible component of
\(\ol{\mathcal{M}}^{\textnormal{red}}\) is again a toric family, as suggested
by Christophersen's observation for two-dimensional toric cones
\cite{Chr91}, and has an almost explicit description.

\begin{theorem}[{\cite[Theorem (8.1)]{Alt97}}]
\label{theorem_deformation_classification}
Let \(C_0\) be a toric Gorenstein singularity defined by a toric
diagram \(\Pcal\). The irreducible components of
\(\ol{\Mcal}^{\mathrm{red}}\) are in bijection with maximal
Minkowski decompositions
\(
\Pcal=\Pcal_0+\dots+\Pcal_m
\)
into lattice summands. The component corresponding to such a
decomposition is affine-isomorphic to \(\Cbb^m\), and the restriction
of the versal family to this component has affine toric total space $X_0$
associated with the cone
\[
\widetilde{\Ccal}
=
\textnormal{Cone}
\left(
\bigcup_{k=0}^m
\Pcal_k\times\set{e_k}
\right)
\subset
\Rbb^2\oplus\Rbb^{m+1},
\]
where \(e_0,\dots,e_m\) are the standard basis vectors of
\(\Zbb^{m+1}\). In other words,  \[
X_0
=
\Spec\Cbb
\left[
\widetilde{\Ccal}^{\vee}
\cap
\left(
(\Zbb^2)^{*}
\oplus
(\Zbb^{m+1})^{*}
\right)
\right].
\]
\end{theorem}

We conclude this subsection with some examples, illustrating the above theory.
\begin{example}\label{examplesss}
From Example \ref{example_toricdelpezzo_minkowski_scheme}, the only toric del Pezzo cones with non-trivial versal deformations (over a reduced base) are those associated with the toric diagrams $\Qcal_1, \Qcal_4, \Qcal_5$. We consider each individually.
\begin{itemize}
\item The toric diagram $\Qcal_1$ has only one maximal Minkowski decomposition into lattice summands with two summands. The restriction of the versal family to the corresponding (reduced) component is then given by an affine morphism $\pi_0:X_0 \to \Cbb$, where $X_0$ is the toric affine variety arising from the polyhedron
\[ \wt{\mathcal{C}} = \textnormal{Cone} \set{ \Spvek{0;0;1;0},\Spvek{1;1;1;0},\Spvek{0;0;0;1},\Spvek{-1;1;0;1}}. \]

\item The reduced versal base space $\ol{\Mcal}^{\textnormal{red}}(\Qcal_4)$ has only one irreducible component $\Cbb$ corresponding to the unique maximal Minkowski decomposition into lattice summands, consisting of two summands. The total space over this component arises as the toric variety defined by 
\begin{equation*}
 \wt{\mathcal{C}} = \textnormal{Cone}\set{\Spvek{0;0;1;0},\Spvek{0;1;1;0},\Spvek{1;1;1;0},\Spvek{0;0;0;1},\Spvek{-1;1;0;1}}.
\end{equation*}
\item The reduced versal base space $\ol{\Mcal}^{\textnormal{red}}(\Qcal_5)$ has two irreducible components, namely $\Cbb, \Cbb^2$, corresponding to its two maximal Minkowski decompositions into two and three summands, respectively. The versal family restricted to each component has toric total deformation space defined by 
\begin{equation*}
\begin{split}
\wt{\mathcal{C}} &= \textnormal{Cone}\set{\Spvek{0;0;1;0}, \Spvek{1;0;1;0},\Spvek{1;1;1;0}, \Spvek{0;0;0;1},\Spvek{0;1;0;1},\Spvek{1;1;0;1}}  \\
\end{split}    
\end{equation*}
and
\begin{equation*}
\begin{split}
\wt{\mathcal{C}} &= \textnormal{Cone} \set{ \Spvek{0;0;1;0;0},\Spvek{1;0;1;0;0},\Spvek{0;0;0;1;0},\Spvek{0;1;0;1;0},\Spvek{0;0;0;0;1},\Spvek{1;1;0;0;1}},
\end{split}    
\end{equation*}
respectively.
\end{itemize}
\end{example}

\subsection{Computing the deformations} \label{subsection_minkowski}
Let \( \Pcal \) be the toric diagram of the Gorenstein cone \( \mathcal{C} \subset \mathfrak{N}_{\mathbb{R}} \). As suggested by Altmann's theory \cite{Alt97}, we proceed with the following steps to get a precise description of the (reduced) versal deformation space of the corresponding toric Gorenstein cone.
\begin{itemize}
    \item Compute the Minkowski summand cone \(C(\Pcal)\), its extremal
summands, and its maximal decompositions. The maximal Minkowski
decompositions into lattice summands, and hence the irreducible
components of the reduced versal base, must then be determined
separately, for example by means of the zero-sum partitions described
in Remark \ref{rmk:toricdiagramminkdec}.
   
    \item 
    For each irreducible component of \(\ol{\Mcal}^{\textnormal{red}}\), obtain the total space \(X_0\) of the deformation produced by base change of the versal family to that component by first writing down the moment cone $\wt{\mathcal{C}}^{\vee}$ of $X_0$, then computing the toric ideal of $X_0$ using the Hilbert basis of $\wt{\Ccal}^{\vee}$. 
\end{itemize}
The relevant Macaulay2 codes are provided in Appendix \ref{s:A1} (and Figure \ref{figure_macaulay2_cfo13}). Note however that we choose not to include the lines of code that would output the total space since the information does not pertain to our purpose. The interested reader may write down such a code by noting that the total space is itself a toric variety and the defining convex cone is already known. We next consider the example of the Gauntlett--Martelli--Sparks--Waldram (GMSW) family of toric Calabi--Yau cones \cite{GMSW}.

\subsection{Example: The Gauntlett--Martelli--Sparks--Waldram family of toric Calabi--Yau cones}\label{egsss}

\begin{figure}[!htbp] 
    \begin{subfigure}[b]{0.45\textwidth}
    \centering
        \resizebox{\linewidth}{!}{
                    \begin{tikzpicture}
                \draw[step=1cm,gray,dashed] (-1.9,-1.9) grid (1.9,1.9);
                \node at (0,0) [circle, fill, inner sep = 1.5pt] {};

                \draw (-1,-1) -- (0,-1) -- (1,1) -- (-1,0) -- cycle;
                \end{tikzpicture}
        }
        \caption{\( \Ycal^{2,1} \)}   
    \end{subfigure}
    \begin{subfigure}[b]{0.45\textwidth}
        \centering
        \resizebox{\linewidth}{!}{
                            \begin{tikzpicture}
                \draw[step=1cm,gray,dashed] (-1.9,-1.9) grid (4.9,4.9);
                \node at (0,0) [circle, fill, inner sep = 1.5pt] {};
                \draw (-1,-1) -- (0,-1) -- (4,4) -- (0,1) -- cycle;
            \end{tikzpicture}
        }
        \caption{\( \Ycal^{5,3} \)}
    \end{subfigure}
\caption{Toric diagrams of the members $Y^{2,1}, Y^{5,3}$ in the Gauntlett--Martelli--Sparks--Waldram family \eqref{eq:gmsw}. The polygon $\Ycal^{2,1}$ is affine equivalent to $\Qcal_3$ (cf.~Figure \ref{figure_delpezzo_family}). Here, $(p-1)$ is the number of interior lattice points.} 

\label{figure_gmsw}
\end{figure}

In \cite{GMSW}, Gauntlett--Martelli--Sparks--Waldram consider a family of Calabi--Yau toric cones $Y^{p,q}$ with toric diagram $\Ycal^{p,q}$ given by the convex hull of the points
\begin{equation} \label{eq:gmsw}
v_1 = \Spvek{0;0}, \quad v_2 = \Spvek{1;0}, \quad v_3 = \Spvek{p;p}, \quad v_4 = \Spvek{p-q-1;p-q}, \qquad p,q \in \Nbb_{>0},\quad p > q,\quad \gcd(p,\,q)=1, 
\end{equation}
in $\mathbb{R}^{2}$ (see \cite{MSY08} for details). For general $p > q$, the cones are irregular. Being the first such examples, they probably form the most well-known infinite family of irregular Calabi--Yau cones in the literature. Unfortunately each cone in this family turns out to have trivial (reduced) versal base space, hence they cannot be deformed into affine Calabi--Yau manifolds. This last fact was observed in \cite[Introduction]{Conlon14} without proof (where ``rigid'' in \cite{Conlon14} means the absence of non-trivial deformations over the reduced versal base, but not rigidity in the sense adopted here). That the reduced versal base space is a point can be seen by appealing to Lemma \ref{lem:practical-crit} and Theorem \ref{thm:minkschemecomp}. Regardless of this short account, we provide a detailed proof to illustrate Altmann's theory. 

\begin{prop} \label{prop:gmsw_trivial}
Every toric cone in the GMSW family is non-rigid, but its reduced versal base space is a point. 
\end{prop}

\begin{proof} 
For a given $p, q$, the oriented edges of \( \Ycal^{p,q} \) are
\[ d_1 = (1,0), \quad d_2 = (p-1,p), \quad d_3 = (-q-1,-q), \quad d_4 = (-p+q+1, -p+q). \] 
The scheme $\Mcal$ of \( \mathcal{Y}^{p,q} \) is given by 
\begin{align*}
\Ical &= \sprod{t_1^k d_1 + t_2^k d_2 + t_3^k d_3 + t_4^k d_4 \,|\, k \geq 1} \\
&= \sprod{t_1^k + (p-1)t_2^k -(q+1) t_3^k + (-p+q+1)t_4^k, \; pt_2^k - q t_3^k + (q-p) t_4^k\,|\, k \geq 1}. 
\end{align*}
After subtracting the second equation from the first, we find that
\[ \Ical = \sprod{t_1^k - t_2^k - t_3^k + t_4^k, \; p(t_2^k - t_4^k) + q(t_4^k - t_3^k)\,|\, k \geq 1}. \]
The linear equations defining \( V(\Ycal^{p,q}) \) are \( t_1 - t_2 - t_3 + t_4 = 0, \quad p(t_2 - t_4) + q(t_4 - t_3) = 0 \), hence only two coordinates \( t := t_1 \), \( s := t_1 - t_4 \) suffice to determine $V(\Ycal^{p,q})$.  
We derive that
\[ t_2 = t - \left(\frac{p}{p+q}\right)s, \qquad t_3 = t - \left(\frac{q}{p+q}\right)s, \qquad t_4 = t - s. \] 
For $k=2$, it follows that
\[ p(t_2^2 - t_4^2) = \left(\frac{pq}{p+q}\right)s\set{2t - \left(\frac{2p+q}{p+q}\right)s}, \quad q(t_4^2 - t_3^2) = -\left(\frac{pq}{p+q}\right)s\set{2t - \left(\frac{p+2q}{p+q}\right) s }. \] 
The quadratic equation \( p(t_2^2 - t_4^2) + q(t_4^2 - t_3^2) = 0 \) then becomes
\[ p(t_2^2 - t_4^2) + q(t_4^2 - t_3^2) = \left(\frac{pq}{p+q}\right) s \cdot \left(\frac{q-p}{p+q}\right)s = 0 \iff s^2 = 0.\]
%and the unconstrained coordinate \( t \) corresponds to the line \( \Cbb(1,1,1,1) \) which is quotiented out by \( l \). 
For $k \geq 2$, we have modulo $s^2$:
\[ (t-\alpha s)^k = t^{k-1}\left(t - k \alpha s\right).\]
Replacing this in the two coordinate functions of $g_k$ yields (modulo $s^2)$:
\[
\begin{aligned}
&t_1^k -t_2^k -t_3^k + t_4^k = t^{k-1}\left( \frac{kp}{p+q}s + \frac{kq}{p+q}s - ks\right) = 0 .\\
&p(t_2^k -t_4^k) + q(t_4^k - t_3^k) = pt^{k-1} \tuple{k - \frac{kp}{p+q}}s + q t^{k-1} \tuple{\frac{kq}{p+q} - k} s = 0. 
\end{aligned}
\]
Thus the resulting ideal in $\Cbb[t,s]$ is $(s^2)$, i.e., $\Mcal = \Spec \Cbb[t,s]/(s^2)$. Moreover, in the coordinates $(t,s)$ the line $\Cbb(1,1,1,1)$ depends on the single coordinate $t$, which is quotiented out by $l \colon V_{\Cbb} \to V_{\Cbb} / \Cbb (1,1,1,1)$. 
It follows that the versal base \( \ol{\Mcal} = \Mcal/ \Cbb(1,1,1,1) \) is a fat point (and \( \ol{\Mcal}^{\textnormal{red}} \) is a point). Thus, \( \dim_{\Cbb} T_0 \ol{\Mcal} = 1 = \dim_{\Cbb} T^1_{Y^{p,q}} \) and the GMSW family is non-rigid, but any deformation over a reduced base is of course trivial because \( \ol{\Mcal}^{\textnormal{red}} \) is only a point. 
\end{proof}

\section{Theorem \ref{mtheorem_cfo}}\label{s:thmb}

In this section, we provide the proof of Theorem
\ref{mtheorem_cfo}. Recall from the Introduction that the
Cho--Futaki--Ono family is defined by the toric diagram
\(\Pcal^{r,s}\) with \(2r+3\) vertices, given by
\begin{align*}
&(p_0, q_0) = (0,0), \dots,(p_k, q_k) = \left(k,\frac{k(k+1)}{2}\right), \quad 0 \leq k \leq r, \\
&(p_{r+1+j},q_{r+1+j}) = \left(r+1-j, \frac{(r+1)(r+2)}{2} + s-\frac{j(j+1)}{2} \right), \quad 0 \leq j \leq r, \\
&\;(p_{2r+2}, q_{2r+2}) = (0,1),
\end{align*}
where \( r, s \in \Nbb_{> 0} \). Throughout this section, we identify toric
diagrams that differ by a lattice translation. When convenient, we
take the lowest interior lattice point as the origin; that is, we
translate \((p_j,q_j)\) by \((-1,-2)\).

In Section \ref{s:mink}, we present a distinguished maximal Minkowski
decomposition into lattice summands for these cones, before determining the Reeb field of a particular subfamily in Section \ref{s:minn} through minimizing the volume function. In Section \ref{egs}, we demonstrate how to compute explicitly the Reeb field of some examples using a computer algorithm, the syntax of which is contained in Appendix \ref{s:A2} (and Figure \ref{figure_mathematica_cfo13}).

\subsection{Minkowski decomposition}\label{s:mink}

Our first theorem gives a distinguished maximal Minkowski
decomposition into lattice summands of \(\Pcal^{r,s}\).

\begin{theorem}\label{theorem_minkowski_decomposition}
The toric diagram \(\Pcal^{r,s}\) admits a maximal Minkowski
decomposition into lattice summands given by 
$\Pcal^{r,s}
=
\Lcal^1+\dots+\Lcal^r+\Delta^s,$
where
\(
\Lcal^j
=
\textnormal{Conv}
\set{\Spvek{0;0},\Spvek{1;j}},
\qquad
j=1,\dots,r,
\)
and
\(
\Delta^s
=
\textnormal{Conv}
\set{\Spvek{0;0},\Spvek{0;1},\Spvek{1;r+s+1}}.
\)
This decomposition determines an irreducible component
$B^{r,s}\simeq\Cbb^r$ of the reduced versal base.
\end{theorem}

\begin{proof}  
We denote by \( (a\to a+1) \) the oriented edge \( (p_{a+1},q_{a+1})-(p_a,q_a)\). 
Observe that \( \Pcal^{r,s}\) contains \( r \) pairs of parallel edges with equal length:
\[\set{(j\to j+1), (r+j+1 \to r+j+2)}, \quad 0 \leq j \leq r-1. \]
Indeed, one can readily check that 
\begin{align*}
&(p_{j+1},q_{j+1}) - (p_j,q_j) = (1,j+1), \\
&(p_{r+j+2},q_{r+j+2}) - (p_{r+j+1},q_{r+j+1}) = (-1,-j-1).
\end{align*}
The three remaining edges \( \set{(2r+1 \to 2r+2), (2r+2\to 0), (r \to r+1)}\)
 sum up to zero and form a triangle by walking along them. 

By Remark \ref{rmk:toricdiagramminkdec}, this partition of the
oriented edges gives the lattice Minkowski decomposition of the claim. Since its summands are segments and a triangle, Corollary
\ref{cor:minkalways} shows that it is maximal as a Minkowski
decomposition. In particular, it is a maximal Minkowski decomposition into lattice summands. By Theorem \ref{thm:minkschemecomp}, it therefore determines an
irreducible component \(B^{r,s}\) of
\(\ol{\Mcal}^{\textnormal{red}}\). Since the decomposition has
$(r+1)$-summands, this component is affine-isomorphic to
$\Cbb^r$. (We do not claim that \(B^{r,s}\) is the only irreducible component of the reduced versal base.)
\end{proof}

We next define smoothing components and recall the criterion that determines them. Let \(B\) be an irreducible component of the reduced versal base and let \(\pi_B:X_B\to B\) denote the restriction of the versal family to \(B\). We call \(B\) a \emph{smoothing component} if there exists a nonempty Zariski-open subset $B^{\mathrm{sm}}\subset B$ such that 
\( 
\pi_B^{-1}(B^{\mathrm{sm}})\longrightarrow B^{\mathrm{sm}} 
\) 
is smooth. 

\begin{lemma}\label{lemma_smoothing_component} 
Let \(C_0\) be a three-dimensional isolated toric Gorenstein singularity associated with a toric diagram \(\Pcal\subset\Rbb^2\), and let 
\(
\Pcal=\Pcal_0+\dots+\Pcal_m 
\) 
be a maximal Minkowski decomposition into lattice summands. Let \(B\simeq\Cbb^m\) be the corresponding irreducible component of the reduced versal base. Then \(B\) is a smoothing component if and only if every summand \(\Pcal_j\) is either a primitive lattice segment or a unimodular lattice triangle.
\end{lemma} 

\begin{proof}
As remarked in \cite[Remark 2.13]{IV12}, this lemma follows from \cite[Corollary 2.12]{IV12}. Here we provide the details. 

\cite[Corollary 2.12]{IV12} describes the singularities on the general fiber of a deformation of a variety with a holomorphic complex torus action of complexity one. (Here, 
the ``complexity'' of a holomorphic torus action $S \curvearrowright X$ on a variety $X$
is given by the difference $\dim_{\Cbb} X - \dim_{\Cbb} S$.)
In our case, the relevant variety $C_0$ is toric and so admits the holomorphic action of a torus $T_{\Cbb}$  of complexity zero. As explained in \cite[Section 4.2]{alty}
and \cite[Remark 1.8]{IV12}, one can ``downgrade'' the action of $T_{\Cbb}$ to a complexity one torus action $\widetilde{T}_{\mathbb{C}} \subset T_{\mathbb{C}}$ on $C_0$. Indeed, every primitive element $R$ in the toric weight lattice $\mathfrak{M}$ of $C_0$ yields a choice of downgrading (with torus coweight lattice $\widetilde{\mathfrak{N}} = \mathfrak{N} \cap R^{\perp})$. From the theory developed in \cite{IV12},  $C_0$ is completely determined---as a $\widetilde{T}_{\Cbb}$-variety---by a polyhedral divisor $\Dcal$ on $\Pbb^1$; that is, a formal finite sum of the form 
\[ \Dcal = \sum \Pcal_D\otimes D, \]
where $D$ runs over all prime divisors of $\Pbb^1$. The deformation theory of  $C_0$ as a $\widetilde{T}_{\mathbb{C}}$-variety is then described by suitable Minkowski decompositions of the polyhedral coefficients $\Pcal_D$ \cite[Theorem 2.8]{IV12}. These coefficients depend a priori on the previous downgrading choice. It turns out that by choosing $\widetilde{T}_{\Cbb}$ to be the torus with coweight lattice $\wt{\mathfrak{N}}$ being the ambient lattice of the toric diagram $\Pcal$ (i.e. by taking the canonical choice $R= (0,0,1)$), $\Dcal$ is given by 
\[ \Dcal = \Pcal \otimes \set{0} + \emptyset \otimes \{\infty\}; \]
see \cite[Remark 1.8]{IV12} and the paragraph following \cite[Remark 3.1]{IV12}. Thus under this choice, the deformation theory of $C_0$ as a $\widetilde{T}_{\mathbb{C}}$-variety translates back to Altmann's deformation theory of toric varieties, that is, in terms of the maximal Minkowski decompositions of $\Pcal$ into lattice summands; see \cite[Example 78]{alty}
and the second paragraph after \cite[Remark 3.1]{IV12}.

Let $\pi$ be a deformation of $C_0$ associated to 
a maximal Minkowski decomposition of $\Pcal$ into lattice summands $\Pcal = \Pcal_0 + \dots + \Pcal_m$. Then \cite[Corollary 2.12]{IV12} tells us that the general fiber of $\pi$ has the same singularities as the toric singularities corresponding to the convex cones over $\Pcal_j$ given by
\[ \sigma_{j}  = \text{Cone}(\Pcal_j \times \set{1}). \]
The lattice of the relevant toric action is a sublattice of the ambient lattice $\mathfrak{N} = \wt{\mathfrak{N}} \oplus \mathbb{Z}(0,0,1)$. The general fiber of $\pi$ is smooth if and only if for all $j$, the toric singularity associated with $\sigma_j$ is smooth. By \cite[Theorem 3.1.19(a)]{CLS11}, this occurs if and only if the primitive lattice generators of $\sigma_j$ form part of a $\mathbb{Z}$-basis of $\mathfrak{N}$. Since $\mathfrak{N}$ has rank 3, this condition requires each $\sigma_j$ to be generated by at most 3 rays, meaning each summand $\mathcal{P}_j$ can have at most 3 vertices (i.e., it must be a segment or a triangle). Furthermore, for segments and triangles, these generators form part of a $\mathbb{Z}$-basis if and only if the segments are primitive and the triangles are unimodular. This allows us to conclude.
%see \cite[Section 5.1]{CFP22} for details. 
\end{proof}
With this criterion, we obtain the following. 

\begin{corollary}\label{corollary_cy_existence}

Let \((C_0^{r,s},T_{\Cbb})\) be the toric Calabi--Yau cone associated
with the toric diagram \(\Pcal^{r,s}\), let \(\omega_0\) denote its
Calabi--Yau cone metric, and let
\(
\pi_{r,s}:\mathcal{W}^{r,s}
\longrightarrow
B^{r,s}\simeq\Cbb^r
\)
denote the family obtained by base change of the Altmann versal family
to the irreducible component \(B^{r,s}\) determined by the
decomposition in Theorem
\ref{theorem_minkowski_decomposition}. Then there exists a nonempty Zariski-open subset $B_{\mathrm{sm}}^{r,s}\subset B^{r,s}$ such that: 
\begin{itemize} 
    \item The restriction 
    \(
    \pi_{r,s}^{-1}(B^{r,s}_{\mathrm{sm}}) \longrightarrow B^{r,s}_{\mathrm{sm}} 
    \) 
    is smooth. In particular, for every \(t\in B^{r,s}_{\mathrm{sm}}\), the fiber 
    \(
    M_t:=\pi_{r,s}^{-1}(t) 
    \) 
    is an affine smoothing of \(C_0^{r,s}\). 
    \item For every \(t\in B^{r,s}_{\mathrm{sm}}\), there exists a diffeomorphism \(\Phi_t\) between complements of compact subsets of \(C_0^{r,s}\) and \(M_t\) such that, for every K\"ahler class \(\mathfrak{k}\) on \(M_t\) and every \(g\in T_{\mathbb{C}}\), there exists a complete Calabi--Yau metric \(\omega_{g,t}\in\mathfrak{k}\) for which 
    \( 
    \Phi_t^{*}\omega_{g,t} 
    \) 
    is asymptotic to \(g^{*}\omega_0\). 
\end{itemize} 
\end{corollary} 

\begin{proof} 
The maximal Minkowski decomposition of $\Pcal^{r,s}$ into lattice summands
given by Theorem \ref{theorem_minkowski_decomposition} is 
\( 
\Pcal^{r,s}=\Lcal^1+\dots+\Lcal^r+\Delta^s, 
\) 
where 
\[ 
\Lcal^j = \textnormal{Conv}\set{\Spvek{0;0},\Spvek{1;j}} \qquad\textrm{and}\qquad \Delta^s = \textnormal{Conv} \set{\Spvek{0;0},\Spvek{0;1},\Spvek{1;r+s+1}}. 
\] 
Each \(\Lcal^j\) is a primitive lattice segment because 
$\gcd(1,j)=1$. 
Moreover, the two edge vectors of \(\Delta^s\) originating from the origin form a lattice basis, as
\[ 
\left| \det \begin{pmatrix} 0 & 1\\ 1 & r+s+1 \end{pmatrix} \right| =1. 
\] 
Thus, \(\Delta^s\) is a unimodular lattice triangle. Lemma \ref{lemma_smoothing_component} therefore implies that 
the component \(B^{r,s}\) is a smoothing component. The existence of an AC Calabi--Yau metric in every K\"ahler class of each smooth fiber now follows from \cite[Theorem 4.3]{CH24}. 
\end{proof}

\subsection{The Reeb field as the volume minimizer}\label{s:minn}
We next compute the Calabi--Yau Reeb fields of the subfamily
\[
\Pcal^{s}:=\Pcal^{1,s}
=
\textnormal{Conv}
\set{
\Spvek{-1;-2},
\Spvek{0;-1},
\Spvek{1;s+1},
\Spvek{0;s},
\Spvek{-1;-1}
},
\]
where \(s\in\Nbb_{>0}\). This polygon has \(s\) interior lattice
points, and we restrict to this subfamily due to computational
constraints, as noted in the Introduction. To keep the formulas below more compact, we set
\[
\sigma:=s+1.
\]
Thus, \(\sigma\geq2\), and the vertices above may equivalently be written as
\[
\Spvek{-1;-2},\quad
\Spvek{0;-1},\quad
\Spvek{1;\sigma},\quad
\Spvek{0;\sigma-1},\quad
\Spvek{-1;-1}.
\]

\begin{figure} 
    \begin{subfigure}[b]{0.32\textwidth}
        \centering
        \resizebox{\linewidth}{!}{
                            \begin{tikzpicture}
                \draw[step=1cm,gray,dashed] (-1.9,-2.9) grid (1.9,2.9);
                                \node at (0,0) [circle, fill, inner sep = 1.5pt] {};
                \draw (-1,-2) -- (0,-1) -- (1,2) -- (0,1) -- (-1,-1) -- cycle;
            \end{tikzpicture}
        }
        \caption{\( \Pcal^1 \)}
        \label{fig:subfig1}
    \end{subfigure}
    \begin{subfigure}[b]{0.32\textwidth}
    \centering
        \resizebox{\linewidth}{!}{
                    \begin{tikzpicture}
                \draw[step=1cm,gray,dashed] (-1.9,-2.9) grid (1.9,3.9);
                                \node at (0,0) [circle, fill, inner sep = 1.5pt] {};
                \draw (-1,-2) -- (0,-1) -- (1,3) -- (0,2) -- (-1,-1) -- cycle;
                \end{tikzpicture}
        }
        \caption{\( \Pcal^2 \)}   
        \label{fig:subfig2}
    \end{subfigure}
    \begin{subfigure}[b]{0.32\textwidth}
        \centering
        \resizebox{\linewidth}{!}{
            \begin{tikzpicture}
                \draw[step=1cm,gray,dashed] (-1.9,-2.9) grid (1.9,4.9);
                                \node at (0,0) [circle, fill, inner sep = 1.5pt] {};

                \draw (-1,-2) -- (0,-1) -- (1,4) -- (0,3) -- (-1,-1) -- cycle;
            \end{tikzpicture}
        }
        \caption{\( \Pcal^3 \) }
        \label{fig:subfig3}
    \end{subfigure}
    \begin{subfigure}[b]{0.32\textwidth}
        \centering
        \resizebox{\linewidth}{!}{
                            \begin{tikzpicture}
                \draw[step=1cm,gray,dashed] (-1.9,-2.9) grid (2.9,6.9);
                                \node at (0,0) [circle, fill, inner sep = 1.5pt] {};

                \draw (-1,-2) -- (0,-1) -- (1,1) -- (2,6) -- (1,5) -- (0,3) -- (-1,-1) -- cycle;
            \end{tikzpicture}
        }
        \caption{\( \Pcal^{2,2} \)}
        \label{fig:subfig4}     
    \end{subfigure}
    \begin{subfigure}[b]{0.32\textwidth}
        \centering
        \resizebox{\linewidth}{!}{
                            \begin{tikzpicture}
                \draw[step=1cm,gray,dashed] (-1.9,-2.9) grid (2.9,7.9);
                                \node at (0,0) [circle, fill, inner sep = 1.5pt] {};

                \draw (-1,-2) -- (0,-1) -- (1,1) -- (2,7) -- (1,6) -- (0,4) -- (-1,-1) -- cycle;
            \end{tikzpicture}
        }
        \caption{\( \Pcal^{2,3} \)}
        \label{fig:subfig5}
    \end{subfigure}
    \begin{subfigure}[b]{0.31\textwidth}
        \centering
        \resizebox{\linewidth}{!}{
                            \begin{tikzpicture}
                \draw[step=1cm,gray,dashed] (-1.9,-2.9) grid (2.9,8.9);
                                \node at (0,0) [circle, fill, inner sep = 1.5pt] {};

                \draw (-1,-2) -- (0,-1) -- (1,1) -- (2,8) -- (1,7) -- (0,5) -- (-1,-1) -- cycle;
            \end{tikzpicture}
        }
        \caption{\( \Pcal^{2,4} \)}
        \label{fig:subfig6}
    \end{subfigure}
\caption{Toric diagrams of the Cho--Futaki--Ono members
\(\Pcal^{1,1}=\Pcal^1\), \(\Pcal^{1,2}=\Pcal^2\),
\(\Pcal^{1,3}=\Pcal^3\), \(\Pcal^{2,2}\),
\(\Pcal^{2,3}\), and \(\Pcal^{2,4}\).}
\label{figure_cfo_family}
\end{figure}

\begin{prop}\label{prop_volume}
The normalized volume of the toric Gorenstein cone defined by \(\Pcal^s\) for a
Reeb field \(\xi=(a,b,c)\) is given by
\[
a_0(\xi)
=
\frac{
-a^2+2ab-b^2
+(s^3+2s^2+s+2)ac
-2bc
+s(s^2+3s+3)c^2
}{
(a-b-c)((s+2)a-b-c)(a+c)
(a-b+sc)((s+1)a-b+sc)
}.
\]
\end{prop}

\begin{proof}
We will use the orbifold version of the localization formula
\cite[equation (7.26) and Section 7.4.4]{MSY08} to prove this. Recall that, with our normalization, we have
\begin{equation}\label{equation_volume_toric}
a_0(\xi)
=
\sum_{p_A}
\frac{1}{d_{p_A}}
\frac{1}{
\sprod{u_{p_A},\xi}
\sprod{u_{p_A,(1)},\xi}
\sprod{u_{p_A,(2)},\xi}
}.
\end{equation}
Here, \(p_A\) runs over the torus fixed points of the toric orbifold
obtained from the triangulation described below, whereas
\(u_{p_A},u_{p_A,(1)},u_{p_A,(2)}\) are the corresponding torus
weights, and \(d_{p_A}\) is the orbifold index at \(p_A\).

Since $\sigma=s+1$, the vertices of \(\Pcal^s\), lifted to the height-one hyperplane in
\(\Rbb^3\), take the form
\[
v_1=\Spvek{-1;-2;1},\quad
v_2=\Spvek{0;-1;1},\quad
v_3=\Spvek{1;\sigma;1},\quad
v_4=\Spvek{0;\sigma-1;1},\quad
v_5=\Spvek{-1;-1;1}.
\]
Taking the triangulation with respect to
\[
v=\Spvek{0;0;1},
\]
we obtain five simplicial cones generated by
\[
v_i,\quad v_{i+1},\quad v,
\qquad
i\in\Zbb/5\Zbb.
\]
Write \(p_i\) for the torus fixed point corresponding to the cone
generated by \(v_i,v_{i+1},v\). The orbifold index at \(p_i\) is given by
$d_{p_i}
=
\left|\det(v_i,v_{i+1},v)\right|$, and so
\[
d_{p_1}=d_{p_2}=d_{p_5}=1
\qquad\textrm{and}\qquad d_{p_3}=d_{p_4}=\sigma-1.
\]
The corresponding torus weights are obtained from the cross products
\[
d_{p_i}u_{p_i}=v_i\times v_{i+1},\qquad
d_{p_i}u_{p_i,(1)}=v\times v_i,\qquad
d_{p_i}u_{p_i,(2)}=v_{i+1}\times v.
\]
This leaves us with
\begin{align*}
&u_{p_1}=\Spvek{-1;1;1},\quad
u_{p_1,(1)}=\Spvek{2;-1;0},\quad
u_{p_1,(2)}=\Spvek{-1;0;0},
\\
&u_{p_2}=\Spvek{-(1+\sigma);1;1},\quad
u_{p_2,(1)}=\Spvek{1;0;0},\quad
u_{p_2,(2)}=\Spvek{\sigma;-1;0},
\\
&(\sigma-1)u_{p_3}=\Spvek{1;-1;\sigma-1},\quad
(\sigma-1)u_{p_3,(1)}=\Spvek{-\sigma;1;0},\quad
(\sigma-1)u_{p_3,(2)}=\Spvek{\sigma-1;0;0},
\\
&(\sigma-1)u_{p_4}=\Spvek{\sigma;-1;\sigma-1},\quad
(\sigma-1)u_{p_4,(1)}=\Spvek{1-\sigma;0;0},\quad
(\sigma-1)u_{p_4,(2)}=\Spvek{-1;1;0},
\\
&u_{p_5}=\Spvek{1;0;1},\quad
u_{p_5,(1)}=\Spvek{1;-1;0},\quad
u_{p_5,(2)}=\Spvek{-2;1;0}.
\end{align*}
It subsequently follows that
\begin{align*}
a_0(\xi)
={}&
\frac{1}{(-a+b+c)(-2a+b)a}
+
\frac{1}{
(a(-\sigma-1)+b+c)(-a)(-\sigma a+b)
}
\\
&+
\frac{(\sigma-1)^2}{
(\sigma a-b)
(a-b+c(\sigma-1))
a(1-\sigma)
}
\\
&+
\frac{(\sigma-1)^2}{
(\sigma a-b+c(\sigma-1))
(\sigma-1)a(a-b)
}
\\
&+
\frac{1}{(a+c)(-a+b)(2a-b)}.
\end{align*}
After simplifying, we see that
\[
a_0(\xi)
=
\frac{
-a^2+2ab-b^2
+(\sigma^3-\sigma^2+2)ac
-2bc
+(\sigma^3-1)c^2
}{
(a-b-c)((\sigma+1)a-b-c)(a+c)
(a-b+c(\sigma-1))
(\sigma a-b+c(\sigma-1))
}.
\]
Finally, since \(\sigma=s+1\), we have
\[
\sigma-1=s,\qquad
\sigma+1=s+2,
\]
\[
\sigma^3-\sigma^2+2
=
s^3+2s^2+s+2,
\qquad
\sigma^3-1
=
s(s^2+3s+3).
\]
Substituting these identities into the preceding expression yields the desired formula.
\end{proof}

As it turns out, the Reeb field that minimizes the volume of $\mathcal{P}^{s}$ is irregular for certain values of $s$.
\begin{prop}\label{prop:irreg}
Let $s\in\Nbb_{>0}$ satisfy
$s\equiv2\pmod5$. Then the volume minimizer of \(\Pcal^s\) is irregular.
\end{prop}

\begin{proof}
Set
\(
\sigma:=s+1.
\)
Then
\(
\sigma\equiv3\pmod5.
\)
Let $\xi_s=(a_s,b_s,3)$ denote the unique normalized volume minimizer of \(\Pcal^s\) under the
normalization condition \(c=3\) from Lemma
\ref{lem:reebnormalization}. Under this condition, write
\[
\mathcal V_s(a,b):=a_0(a,b,3)=\frac{N_s(a,b)}{D_s(a,b)}
\]
for the volume $\mathcal{V}_s$ of $\mathcal{P}^{s}$, where
$a_{0}$ is as in Proposition \ref{prop_volume}. Then
\[
\mathcal V_s(a,b)
=
\frac{
-a^2-6b+2ab-b^2
+9(\sigma^3-1)
+3a(\sigma^3-\sigma^2+2)
}{
(a+3)(a-b-3)
(a-b+3(\sigma-1))
(\sigma a-b+3(\sigma-1))
((\sigma+1)a-b-3)
},
\]
which gives the definition of $N_{s}$ and $D_{s}$.

Define
\[
F_s:=(\partial_aN_s)D_s-N_s(\partial_aD_s)
\qquad\textrm{and}\quad
G_s:=(\partial_bN_s)D_s-N_s(\partial_bD_s).
\]
Thus, \(F_s\) and \(G_s\) are the numerators of
\(\partial_a\mathcal V_s\) and \(\partial_b\mathcal V_s\), respectively.
Since \(\xi_s\) lies in the interior of the Reeb cone and minimizes
\(\mathcal V_s\), we must have that
\begin{equation}\label{lovely}
F_s(a_s,b_s)=G_s(a_s,b_s)=0.
\end{equation}
We eliminate \(b\) from these equations by taking their resultant \cite{Gasull} with respect to the $b$ variable. Since \(b_s\) is a common root of
\(F_s(a_s,\cdot)\) and \(G_s(a_s,\cdot)\), this will vanish at
\(a=a_s\).
The Mathematica code in Appendix~\ref{s:prop55code} derives
$F_s$ and $G_s$, computes their resultant as:
\begin{equation}\label{eq:irreg-resultant}
\operatorname{Res}_{b}(F_s,G_s)
=-3^9a^2\sigma^{20}(3\sigma-a)^4(a-3)(a+3)^{10}
(\sigma-1)^6
\bigl((\sigma-1)a+3\sigma\bigr)^4P_\sigma(a),
\end{equation}
verifies this factorization of the resultant, and verifies the coefficients of $P_\sigma$, defined as:
\begin{align*}
P_\sigma(a):={}&11664\sigma^2(\sigma-1)(3\sigma-1) \\
&+1944\sigma
(-4+27\sigma-34\sigma^2-26\sigma^3+32\sigma^4)a \\
&+1296
(1-14\sigma+24\sigma^2+79\sigma^3-185\sigma^4+80\sigma^5)a^2 \\
&+54
(35-50\sigma-1433\sigma^2+4956\sigma^3-4872\sigma^4
+1280\sigma^5)a^3 \\
&+9
(-53+2666\sigma-14205\sigma^2+24552\sigma^3-15088\sigma^4
+2560\sigma^5)a^4 \\
&+3
(-829+8842\sigma-26685\sigma^2+30096\sigma^3-12272\sigma^4
+1280\sigma^5)a^5 \\
&+(\sigma-1)
(1879-10367\sigma+14528\sigma^2-4784\sigma^3+256\sigma^4)a^6 \\
&+(\sigma-1)^2
(-599+1448\sigma-272\sigma^2)a^7 \\
&+72(\sigma-1)^3a^8.
\end{align*}

The interior of the Reeb cone is defined by the conditions
\begin{equation*}
b-a+3>0,\qquad b-(\sigma+1)a+3>0,\qquad a-b+3(\sigma-1)>0,\qquad
\sigma a-b+3(\sigma-1)>0,\qquad a+3>0.
\end{equation*}
Combining the second and third inequalities gives us that
\[
(\sigma+1)a-3<b<a+3(\sigma-1)
\]
so that \(a<3\). The final inequality yields \(a>-3\). Thus,
\[
-3<a_s<3.
\]
Since \(\sigma\equiv3\pmod5\), we know that \(\sigma\geq3\). Consequently,
\[
3\sigma-a_s>0,
\qquad
(\sigma-1)a_s+3\sigma>3.
\]
All the factors in \eqref{eq:irreg-resultant}, except possibly \(a_s\)
and \(P_\sigma(a_s)\), are therefore nonzero. Since
\(F_s(a_s,b_s)=G_s(a_s,b_s)=0\), it follows that
\[
a_s=0
\qquad\text{or}\qquad
P_\sigma(a_s)=0.
\]

We next exclude the first possibility here using \eqref{lovely}. 
When \(a=0\), the two polynomials factor as
follows:
\[
F_s(0,b)
=
(b+3)\bigl(3(\sigma-1)-b\bigr)A_\sigma(b)
\]
and
\[
G_s(0,b)
=
-6(b+3)\bigl(3(\sigma-1)-b\bigr)B_\sigma(b),
\]
where
\begin{equation*}
A_\sigma(b):=-b^4+9(\sigma-1)b^3+27(3\sigma-1)b^2 
-27(2\sigma^4+4\sigma^3-9\sigma+1)b
+81\sigma(\sigma^4-4\sigma^2+3)
\end{equation*}
and
\[
B_\sigma(b)
:=
b^3+9b^2+(27-18\sigma^3)b
+27(\sigma^4-2\sigma^3+1).
\]
(The Mathematica code in Appendix~\ref{s:prop55code} 
verifies these formulas by
computing
$A_\sigma$ and $B_\sigma$ as polynomial quotients.) If \((0,b,3)\) lies in the interior of the Reeb cone, then we read from the defining inequalities that
\[
b+3>0,
\qquad
3(\sigma-1)-b>0.
\]
Thus, if \((0,b,3)\) were a critical point, then we would have that
\[
A_\sigma(b)=B_\sigma(b)=0.
\]
Direct polynomial division yields
\[
A_\sigma(b)
=
(-b+9\sigma)B_\sigma(b)
-9\sigma^3(-2b+3\sigma-6)(-b+6\sigma-6).
\]
(This is verified by the Mathematica code in Appendix~\ref{s:prop55code}.)
It follows that either
\[
b=6(\sigma-1)
\qquad\text{or}\qquad
b=\frac{3\sigma-6}{2}.
\]
The first possibility contradicts the fact that \(b<3(\sigma-1)\). As for the second, direct substitution (which is also verified by the Mathematica code in Appendix~\ref{s:prop55code})
gives us that
\[
B_\sigma\left(\frac{3\sigma-6}{2}\right)
=
\frac{27\sigma^3}{8}\neq0,
\]
a contradiction. Thus, \(a_s\neq0\), and it follows that
$P_\sigma(a_s)=0$.

It remains to show that \(P_\sigma\) has no rational root when
\(\sigma\equiv3\pmod5\). To this end, reducing the coefficients of \(P_\sigma\) modulo \(5\), we find that
\[
P_\sigma(x)
\equiv
x^8-2x^7+x^6+x^5-x^4+2x^3-2x^2+2x+1
\pmod5.
\]
The values of this polynomial at \(x=0,1,2,3,4\) are, respectively,
\[
1,\ 3,\ 3,\ 1,\ 2
\qquad\text{in }\Zbb/5\Zbb.
\]
(Both the reduction of $P_{\sigma}(x)$ modulo $5$ and these five values are
verified by the Mathematica code in
Appendix~\ref{s:prop55code}.)
Thus, the reduction of \(P_\sigma\) modulo \(5\) has no root in
\(\Zbb/5\Zbb\).
Suppose, for sake of a contradiction, that $a_{s}$ is rational, i.e., 
$$a_s=\frac{p}{q}\in\Qbb,\qquad
\gcd(p,q)=1.$$
By the rational root theorem, we see that
$q\mid72(\sigma-1)^3=72s^3$.
Since \(s\equiv2\pmod5\), we have
$72s^3\equiv1\pmod5$,
and so \(5\nmid q\). Multiplying the identity
\(P_\sigma\left(\frac{p}{q}\right)=0\) by \(q^8\) and reducing modulo \(5\) would therefore
imply that $pq^{-1}\in\Zbb/5\Zbb$
is a root of the reduction of \(P_\sigma\), which is impossible. We therefore conclude that $a_s\notin\Qbb$.

Since the normalized Reeb field is of the form 
\(\xi_s=(a_s,b_s,3)\) with $a_{s}$ irrational, the ratio $\frac{a_s}{3}$ is irrational. It follows that no non-zero positive scalar multiple of $\xi_{s}$ will be a lattice point in the Lie algebra. This means that \(\xi_s\) is an irregular Reeb field.
\end{proof}

\subsection{Examples}\label{egs}
We compute the Calabi--Yau Reeb fields and give a distinguished
lattice\linebreak Minkowski decomposition for several members of the
Cho--Futaki--Ono family using the code from Appendix \ref{appendix_hilbert_volume}. Each example proceeds in the following order. We begin by computing the Hilbert basis and Hilbert series, together
with the Minkowski summand cone and its maximal decompositions, using
the Macaulay2 code from Appendix \ref{s:A1} (see Figure \ref{figure_macaulay2_cfo13}). We then obtain the volume function and the Calabi--Yau Reeb field from the Hilbert series output by means of the Mathematica code in Appendix \ref{s:A2} (see Figure \ref{figure_mathematica_cfo13}). 
For practical purposes, we only show the weight matrix \(W\) (whose rows serve as input for
the Mathematica code), the volume function, and the numerical coordinates of the volume minimizer.

\begin{figure} [!htbp]
    \begin{subfigure}[b]{0.45\textwidth}
    \centering
        \resizebox{\linewidth}{!}{
                    \begin{tikzpicture}[ thick,decoration={markings,mark=at position 0.5 with {\arrow{Stealth}}}]
                \draw[step=1cm,gray,dashed] (-1.9,-2.9) grid (6.9,3.9);
                \node at (0,0) [circle, fill, inner sep = 1.5pt] {};

                \draw[postaction={decorate}] (-1,-2) node[below]{\((-1,-2)\)} to node[below right]{\(d_2\)} (0,-1) node[below right]{\((0,-1)\)};
                \draw[postaction={decorate}] (0,-1) to node[below right]{\(d_3\)} (1,3) node[below right]{\( (1,3) \)};
                \draw[postaction={decorate}](1,3) to node[above left]{\(d_5\)} (0,2) node[above left]{\( (0,2) \)};
                \draw[postaction={decorate}] (0,2) to node[above left]{\(d_4\)} (-1,-1) node[above left]{\((-1,-1)\)};
                \draw[postaction={decorate}] (-1,-1) to node[ left]{\(d_1\)} (-1,-2);
                \node at (1,0) {\huge \textbf{=}};
                \node at (2,0) [circle, fill, inner sep = 1.5pt]{};
                \draw[postaction={decorate}] (2,0) -- (3,1);
                \draw[postaction={decorate}] (3,1) -- (2,0);
                \node at (4,0) {\huge \textbf{+}};
                \node at (5,0) [circle, fill, inner sep = 1.5pt]{};
                \draw[postaction={decorate}] (5,0) -- (5,-1);
                \draw[postaction={decorate}] (5,-1) -- (6,3);
                \draw[postaction={decorate}] (6,3) -- (5,0);
                \end{tikzpicture}
        }
        \caption{\( \Pcal^2 = \Lcal^{1} + \Delta^{2} \)}   
    \end{subfigure}
    \begin{subfigure}[b]{0.45\textwidth}
        \centering
        \resizebox{\linewidth}{!}{
                            \begin{tikzpicture}[ thick,decoration={markings,mark=at position 0.5 with {\arrow{Stealth}}}]
                \draw[step=1cm,gray,dashed] (-1.9,-2.9) grid (9.9,7.9);
                                \node at (0,0) [circle, fill, inner sep = 1.5pt] {};

                \draw[postaction={decorate}] (-1,-2) -- (0,-1);
                \draw[postaction={decorate}] (0,-1) -- (1,1);
                \draw[postaction={decorate}] (1,1) -- (2,7);
                \draw[postaction={decorate}] (2,7) -- (1,6);
                \draw[postaction={decorate}] (1,6) -- (0,4);
                \draw[postaction={decorate}] (0,4) -- (-1,-1);
                \draw[postaction={decorate}] (-1,-1) -- (-1,-2);
                \node at (1,0) {\huge \textbf{=}}; 
                \node at (2,0) [circle, fill, inner sep = 1.5pt]{};
                \draw[postaction={decorate}] (2,0) -- (3,1);
                \draw[postaction={decorate}] (3,1)-- (2,0);
                \node at (4,0) {\huge \textbf{+}};
                \node at (5,0) [circle, fill, inner sep = 1.5pt]{};
                \draw[postaction={decorate}] (5,0) -- (6,2);
                \draw[postaction={decorate}] (6,2) -- (5,0);
                \node at (7,0) {\huge \textbf{+}};
                \node at (8,0) [circle, fill, inner sep = 1.5pt]{};
                \draw[postaction={decorate}] (8,0) -- (8,-1);
                \draw[postaction={decorate}] (8,-1) -- (9,5);
                \draw[postaction={decorate}] (9,5) -- (8,0);
            \end{tikzpicture}
        }
        \caption{\( \Pcal^{2,3} = \Lcal^{1} + \Lcal^{2} + \Delta^{3} \)}
    \end{subfigure}

\caption{Lattice Minkowski decompositions of the Cho--Futaki--Ono members $\Pcal^{2}$ and $\Pcal^{2,3}$. The solid circle on each polygon represents the origin in their relative vector space. In Macaulay2, the oriented edges of $\Pcal^{s}$ are labeled as follows: \( d_1 = (0,-1), \; d_2 = (1,1), \; d_3 = (1,s+2), \; d_4 = (-1,-s-1), \; d_5 = (-1,-1)\); see also Figure \ref{figure_macaulay2_cfo13} for the detailed Macaulay2 computation for $\mathcal{P}^{2}$.}
\label{figure_cfo_family_decomposition}
\end{figure}

\begin{example}
Consider the member of the family defined by \(r=1\) and \(s=1\).
Its toric diagram has \(2r+3=5\) vertices and is given by
\[ \Pcal^{1} := \textnormal{Conv} \set{\Spvek{-1;-2}, \Spvek{0;-1}, \Spvek{1;2}, \Spvek{0;1}, \Spvek{-1;-1}}. \] 
Consider the matrix
\[
A=
\begin{pmatrix}
1 & 0 \\
-2 & 1
\end{pmatrix}
\in\operatorname{SL}_2(\mathbb{Z}).
\]
Thus, \(A\) acts on column vectors by
\[
A(x,y)=(x,y-2x).
\]
A direct calculation gives
\begin{align*}
A\Spvek{-1;-2}&=\Spvek{-1;0},&
A\Spvek{0;-1}&=\Spvek{0;-1},&
A\Spvek{1;2}&=\Spvek{1;0},\\
A\Spvek{0;1}&=\Spvek{0;1},&
A\Spvek{-1;-1}&=\Spvek{-1;1}.
\end{align*}
Consequently,
\[
A(\mathcal{P}^1)
=
\textnormal{Conv}
\set{
\Spvek{-1;1},
\Spvek{-1;0},
\Spvek{0;-1},
\Spvek{1;0},
\Spvek{0;1}
}
=
\mathcal{Q}_4.
\]
Thus, \(\mathcal{P}^1\) and \(\mathcal{Q}_4\) are lattice-equivalent, and the corresponding toric Calabi--Yau cones are \(T_{\mathbb{C}}\)-equivariantly biholomorphic. After the corresponding change of lattice basis, the volume function and Reeb minimizer are those computed in Example \ref{ex:dp2reeb}.
The distinguished lattice Minkowski decomposition of $\Pcal^{1}$ is given by
\[ \Pcal^{1} = \Lcal^{1} + \Delta^{1}, \]
where \( \Lcal^{1} = \textnormal{Conv}\set{\Spvek{0;0},\Spvek{1;1}} \) and \( \Delta^{1} = \textnormal{Conv} \set{\Spvek{0;0},\Spvek{0;1},\Spvek{1;3}} \). 
\end{example}

\begin{example} \label{example_cfo_13}
Consider the Cho--Futaki--Ono member with parameters \(r=1\) and
\(s=2\). Its toric diagram has \(2r+3=5\) vertices and is given by
\[ \Pcal^{2} = \textnormal{Conv} \set{\Spvek{-1;-2}, \Spvek{0;-1}, \Spvek{1;3}, \Spvek{0;2}, \Spvek{-1;-1}}. \]
Using the code in Appendix \ref{s:A1}, we find that the Hilbert basis has \( 10 \) elements, forming the columns of the weight matrix
\[ W = \left(\!\begin{array}{cccccccccc}
-1&1&-4&0&2&-3&1&3&-2&-1\\
0&-1&1&0&-1&1&0&-1&1&1\\
1&2&1&1&2&1&1&2&1&1
\end{array}\!\right) = \begin{pmatrix}
    e_1 \\
    e_2 \\
    e_3
\end{pmatrix}.
\]
The normalized volume is given by 
\[a_0(\xi) = \frac{-a^2 + 2 a b - b^2 + 20 a c - 2 b c + 
 26 c^2}{(a - b - c) (4 a - b - c) (a + c) (a - b + 2 c) (3 a - b + 
   2 c)}. \] 
Minimizing this function under the constraints \( \xi = a e_1 + b e_2 + c e_3 > 0 \) and $c = 3$ yields the numerical approximation
\[
a = -0.2119737845\dots, \quad
b = 1.244664220\dots, \quad
a_0(\xi_0) = 0.1787519891\dots \;.
\]
By Proposition \ref{prop:irreg}, the first coordinate $a$ is irrational. The numerical computation of the minimizer is illustrated in Figure \ref{figure_mathematica_cfo13}.
The distinguished lattice Minkowski decomposition is given by 
\[ \Pcal^{2} =  \Lcal^{1} + \Delta^{2}, \]
with \( \Lcal^{1} = \textnormal{Conv} \set{\Spvek{0;0},\Spvek{1;1}} \) and \( \Delta^{2}  = \textnormal{Conv} \set{\Spvek{0;0},\Spvek{0;1},\Spvek{1;4}} \). 
\end{example}

\begin{example}
For \( r = 1, s = 3 \), we have 
\begin{align*}
\Pcal^{3} = \textnormal{Conv}\set{\Spvek{-1;-2},\Spvek{0;-1},\Spvek{1;4},\Spvek{0;3},\Spvek{-1;-1}},   
\end{align*}
with the following weight matrix $W$ representing the \( 12 \) elements of the Hilbert basis: 
\begin{align*}
W = \left(\!\begin{array}{cccccccccccc}
-1&1&-5&0&1&2&-4&3&-3&4&-2&-1\\
0&-1&1&0&0&-1&1&-1&1&-1&1&1\\
1&3&1&1&1&3&1&3&1&3&1&1
\end{array}\!\right).   
\end{align*}
The normalized volume is given by
\[ a_0(\xi) = \frac{-a^2 + 2 a b - b^2 + 50 a c - 2 b c + 
 63 c^2}{(a - b - c) (5 a - b - c) (a + c) (a - b + 3 c) (4 a - b + 
   3 c)},\]
   which is minimized under the constraint $c=3$ by 
\[ a = -0.1547965799\dots, \quad b = 2.785162197\dots, \quad a_0(\xi_0) = 0.1379166501\dots. \] 
The distinguished lattice Minkowski decomposition is
given by \[ \Pcal^{3} =  \Lcal^{1} + \Delta^{3}, \]
with \( \Lcal^{1} = \textnormal{Conv} \set{\Spvek{0;0},\Spvek{1;1}} \) and \( \Delta^{3}  = \textnormal{Conv} \set{\Spvek{0;0},\Spvek{0;1},\Spvek{1;5}} \). 
\end{example}

\section{Theorem \ref{mtheorem_quadrilateral+segment}
and Example \ref{ex:gmsw21+segment}}\label{s:thma}

In this section, we give the proof of Theorem \ref{mtheorem_quadrilateral+segment}, after which we provide the details of Example \ref{ex:gmsw21+segment}.

\subsection{Proof of Theorem
\ref{mtheorem_quadrilateral+segment}}\label{s:6.1}

Part \ref{mtheorem_diagram} follows from Lemma
\ref{lem:sumtoricdiag}. We now prove part
\ref{mtheorem_smoothable}.

Suppose first that \(\Qcal^{p,q}\) is a toric diagram. Since
\(\Ycal^{p,q}\) is a quadrilateral which is not a parallelogram, it
has no non-trivial lattice Minkowski decomposition by Lemma
\ref{lem:practical-crit}. Moreover, \(\Lcal\) is a primitive lattice
segment by Lemma \ref{lem:sumtoricdiag}. Thus,
\(
\Qcal^{p,q}
=
\Ycal^{p,q}+\Lcal
\)
is a maximal Minkowski decomposition into lattice summands. By
Theorem \ref{theorem_deformation_classification}, it determines a
one-dimensional irreducible component of the reduced versal base, and the restriction of the versal family to this component has affine
toric total space. Since one summand is the
quadrilateral \(\Ycal^{p,q}\), Lemma
\ref{lemma_smoothing_component} shows that this component is not a
smoothing component.

This proves the first assertion of part
\ref{mtheorem_smoothable}. It remains to determine when
\(\Qcal^{p,q}\) admits an additional maximal Minkowski decomposition
into two lattice triangles. Such a decomposition determines another
one-dimensional irreducible component of the reduced versal base. The
following lemma gives necessary and sufficient conditions for the
existence of this additional decomposition. After proving the lemma,
we determine whether the corresponding component is a smoothing
component by examining whether its triangle summands are unimodular.

\begin{lemma}\label{lem:gmsw-extrarigid}
Let \(\Lcal\) be a lattice segment such that
\(\Ycal^{p,q}+\Lcal\) is a toric diagram of an isolated Gorenstein
singularity. Then, in addition to the decomposition
\(
\Ycal^{p,q}+\Lcal,
\)
the polygon \(\Ycal^{p,q}+\Lcal\) admits a maximal Minkowski
decomposition into two lattice triangles if and only if either:
\begin{itemize}
\item
$q=1$ and $\Lcal$ is a translate of 
$\textnormal{Conv}
\set{\Spvek{0;0},\Spvek{1;1}}$,
or
\item
$p-q$ is odd and $\Lcal$
is a translate of $\textnormal{Conv}
\set{\Spvek{0;0},\Spvek{-p+q+2;-p+q}}$.
\end{itemize}
\end{lemma}

\begin{proof}
Recall from the proof of Proposition \ref{prop:gmsw_trivial} that, for
a given pair \((p,q)\), the oriented edges of \(\Ycal^{p,q}\) are
given by
\[
d_1=(1,0),\qquad
d_2=(p-1,p),\qquad
d_3=(-q-1,-q),\qquad
d_4=(-p+q+1,-p+q).
\]
Let \(d\) denote the suitably oriented edge built from \(\Lcal\).

First suppose that \(\Ycal^{p,q}+\Lcal\) has an additional maximal
Minkowski decomposition into two lattice triangles. Then \(d\) must
satisfy one of the conditions \ref{12}--\ref{14} of Lemma
\ref{cor:sumtoricdiag}. Clearly, we have that
\begin{enumerate}
\item \(d_1+d_2=d\) if and only if \(d=(p,p)\);
\item \(d_1+d_3=-d\) if and only if \(d=(q,q)\);
\item \(d_1+d_4=d\) if and only if
\(
d=(-p+q+2,-p+q).
\)
\end{enumerate}
Since \(\Ycal^{p,q}+\Lcal\) is a toric diagram of an isolated
Gorenstein singularity, \(\Lcal\) has no interior lattice points by
Lemma \ref{lem:sumtoricdiag}. This excludes the case
\(d=(p,p)\) because \(p\geq2\), and the case \(d=(q,q)\) when
\(q\geq2\). The second possibility therefore yields \(q=1\).

For the third possibility, we have that
\(
\gcd(-p+q+2,-p+q)
=
\gcd(p-q,2).
\)
Thus, the vector
\(
(-p+q+2,-p+q)
\)
is primitive if and only if \(p-q\) is odd. This proves the necessity of the two conditions in the statement, up
to translation. Conversely, after translating \(\Lcal\), suppose first that \(q=1\) and
\(
\Lcal
=
\textnormal{Conv}
\set{\Spvek{0;0},\Spvek{1;1}}.
\)
Taking \(d=(1,1)\), we have that
\(
d_1+d_3=-d.
\)
Lemma \ref{cor:sumtoricdiag} therefore gives the required maximal Minkowski decomposition into two lattice triangles. For the second case, let $h = p-q$ be odd and set $d = (2-h, -h)$. Then $d$ is primitive because $\gcd(2-h, h) = \gcd(2, h) = 1$, and $d_1 + d_4 = d$. Lemma \ref{cor:sumtoricdiag} again gives the required maximal Minkowski decomposition into two lattice triangles. This proves the converse. The lemma now follows.
\end{proof}

To conclude the proof of part \ref{mtheorem_smoothable}, we must
determine which of the additional components classified in Lemma
\ref{lem:gmsw-extrarigid} are smoothing components. Each additional
maximal Minkowski decomposition found in this lemma is a sum of two lattice
triangles. By Theorem
\ref{theorem_deformation_classification}, it determines a
one-dimensional irreducible component of the reduced versal base. Consider first the case \( q=1, d=(1,1). \) The two triangle summands have oriented edge sets  $\set{d_1,d_3,d}$ and $\set{d_2,d_4,-d}$.  Their normalized lattice areas are  $\left|\det(d_1,d_3)\right|=1$ and
$\left|\det(d_2,d_4)\right|=1$, respectively. Thus, both triangles are unimodular, and so Lemma \ref{lemma_smoothing_component} tells us that the corresponding component is a smoothing component. The existence of AC Calabi--Yau metrics in every K\"ahler class on its smooth fibers now follows from \cite[Theorem 4.3]{CH24}. 

Next consider the second case, where \( d=(-p+q+2,-p+q). \) The two triangle summands have oriented edge sets $\set{d_1,d_4,-d}$ and $\set{d_2,d_3,d}$. Their normalized lattice areas are $\left|\det(d_1,d_4)\right|=p-q$ and $\left|\det(d_2,d_3)\right|=p+q$. Since \(p>q>0\), we have \(p+q\geq3\), and so the second triangle is not unimodular. Lemma \ref{lemma_smoothing_component} therefore shows that the corresponding irreducible component of the reduced versal base is not a smoothing component. This completes the proof of part \ref{mtheorem_smoothable}.

To obtain the vertices $(a_0,\dots,a_5)$ of $\mathcal{Q}^{p,q}$, we walk along the edges of $\Ycal^{p,q}$ and $\set{d,-d}$ in a suitable order, starting from the point $a_0 =(0,0)$. For example, after adding the segment $\Lcal = \textnormal{Conv}\set{\Spvek{0;0},\Spvek{1;1}}$ to $\Ycal^{p,1}$, the resulting polygon is given by
\begin{equation*}
\mathcal{Q}^{p,1} = \textnormal{Conv} \set{\Spvek{0;0},\Spvek{1;0},\Spvek{2;1},\Spvek{p+1;p+1},\Spvek{p-1;p},\Spvek{p-2;p-1}}.
\end{equation*}
The first vertex is prescribed by $a_0 = (0,0)$, and the subsequent ones are obtained by adding successively the oriented edges: $a_1 = a_0 + d_1 = (1,0)$, $a_2 = a_1 + d = (2,1)$, $a_3 = a_2 + d_2 = (2,1) + (p-1,p) = (p+1,p+1),$ and so on. 
Likewise, adding $\Lcal = \textnormal{Conv}\set{\Spvek{0;0},\Spvek{-p+q+2;-p+q}}$ to $\Ycal^{p,q}$ yields
\begin{equation*}
\mathcal{Q}^{p,q} = \textnormal{Conv} \set{\Spvek{0;0},\Spvek{-p+q+2;-p+q},\Spvek{-p+q+3;-p+q},\Spvek{q+2;q},\Spvek{p;p},\Spvek{p-q-1;p-q}}.
\end{equation*}
The resulting polygons in both cases are represented in Figure \ref{figure_quadrilateral+segment} for $(p,q) = (2,1)$. 

\begin{figure}[H] 
    \begin{subfigure}[b]{0.45\textwidth}
        \centering
        \resizebox{\linewidth}{!}{
                            \begin{tikzpicture} [thick,decoration={markings,mark=at position 0.5 with {\arrow{Stealth}}}]  
                \draw[step=1cm,gray,dashed] (-1.9,-3.9) grid (8.9,2.9);
                                \node at (0,0) [circle, fill, inner sep = 1.5pt] {};
                \draw[postaction={decorate}]  (0,-1) -- (1,0); \draw[postaction={decorate}](1,0) -- (2,2);
                \draw[postaction={decorate}] (2,2)-- (0,1);
                \draw[postaction={decorate}] (0,1)-- (-1,0); 
                \draw[postaction={decorate}] (-1,0)-- (-1,-1); \draw[postaction={decorate}] (-1,-1) -- (0,-1);
            \node at (2.5,0) {\huge \textbf{=}};   
            \node at (4,0) [circle, fill, inner sep = 1.5pt]{};
            \draw [postaction={decorate}] (3,-1) -- (4,-1);
            \draw [postaction={decorate}] (4,-1) -- (5,1);
            \draw [postaction={decorate}](5,1) -- (3,0);
            \draw [postaction={decorate}] (3,0) -- (3,-1);
             \node at (6,0) {\huge \textbf{+}};
             \node at (7,0) [circle, fill, inner sep = 1.5pt]{};
             \draw [postaction={decorate}] (7,0) -- (8,1);
             \draw [postaction={decorate}] (8,1) -- (7,0);
            
             \node at (2.5, -2) {\huge \textbf{=}};
              \node at (3,-3) [circle, fill, inner sep = 1.5pt]{};
             \draw[postaction={decorate}] (3,-3) -- (4,-3);
             \draw[postaction={decorate}] (4,-3) -- (5,-2);
             \draw[postaction={decorate}] (5,-2) -- (3,-3);
             \node at (6,-2) {\huge \textbf{+}};
              \node at (7,-2) [circle, fill, inner sep = 1.5pt]{};
             \draw[postaction={decorate}] (7,-2) -- (7,-3);
             \draw[postaction={decorate}] (7,-3) -- (8,-1);
             \draw[postaction={decorate}] (8,-1) -- (7,-2);
            \end{tikzpicture}
        }
        \caption{$\mathcal{Q}_{+}^{2,1} = \Ycal^{2,1} + \textnormal{Conv}\set{\Spvek{0;0},\Spvek{1;1}}$}
     \end{subfigure}   
    \begin{subfigure}[b]{0.45\textwidth}
        \centering
        \resizebox{\linewidth}{!}{
                            \begin{tikzpicture} [ thick,decoration={markings,mark=at position 0.5 with {\arrow{Stealth}}}] 
                \draw[step=1cm,gray,dashed] (-1.9,-3.9) grid (8.9,2.9);
                                \node at (0,0) [circle, fill, inner sep = 1.5pt] {};
                \draw[postaction={decorate}] (0,-1) -- (1,-1);
                \draw[postaction={decorate}] (1,-1) -- (2,1);
                \draw[postaction={decorate}] (2,1) -- (1,2);
                \draw[postaction={decorate}] (1,2) -- (-1,1);
                \draw[postaction={decorate}] (-1,1) -- (-1,0);
                \draw[postaction={decorate}] (-1,0) -- (0,-1);
            \node at (2.5,0) {\huge \textbf{=}};   
            \node at (4,0) [circle, fill, inner sep = 1.5pt]{};
            \draw[postaction={decorate}] (3,-1) -- (4,-1);
            \draw[postaction={decorate}] (4,-1) -- (5,1);
            \draw[postaction={decorate}](5,1) -- (3,0);
            \draw[postaction={decorate}] (3,0) -- (3,-1);
             \node at (6,0) {\huge \textbf{+}};
             \node at (7,0) [circle, fill, inner sep = 1.5pt]{};
             \draw[postaction={decorate}] (7,0) -- (8,-1);
             \draw[postaction={decorate}] (8,-1) -- (7,0);

             \node at (2.5, -2) {\huge \textbf{=}};
              \node at (4,-2) [circle, fill, inner sep = 1.5pt]{};
             \draw[postaction={decorate}] (3,-2) -- (4,-3);
             \draw[postaction={decorate}] (4,-3) -- (5,-1);
             \draw[postaction={decorate}] (5,-1) -- (3,-2);
             \node at (6,-2) {\huge \textbf{+}};
              \node at (7,-2) [circle, fill, inner sep = 1.5pt]{};
             \draw[postaction={decorate}] (7,-2) -- (7,-3);
             \draw[postaction={decorate}] (7,-3) -- (8,-3);
             \draw[postaction={decorate}] (8,-3) -- (7,-2);
            \end{tikzpicture}
        }
        \caption{$\mathcal{Q}^{2,1}_{-} = \Ycal^{2,1} + \textnormal{Conv}\set{\Spvek{0;0},\Spvek{1;-1}}$}
    \end{subfigure}
\caption{Two additional maximal Minkowski decompositions into lattice
triangles obtained by adding primitive lattice segments to the
non-smoothable toric diagram \(\Ycal^{2,1}\). In
\textnormal{(a)}, both triangle summands are unimodular, so the
corresponding additional one-parameter component is a smoothing
component. In \textnormal{(b)}, the two triangle summands have
normalized lattice areas \(1\) and \(3\), so the corresponding
additional component is not a smoothing component.}
\label{figure_quadrilateral+segment}
\end{figure}

\subsection{Example \ref{ex:gmsw21+segment}}\label{s:6.2}
We now justify Example \ref{ex:gmsw21+segment}. Up to translation, the sums 
\[
\Ycal^{2,1}
+
\textnormal{Conv}\set{\Spvek{0;0},\Spvek{1;1}},
\qquad
\Ycal^{2,1}
+
\textnormal{Conv}\set{\Spvek{0;0},\Spvek{1;-1}}
\] are given by
\begin{equation*}
\mathcal{Q}_{+}^{2,1} = \textnormal{Conv} \set{\Spvek{0;-1},\Spvek{1;0},\Spvek{2;2},\Spvek{0;1},\Spvek{-1;0},\Spvek{-1;-1}}
\end{equation*}
and
\begin{equation*}
\mathcal{Q}^{2,1}_{-} = \textnormal{Conv} \set{\Spvek{0;-1},\Spvek{1;-1},\Spvek{2;1},\Spvek{1;2},\Spvek{-1;1},\Spvek{-1;0}},
\end{equation*}
respectively; see Figure \ref{figure_quadrilateral+segment}. Note that they have the same number of vertices, but not the same number of interior lattice points. This means that they define non-$T_{\C}$-equivariantly biholomorphic toric Calabi--Yau cones. 

We next compute the Hilbert bases and toric ideals of both cones using the Macaulay2 code found in Appendix \ref{s:A1}. The Hilbert bases are represented by the columns of the matrices
\begin{equation*}
 W = \begin{pmatrix}
 -2 & -1 & -1 &0 &0 &0 & 1 & 1&1 \\
 1 & 0 & 1 & -1 & 0 & 1 & -2 & -1 & 0 \\
 2 & 2 & 1 & 2 & 1 & 1 & 2 & 1 & 1 
 \end{pmatrix}
\end{equation*}
and 
\begin{equation*}
W =  \begin{pmatrix}
-2 & -1 & -1 &-1 & 0 & 0 & 0 & 1 & 1 & 1 & 1 \\
1 & -1 & 0 & 1 & -1 & 0 & 1 & -2 & -1 & 0  & 1\\
3 & 3 & 2 & 2 & 2 & 1 & 1 & 3 & 2 & 1 & 1 
 \end{pmatrix},
 \end{equation*}
respectively. The corresponding toric cones are then cut out by the following quadrics and cubics: 
\begin{align*}
&z_{5}^{2}-z_{3}z_{8},\,\,z_{6}z_{8}-z_{5}z_{9},\,\,z_{5}z_{6}-z_{3}z_{9},\,\,z_{5}z_{7}-z_{4}z_{8},\,\,z_{3}z_{7}-z_{2}z_{8}, \\&z_{4}z_{5}-z_{2}z_{8},\,\,z_{2}z_{5}-z_{1}z_{8},\,\,z_{3}z_{4}-z_{1}z_{8},\,\,z_{2}z_{3}-z_{1}z_{5},\,\,z_{5}z_{8}^{2}-z_{7}z_{9},\\&z_{3}z_{8}^{2}-z_{4}z_{9},\,\,z_{3}z_{5}z_{8}-z_{2}z_{9},\,\,z_{3}^{2}z_{8}-z_{1}z_{9},\,\,z_{6}z_{7}-z_{4}z_{9},\,\,z_{4}z_{6}-z_{2}z_{9},\\
&z_{2}z_{6}-z_{1}z_{9},\,\,z_{3}^{2}z_{5}-z_{1}z_{6},\,\,z_{4}^{2}-z_{2}z_{7},\,\,z_{2}z_{4}-z_{1}z_{7},\,\,z_{2}^{2}-z_{1}z_{4},
\end{align*}
and
\begin{align*}
&z_{7}z_{10}-z_{6}z_{11},\,\,z_{6}z_{9}-z_{5}z_{10},\,\,z_{5}z_{7}-z_{3}z_{10},\,\,z_{6}^{3}-z_{3}z_{10},\,\,z_{4}z_{6}-z_{3}z_{7},\,\,z_{6}^{2}z_{10}-z_{5}z_{11},\,\,z_{4}z_{10}-z_{3}z_{11}, \\
&\,\,z_{7}z_{9}-z_{5}z_{11},\,\,z_{6}^{2}z_{7}-z_{3}z_{11},\,\,z_{3}z_{5}-z_{2}z_{6},\,\,z_{6}z_{10}^{2}-z_{9}z_{11},\,\,z_{3}z_{9}-z_{2}z_{10},\,\,z_{6}z_{8}-z_{5}z_{9},\,\,z_{6}z_{7}^{2}-z_{4}z_{11},\\
&\,\,z_{5}z_{6}^{2}-z_{2}z_{10}, \,\,z_{3}z_{6}^{2}-z_{2}z_{7},\,\,z_{4}z_{5}-z_{2}z_{7},\,\,z_{3}z_{4}-z_{1}z_{6},\,\,z_{3}z_{6}z_{10}-z_{2}z_{11}, \,\,z_{9}^{2}-z_{8}z_{10},\,\,z_{4}z_{9}-z_{2}z_{11},\\
&\,\,z_{7}z_{8}-z_{5}z_{6}z_{10},\,\,z_{3}z_{6}z_{7}-z_{1}z_{10},\,\,z_{4}^{2}-z_{1}z_{7},\,\,z_{5}z_{10}^{2}-z_{8}z_{11},\,\,z_{3}z_{8}-z_{2}z_{9},\,\,z_{3}z_{7}^{2}-z_{1}z_{11},\,\,z_{5}^{2}z_{6}-z_{2}z_{9}, \\
&\,\,z_{1}z_{5}-z_{3}^{2}z_{6},\,\,z_{2}z_{4}-z_{3}^{2}z_{6},\,\,z_{1}z_{9}-z_{3}^{2}z_{10},\,\,z_{4}z_{8}-z_{2}z_{6}z_{10},\,\,z_{5}^{3}-z_{2}z_{8},\,\,z_{1}z_{2}-z_{3}^{3},\,\,z_{1}z_{8}-z_{2}z_{3}z_{10},
\end{align*}
respectively. With the Hilbert series output from the Macaulay2 code, we can now compute the corresponding volume functions by running the Mathematica code in Appendix \ref{s:A2}. The results are, respectively, as follows:
\begin{equation*}
a_0(\xi) = \frac{-2 a^3 + 2 a^2 b + 2 a b^2 - 2 b^3 - 9 a^2 c + 16 a b c - 9 b^2 c + 
 4 a c^2 + 4 b c^2 + 
 19 c^3}{(2 a - b - 2 c) (a - b - c) (a + c) (a - b + c) (b + 
   c) (a - 2 b + 2 c)}   
\end{equation*}
and 
\begin{equation*}
a_0(\xi) = \frac{-2 (a^3 + 3 a^2 b + 3 a b^2 + b^3 + 5 a^2 c - 8 a b c + 5 b^2 c - 
    30 a c^2 - 30 b c^2 - 54 c^3)}{(2 a - b - 3 c) (a + b - 3 c) (a +
     c) (b + c) (a + b + c) (a - 2 b + 3 c)}    .
\end{equation*}
The output of the code in Appendix \ref{s:A2} also shows that the volume functions are minimized under the constraint $c = 3$ by irrational Reeb fields with numerical coordinates
\begin{equation*}
a = 0.979 \dots, \quad b = 0.979\dots, \qquad\qquad a = 1.38\dots,\quad b = 1.38\dots,   
\end{equation*}
respectively. The irrationality of these coordinates can be checked using Mathematica, as demonstrated in Figure \ref{figure_mathematica_cfo13} for the Cho--Futaki--Ono member $\Pcal^{2}$.

\newpage
\appendix
\section{Macaulay2 and Mathematica codes} \label{appendix_hilbert_volume}

\subsection{Getting Started with the Macaulay2Web Interface}\label{s:started}
A freely available Macaulay2 interface can be accessed via
\begin{center}
\href{https://www.unimelb-macaulay2.cloud.edu.au/#editor}{\textnormal{https://www.unimelb-macaulay2.cloud.edu.au/\#editor}}.
\end{center}
A detailed tutorial can be found at
\begin{center}
\href{https://www.unimelb-macaulay2.cloud.edu.au/#tutorial}{\textnormal{https://www.unimelb-macaulay2.cloud.edu.au/\#tutorial}}.
\end{center}
For convenience, we only provide an outline of the relevant sections. The web interface above features two half-screens separated by a movable blue line, which can be dragged with the cursor to expand or shrink either side.
\begin{itemize} 

\item 
On the right-hand side is the \emph{Terminal} tab, where we typically enter short Macaulay2 command sequences--such as defining an object or testing a function on an example--and press ``Enter'' to execute them. Each input line begins with \verb|i| followed by a number, which tracks the input sequence. Correspondingly, each output line begins with \verb|o| followed by its output number. For example, if the first command entered is:

\noindent\rule{8cm}{0.4pt}
\begin{verbatim}
i1: M = matrix{{-1,0},{0,1}}
\end{verbatim}
\noindent\rule{8cm}{0.4pt}

\noindent then after pressing ``Enter'', the \emph{Terminal} tab will display:

\noindent\rule{8cm}{0.4pt}
\begin{lstlisting}[]
o1 = $\begin{pmatrix}
-1 & 0 \\
0 & 1 
\end{pmatrix}$
o1 : Matrix $\Zbb^2 \longleftarrow \Zbb^2$ 
\end{lstlisting}
\noindent\rule{8cm}{0.4pt}

\noindent \item The left-hand side displays the \emph{Editor} tab, where we write longer pieces of code--such as defining new Macaulay2 functions, as we will do below. The four buttons at the top right of the \emph{Editor} tab execute specific commands on the code. Their functions are, from left to right:
\begin{itemize}
    \item Run all editor code. 
    \item Run selected code, or current line. 
    \item Load a file to the editor. 
    \item Save the editor file locally. 
\end{itemize}
The ``Run all editor code'' button is highlighted in Figure \ref{figure_macaulay2_cfo13}. 
\end{itemize}

\subsection{Computing the Hilbert series, Minkowski summand cone, and
maximal decompositions}\label{s:A1}
For a given three-dimensional toric Calabi--Yau cone, the associated polygonal toric diagram $\Pcal \subset \Rbb^2$ contains essentially all the relevant information. Thus, a natural input of the code is the toric diagram represented by a $(2 \times N)$-matrix, the columns of which are the $N$ lattice vertices of $\Pcal$. The Macaulay2 output consists of the multivariate Hilbert series of the
associated toric cone, the Minkowski summand cone of \(\Pcal\), its
extremal summands, and the maximal decompositions of \(\Pcal\) into
these summands.

On one hand, the code to compute the Hilbert series, given the polytope $\Pcal$, can be written solely with the Quasidegrees package. On the other hand, 
the Minkowski summand cone and its maximal decompositions can be
computed using the function \verb|minkSummandCone|, available after loading the Polyhedra package. Since we want a code as versatile as possible, we will define a novel two-in-one function in Macaulay2 that inputs a $(2 \times N)$-matrix and outputs both the Hilbert series and the Minkowski summand cone,
together with its extremal summands and maximal decompositions. This can be carried out in the following two steps (see the \emph{Editor} tab on the left-hand side of Figure \ref{figure_macaulay2_cfo13}): 
\begin{itemize}
\item Copy and paste the lines of code below (except the texts in gray following a ``\#'') into the \emph{Editor} tab (the left-hand side of the web interface).\\
\\

\noindent\rule{8cm}{0.4pt}
    \begin{Verbatim}[commandchars=\\\{\}]
    loadPackage "Polyhedra"
    loadPackage "Quasidegrees" 
    \textcolor{gray}{#Hilbert basis of a Gorenstein cone}
    \textcolor{gray}{#M is the (2 x N) matrix whose columns are the N lattice vertices of P}
    \textcolor{gray}{#THMS is abbreviation for ``\textbf{t}oric \textbf{H}ilbert series and \textbf{Mink}owski summands''}
    THMS = method()
    THMS := M -> 
    (
    P := convexHull M;
    m := numColumns M;
    ones := matrix table(1, m, (i,j) -> 1);
    MC := M || ones;
    C := posHull MC;
    \textcolor{gray}{#Cv is the moment cone}
    Cv := dualCone C;
    HB := hilbertBasis Cv;
    \textcolor{gray}{#grading matrix whose rows generate the Reeb cone}
    W := matrix \{HB\};
    N := #HB;
    
    \textcolor{gray}{#Hilbert series}
    R := toGradedRing(W, QQ[z_1..z_N]);
    I := toricIdeal(W,R);
    HWI := hilbertSeries(I, Reduce => true);

    \textcolor{gray}{#Minkowski decomposition}
    (CP, LP, MP) := minkSummandCone P;
    apply(values LP, vertices);
    
    \textcolor{gray}{#Output}
    << I << endl
    << W << endl
    << HWI << endl
    << texMath I << endl
    << texMath HWI << endl
    << rays CP << endl
    << apply(values LP, vertices) << endl
    << MP << endl
    )
    \end{Verbatim}
\noindent\rule{8cm}{0.4pt}
   
\noindent This defines a new function in Macaulay2 named ``THMS'', which takes as input a $(2 \times N)$-matrix and returns the relevant information.
\item 
To run the code, click on ``Run all editor code''--the leftmost button in the top-right corner of the \emph{Editor} tab (see Figure \ref{figure_macaulay2_cfo13}, where it is highlighted). Once executed, the function THMS will be defined and ready for use.
\end{itemize}
Now that the new function serving our purpose has been defined, we can apply it in the following way.
\begin{itemize}
\item Form a matrix $M$, the columns of which are the vertex coordinates of $\Pcal \subset \Rbb^2$.
\item Enter this matrix into an input line in the \emph{Terminal} tab (on the right-hand side of the web interface) by typing the following command, followed by ``Enter'': 

\noindent\rule{8cm}{0.4pt}
\begin{Verbatim}
M = matrix{{first row entries}, {second row entries}} 
\end{Verbatim}
\noindent\rule{8cm}{0.4pt}

\noindent Here, the row entries are separated by commas (see the first bullet point of Appendix \ref{s:started}).

\item Type ``THMS $M$'' into the \emph{Terminal} tab (right-hand side of the web interface), then press ``Enter''.
\end{itemize}

Again, we refer the reader to Figure \ref{figure_macaulay2_cfo13} for a code test on the Cho--Futaki--Ono toric diagram $\Pcal^{2}$ shown in Figure \ref{figure_cfo_family}. To be precise about the output:
\begin{itemize}
\item 

The first three lines of the output display the homogeneous toric ideal \( I \), the Hilbert basis matrix \( W \), and the \( W \)-graded Hilbert series \( H_{W,I} \), respectively. Recall that all linear relations among elements of the Hilbert basis are encoded in the toric ideal $I$. Specifically, \( I \) contains a binomial \( Z_i^a Z_j^b - Z_k^c Z_l^d \) if and only if \( a\alpha_i + b\alpha_j = c\alpha_k + d\alpha_l \), where \( \alpha_i \) denotes the \( i \)-th column of \( W \). The fourth and fifth lines are the output for \( I \) and \( H_{W,I} \) as \TeX-formatted strings.
\item The sixth line \( CP \) denotes the Minkowski summand cone of $\Pcal$, with \verb|rays CP| displaying the minimal generators of $CP$ in columns. 
\item The seventh line displaying the hash table \(LP\) encodes the maximal
Minkowski summands corresponding to the primitive generators of the
Minkowski summand cone \(C(\Pcal)\). Thus, \(LP\) contains one
representative for each extremal ray of \(C(\Pcal)\). Each summand \( \mathcal{P}_i \) is represented by a matrix whose columns correspond to the vertices of \( \mathcal{P}_i \). In fact, \( LP \) is in bijection with the minimal set of generators of \( C(\mathcal{P}) \) via the correspondence described in Lemma \ref{lemma_minkowski_cone}; see Example \ref{ex:dp1mink} and Remark \ref{rmk:toricdiagramminkdec} for how to recover a summand from a given generator.
\item The final line displaying the columns of the matrix \(MP\) represents
the maximal decompositions of \(\Pcal\) into the maximal Minkowski
summands contained in \(LP\).
Specifically, let $LP = \set{\Pcal_1, \dots, \Pcal_K}$. Then, given a column $C_j =(C_{1j}, \dots, C_{Kj})^T$ of $MP$, the corresponding decomposition is $\Pcal = \sum_{i=1}^K C_{ij} \Pcal_i$. 
\end{itemize}

These maximal decompositions should not be confused with all maximal
Minkowski decompositions into lattice summands in the sense of
Theorem \ref{thm:minkschemecomp}. A maximal decomposition into lattice
summands may contain a lattice polytope which is decomposable into
non-lattice Minkowski summands, and such a decomposition need not be
displayed by \(MP\). The maximal decompositions into lattice summands
are determined separately using the zero-sum partitions described in
Remark \ref{rmk:toricdiagramminkdec}.

\begin{remark} \label{rmk:orientation}
The reader should be aware that the default Macaulay2 labeling and orientation of the edges of 
$\Pcal$ may differ from ours. However, this only affects the presentation; the Minkowski summands remain the same up to affine transformation. Consequently, the essential information (i.e., whether the polygon is lattice-decomposable,
the maximal decompositions returned by the Minkowski summand cone
calculation, etc.) remains unchanged.
\end{remark}

\subsection{Computing the index character and volume minimizer}\label{s:A2} 
The index character and volume minimizer can be computed using Mathematica. The code we present here works for Mathematica 13.0. An issue likely encountered here is how one can effectively translate the output of the Hilbert series in Macaulay2 (which usually spreads out) into Mathematica syntax. To address this, we first use the \verb|texMath| function in Macaulay2 to display the output in \TeX-form, then use the function \verb|ToExpression| in Mathematica to convert the \TeX-form output into Mathematica syntax. 

To process the above data, first run the relevant code in Macaulay2, then copy the \TeX-formatted output of \verb|HWI| from the Macaulay2 interface 
(line $5$ of the output)
into the argument of Mathematica’s \verb|ToExpression| function. Below are the corresponding commands for Mathematica, along with a more detailed tutorial (see Figure \ref{figure_mathematica_cfo13} for an example). Note that the hotkey to execute a Mathematica command is ``Shift+Enter''.

\begin{itemize}
\item Copy the \TeX-form output of \verb|HWI| by first double-clicking on the fifth output line in Macaulay2 (to select the line as a whole), right-clicking the selected line, and then selecting ``Copy''.
\item Next, open the Mathematica interface and type the following in the first input line:
\begin{verbatim}
ToExpression["texMath HWI", TeXForm]    
\end{verbatim}
Now paste the line copied from the Macaulay2 output into \verb|texMath HWI| and press ``Shift + Enter'' to execute the command in Mathematica. 
You will be prompted that ``Your text contains quotes or backslashes. Do you want to escape these so they appear verbatim in the string?''. Click on ``Yes''. The output will then be the toric Hilbert series in Mathematica syntax. 
\item The Hilbert series output line ends with a bracket (farthest to the right). To begin a new input line, click on the line below this bracket (or move the cursor below the output line until it turns into a horizontal bar, then left-click). In the new input line, type the following code (except the gray text):

\noindent\rule{8cm}{0.4pt}
\begin{center}
\begin{Verbatim}[commandchars=\\\{\}]
\textcolor{gray}{#e1, e2, e3 are the generators of the Reeb cone, i.e. rows of W}
e1 := \{\};
e2 := \{\}; 
e3 := \{\};
\textcolor{gray}{#a,b,c are coordinates of the Reeb field.}
\textcolor{gray}{#copy the output of ToExpression into "HWI".}
F[x_, y_, z_, t_]:= 
FullSimplify["HWI"] /. \{Subscript[T, 0] -> Exp[-x*t], Subscript[T, 1] -> Exp[-y*t], 
Subscript[T, 2] -> Exp[-z*t]\}
Series[F[a,b,c,t], \{t,0,-2\}]
a0 := SeriesCoefficient[Series[F[a,b,c,t], \{t,0,-2\}], -3]
e := a*e1 + b*e2 + c*e3
Minimize[\{a0, e > 0 && c == 3\}, Element[\{a,b,c\},Reals]]
Element[%, Rationals]
\end{Verbatim}
\end{center}
\noindent\rule{8cm}{0.4pt}

\noindent Next, copy the Hilbert series output in Mathematica syntax into ``HWI'', then fill the braces following \verb|e1, e2, e3| with the $N_0$ coordinates of the corresponding Reeb cone generators (recall that $N_0$ is the cardinality of the Hilbert basis). These latter generators are precisely the rows of the weight matrix $W$ appearing in the second line of the Macaulay2 output. For example, when $N_0 = 8$ (as in Example \ref{ex:dp2reeb}), the syntax is as follows: 

\noindent\rule{8cm}{0.4pt}
\begin{Verbatim}
e1 := {-1,-1,-1,0,0,0,1,1};
e2 := {-1,0,1,-1,0,1,-1,0};
e3 := {1,1,1,1,1,1,1,1};
\end{Verbatim}
\noindent\rule{8cm}{0.4pt}
\item Press ``Shift$+$Enter'' to run the code. 
\end{itemize}

Figure \ref{figure_mathematica_cfo13} shows the Mathematica code and output for the toric diagram $\Pcal^{2}$ shown in Figure \ref{figure_cfo_family}. Here are two comments.
\begin{itemize}
    \item The semicolon at the end of a line tells Mathematica not to display that line as output. 
    \item 
    After running the code, Mathematica will display the index character expansion up to order \(-2\), with the volume function \( a_0 \) appearing as the coefficient at order \(-3\). This is followed by the minimized volume value and the coordinates of the Reeb minimizer. The final line instructs Mathematica to check whether the Reeb minimizer is rational, providing a binary output (``True'' or ``False'').
\end{itemize}

\begin{figure}[!htbp]
\begin{center}Macaulay2
\includegraphics[scale=0.34]{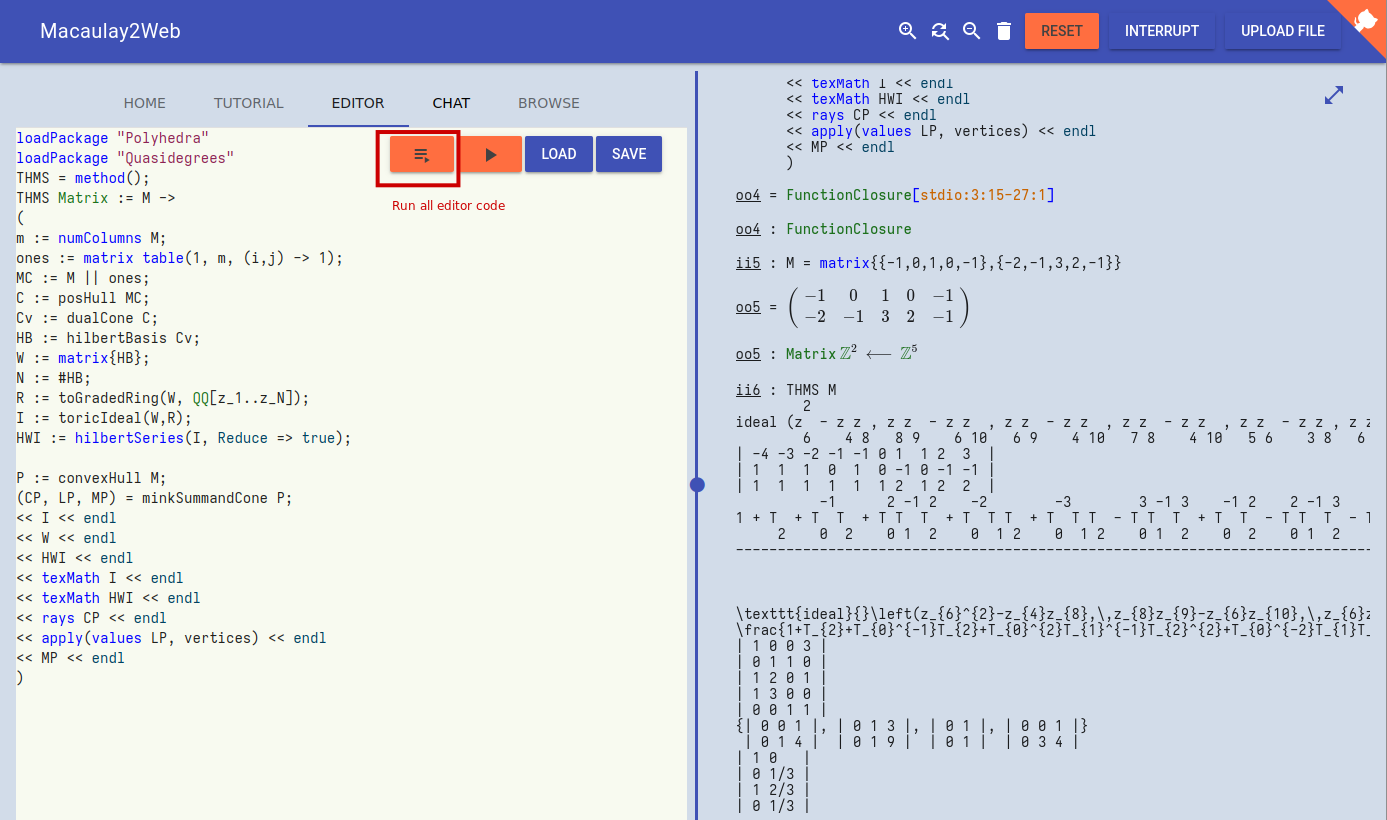}
\end{center}
\caption{Macaulay2 output for the toric diagram $\Pcal^{2}$ (cf.~Example \ref{example_cfo_13}), with oriented edges $(d_1, \dots, d_5)$ shown in Figure \ref{figure_cfo_family_decomposition}. Here, the Minkowski summand cone $C(\Pcal)$ is a three-dimensional cone in $\Rbb^5$ with minimal generators $u_1, u_2,u_3,u_4 = \Spvek{1;0;1;1;0}, \Spvek{0;1;2;3;0}, \Spvek{0;1;0;0;1}, \Spvek{3;0;1;0;1}$, respectively. The summand corresponding to $u_1$ is given by $\Pcal_1 = \Delta^{2} = \textnormal{Conv} \set{ \Spvek{0;0}, \Spvek{0;1}, \Spvek{1;4} }$, obtained by walking along $d_1, d_3, d_4$. The remaining summands $\Pcal_{2,3,4}$ are likewise obtained. The columns of the matrix $MP$ are $\Spvek{1;0;1;0}, \Spvek{0;1/3;2/3;1/3}$, corresponding to the lattice decomposition $\Pcal = \Pcal_1 + \Pcal_3 = \Delta^2 + \Lcal^1$ and a rational decomposition $\Pcal = (1/3) \Pcal_2 + (2/3) \Pcal_3 + (1/3) \Pcal_4$, respectively.}
\label{figure_macaulay2_cfo13}
\end{figure}

\subsection{Verification of the calculation in Proposition \ref{prop:irreg}}
\label{s:prop55code}

For completeness, we provide below the Mathematica code used to
verify the computer-algebra calculations in the proof of
Proposition \ref{prop:irreg}. Starting from the numerator and
denominator of the volume function, the code derives the
polynomials defining $F_s$ and $G_s$, eliminates $b$ by computing their
resultant, and extracts the polynomial $P_\sigma(a)$ appearing in
the proof. It also checks the polynomial identities used to exclude
the possibility that $a_s=0$ and verifies the reduction of
$P_\sigma$ modulo $5$ when $\sigma\equiv3\pmod5$.

\begin{lstlisting}
Module[
  {a, b, sigma, num, den, f, g, res, exceptional, p,
   coefficients, expectedP, common, aa, bb, remA, remB,
   p5, values, checks},

  num = (
    -a^2 - 6*b + 2*a*b - b^2 +
    9*(sigma^3 - 1) +
    3*a*(sigma^3 - sigma^2 + 2)
  );

  den = (
    (a + 3)*(a - b - 3)*
    (a - b + 3*(sigma - 1))*
    (sigma*a - b + 3*(sigma - 1))*
    ((sigma + 1)*a - b - 3)
  );

  (* Numerators of the two partial derivatives. *)
  f = Expand[D[num, a]*den - num*D[den, a]];
  g = Expand[D[num, b]*den - num*D[den, b]];

  (* Eliminate b. *)
  res = Resultant[f, g, b];

  exceptional = (
    -3^9*a^2*sigma^20*(3*sigma - a)^4*
    (a - 3)*(a + 3)^10*(sigma - 1)^6*
    ((sigma - 1)*a + 3*sigma)^4
  );

  p = Cancel[res/exceptional];

  (* Coefficients of the displayed P_sigma(a), from a^0 to a^8. *)
  coefficients = {
    11664*sigma^2*(sigma - 1)*(3*sigma - 1),

    1944*sigma*
      (-4 + 27*sigma - 34*sigma^2 - 26*sigma^3 +
       32*sigma^4),

    1296*
      (1 - 14*sigma + 24*sigma^2 + 79*sigma^3 -
       185*sigma^4 + 80*sigma^5),

    54*
      (35 - 50*sigma - 1433*sigma^2 + 4956*sigma^3 -
       4872*sigma^4 + 1280*sigma^5),

    9*
      (-53 + 2666*sigma - 14205*sigma^2 +
       24552*sigma^3 - 15088*sigma^4 + 2560*sigma^5),

    3*
      (-829 + 8842*sigma - 26685*sigma^2 +
       30096*sigma^3 - 12272*sigma^4 + 1280*sigma^5),

    (sigma - 1)*
      (1879 - 10367*sigma + 14528*sigma^2 -
       4784*sigma^3 + 256*sigma^4),

    (sigma - 1)^2*(-599 + 1448*sigma - 272*sigma^2),

    72*(sigma - 1)^3
  };

  expectedP = coefficients.Table[a^j, {j, 0, 8}];

  (* Specialize to a = 0 and recover A_sigma(b), B_sigma(b).
     Check remainders as well as quotients. *)
  common = (b + 3)*(3*(sigma - 1) - b);

  {aa, remA} =
    PolynomialQuotientRemainder[f /. a -> 0, common, b];

  {bb, remB} =
    PolynomialQuotientRemainder[g /. a -> 0, -6*common, b];

  (* sigma == 3 modulo 5. *)
  p5 = PolynomialMod[Expand[p /. sigma -> 3], 5];
  values = Table[Mod[p5 /. a -> j, 5], {j, 0, 4}];

  checks = <|
    "Full resultant factorization" ->
      (Expand[res - exceptional*expectedP] === 0),

    "Quotient is a polynomial" ->
      PolynomialQ[p, {a, sigma}],

    "Displayed P coefficients" ->
      (Expand[p - expectedP] === 0),

    "Degree is eight" ->
      (Exponent[p, a] === 8),

    "Leading coefficient" ->
      (Expand[
        Coefficient[p, a, 8] - 72*(sigma - 1)^3
      ] === 0),

    "F specialization has zero remainder" ->
      (Expand[remA] === 0),

    "G specialization has zero remainder" ->
      (Expand[remB] === 0),

    "A/B division identity" ->
      (Expand[
        aa - (
          (-b + 9*sigma)*bb -
          9*sigma^3*(-2*b + 3*sigma - 6)*
            (-b + 6*sigma - 6)
        )
      ] === 0),

    "B evaluated at (3 sigma - 6)/2" ->
      (Expand[
        (bb /. b -> (3*sigma - 6)/2) - 27*sigma^3/8
      ] === 0),

    "Polynomial modulo five" ->
      (PolynomialMod[
        p5 - (
          a^8 - 2*a^7 + a^6 + a^5 - a^4 +
          2*a^3 - 2*a^2 + 2*a + 1
        ),
        5
      ] === 0),

    "Residues are {1,3,3,1,2}" ->
      (values === {1, 3, 3, 1, 2})
  |>;

  <|
    "AllChecksPassed" -> (And @@ Values[checks]),
    "Checks" -> checks
  |>
]
\end{lstlisting}

\begin{figure}[H]
\begin{center}Mathematica
\includegraphics[scale=0.38]{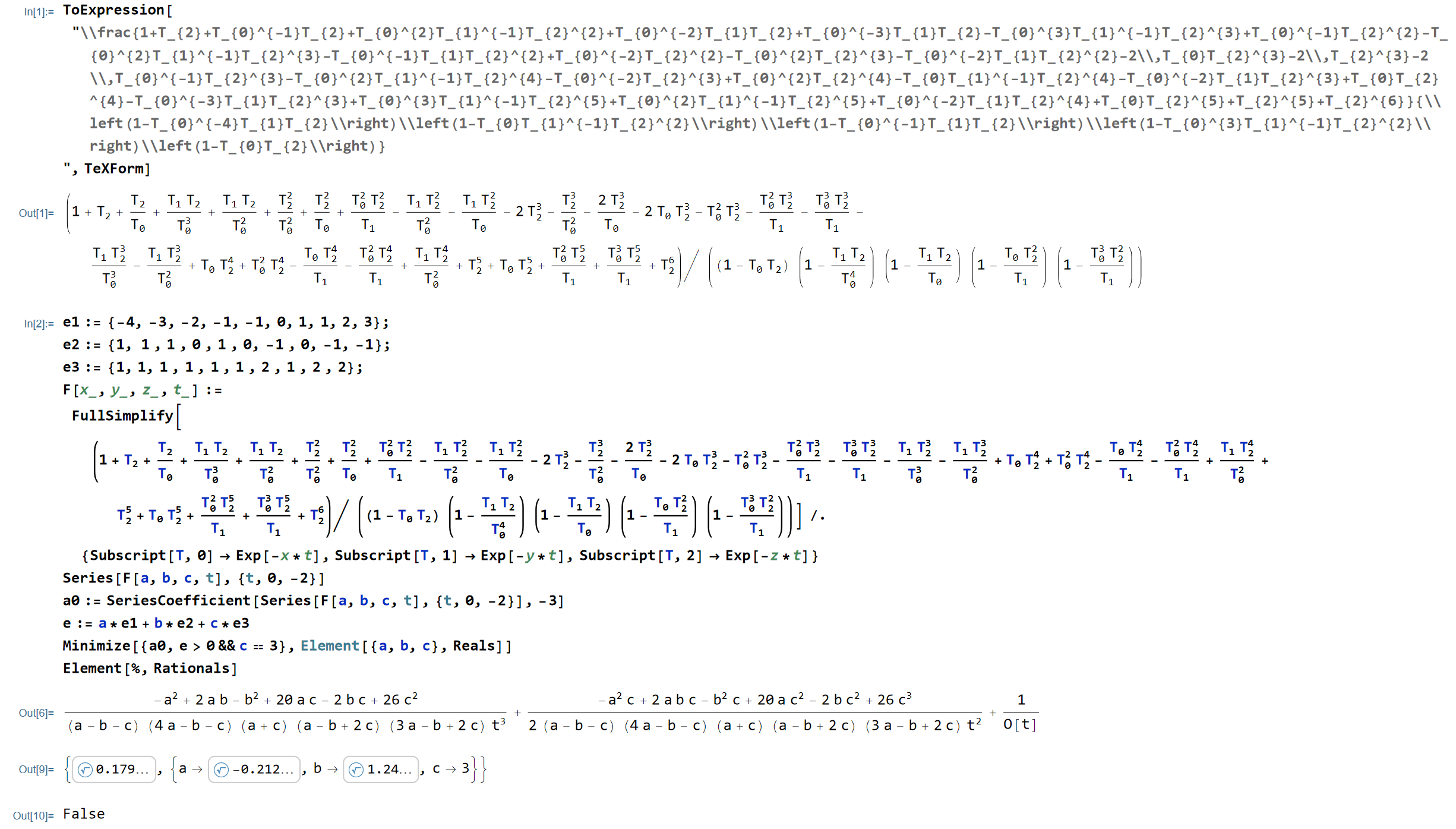}
\end{center}
\caption{Mathematica output for the toric diagram $\Pcal^{2}$ (cf.~Example \ref{example_cfo_13}).}
\label{figure_mathematica_cfo13}
\end{figure}

\section*{Declaration}
\subsection*{Conflict of interest} The authors declare that there are no conflicts of interest.
\subsection*{Data availability statement} The authors confirm that the data supporting the findings of this study are available within the article.

\newpage
\bibliography{biblio}
\bibliographystyle{amsplain}

\end{document}